\documentclass[11pt,letterpaper]{article}

\title{A nonparametric test for rough volatility\thanks{We would like to thank  Torben G.\ Andersen, Federico M.\ Bandi, Markus Bibinger, Tim Bollerslev, Dachuan Chen,  Rama Cont, Yi Ding, Masaaki Fukasawa, Jean Jacod, Z.\ Merrick Li, Nour Meddahi, Per Mykland, Yingying Li, Mark Podolskij, Eric Renault, Roberto Ren\`o (discussant), Mathieu Rosenbaum, Neil Shephard, Shuping Shi, George Tauchen, Jun Yu, Xinghua Zheng and participants at various conferences and seminars for their helpful comments and suggestions. An earlier version of the manuscript was circulated under the title ``The Fine Structure of Volatility Dynamics.''}}

\author{Carsten H.\ Chong\thanks{Department of Information Systems, Business Statistics and Operations Management,
		The Hong Kong University of Science and Technology, e-mail: carstenchong@ust.hk}	\and Viktor Todorov\thanks{Department of Finance, Northwestern University, e-mail: v-todorov@kellogg.northwestern.edu}}
\date{}
 
\usepackage{amsmath, amsfonts, amsthm, verbatim,
	amssymb,dsfont,color, accents}

\usepackage[utf8x]{inputenc}
\usepackage{mathtools}

\usepackage{tikz}
\usetikzlibrary{arrows.meta,arrows,decorations.pathreplacing}
\RequirePackage[colorlinks,citecolor=blue,urlcolor=blue]{hyperref}
\usepackage{tabularx,booktabs}
\usepackage{multirow}
\usepackage[round]{natbib}

\newcommand{\R}{\mathbb{R}}

\newcommand{\N}{\mathbb{N}}
\newcommand{\E}{\mathbb{E}}
\newcommand{\F}{\mathbb{F}}

\renewcommand{\P}{\mathbb{P}}

\newcommand{\Z}{\mathbb{Z}}

\newcommand{\Var}{\mathbf{Var}}
\newcommand{\Cov}{\mathbf{Cov}}
\newcommand{\bone}{\mathbf 1}

\theoremstyle{plain}
\newtheorem{theorem}{Theorem}[section]
\newtheorem{lemma}[theorem]{Lemma}
\newtheorem{proposition}[theorem]{Proposition}
\newtheorem{corollary}[theorem]{Corollary}
\newtheorem{condition}{Condition}
\newtheorem{Assumption}{Assumption}
\renewcommand{\theAssumption}{\Alph{Assumption}}
\makeatletter
\newcommand{\settheoremtag}[1]{
	\let\oldtheAssumption\theAssumption
	\renewcommand{\theAssumption}{#1}
	\g@addto@macro\endAssumption{
		\global\let\theAssumption\oldtheAssumption}
}
\makeatother
\theoremstyle{remark}
\newtheorem{definition}{Definition}
\newtheorem{remark}{Remark}
\newtheorem{Example}{Example}

\newcommand{\calc}{{\cal C}}
\newcommand{\calh}{{\cal H}}

\newcommand{\calu}{{\cal U}}
\newcommand{\calf}{{\cal F}}

\newcommand{\call}{{\cal L}}

\newcommand{\calg}{{\cal G}}

\newcommand{\calj}{{\cal J}}

\newcommand{\cf}{{\frak c}}
\newcommand{\Lf}{{\frak L}}
\newcommand{\dsone}{\mathds{1}}

\newcommand{\al}{{\alpha}}
\newcommand{\la}{{\lambda}}
\newcommand{\La}{{\Lambda}}
\newcommand{\eps}{{\epsilon}}
\newcommand{\vareps}{{\varepsilon}}
\newcommand{\ga}{{\gamma}}
\newcommand{\Ga}{{\Gamma}}
\newcommand{\vp}{{\varphi}}
\newcommand{\si}{{\sigma}}

\newcommand{\om}{{\omega}}
\newcommand{\Om}{{\Omega}}

\newcommand{\ov}{\overline}
\newcommand{\un}{\underline}
\newcommand{\wh}{\widehat}
\newcommand{\wt}{\widetilde}

\newcommand{\Den}{\Delta_n}

\newcommand{\limst}{\stackrel{\mathcal{L}-s}{\longrightarrow}}

\newcommand{\bthm}{\begin{theorem}}
	\newcommand{\ethm}{\end{theorem}}

\newcommand{\bcor}{\begin{corollary}}
	\newcommand{\ecor}{\end{corollary}}

\newcommand{\blem}{\begin{lemma}}
	\newcommand{\elem}{\end{lemma}}

\newcommand{\bprop}{\begin{proposition}}
	\newcommand{\eprop}{\end{proposition}}

\newcommand{\bcond}{\begin{condition}}
	\newcommand{\econd}{\end{condition}}

\newcommand{\bdf}{\begin{definition}}
	\newcommand{\edf}{\end{definition}}

\newcommand{\bex}{\begin{example}}
	\newcommand{\eex}{\end{example}}

\newcommand{\brem}{\begin{remark}}
	\newcommand{\erem}{\end{remark}}

\newcommand{\bpr}{\begin{proof}}
	\newcommand{\epr}{\end{proof}}

\newcommand{\benu}{\begin{enumerate}}
	\newcommand{\eenu}{\end{enumerate}}

\newcommand{\beq}{\begin{equation}}
	\newcommand{\eeq}{\end{equation}}

\newcommand{\bit}{\begin{itemize}}
	\newcommand{\eit}{\end{itemize}}

\newcommand{\bass}{\begin{Assumption}}
	\newcommand{\eass}{\end{Assumption}}

\numberwithin{equation}{section}

\DeclareMathOperator\Log{Log}

\allowdisplaybreaks[4]

\begin{document}

	\maketitle

\begin{abstract}
	We develop a nonparametric test for deciding whether volatility of an asset follows a standard semimartingale process, with paths of finite quadratic variation, or a rough process with paths of infinite quadratic variation. The test utilizes the fact that volatility is rough if and only if volatility increments are negatively autocorrelated at high frequencies. It is based on the sample autocovariance of increments of spot volatility estimates computed from high-frequency asset return data. By showing a feasible CLT for this statistic under the null hypothesis of semimartingale volatility paths, we construct a test with fixed asymptotic size and an asymptotic power equal to one. The test is derived under very general conditions for the data-generating process. In particular, it is robust to jumps with arbitrary activity and
	to the presence of market microstructure noise. In an application of the test to SPY high-frequency data, we find evidence for rough volatility. 
\end{abstract}
	
\bigskip

	
	
\bigskip
	

\section{Introduction}

Most economic and financial time series exhibit time-varying volatility. The standard way of modeling volatility in continuous time is via stochastic integrals driven by Brownian motions and/or L\'evy jumps, see e.g., the review articles of \citet{GHYSELS1996119} and \citet{shephard2009stochastic}. This way of modeling volatility implies that,  while volatility paths can exhibit discontinuities, they nevertheless remain smooth in squared mean 
and have finite quadratic variations in particular. An alternative way of modeling volatility, which has gained significant popularity recently, is via stochastic integrals driven by a fractional Brownian motion, see \citet{comte1996long,comte1998long} and   the more recent work of \citet{gatheral2018volatility}. In this case,   volatility paths can be very rough, with a lot of oscillations at short time scales leading to an explosive quadratic variation. 

Naturally, the degree of roughness of volatility (controlled by the Hurst parameter of the driving fractional Brownian motion) determines the optimal rate of convergence of nonparametric spot estimators of volatility, with the latter dropping to zero as the degree of roughness increases. In fact, many of the tools developed for the analysis of high-frequency data,  see e.g., \citet{JP12}, depend critically on volatility being a semimartingale process. The goal of this paper, therefore, is to develop a general nonparametric test  to decide whether  volatility is a standard semimartingale with a finite quadratic variation or is a rough process with infinite quadratic variation.   

If volatility were directly observable, it would be relatively easy to design such a test based on variance ratios. 
For example, if one were to observe volatility at high frequency, then one way of testing for roughness of volatility would be by assessing the scaling of the quadratic variation of the discretized volatility process computed at different frequencies, see e.g., \citet{barndorff2011multipower}. 
Volatility is of course not directly observable, and testing for rough volatility is significantly more challenging when using high-frequency observations of the underlying process only. This is mainly for three reasons. 

First, spot volatility estimates contain  nontrivial estimation errors. As noted by \citet{cont2022rough}, this error can make the volatility estimates appear rough even if the true volatility process is not. Therefore, rough volatility, that is, the roughness of unobserved spot volatility, has to be distinguished from the roughness of volatility estimates such as realized variance (as studied by e.g., \cite{WANG2021}). Second,
rough volatility will result in larger increments of spot volatility estimates, but bigger in size increments can be also due to jumps in volatility,   which are well documented in practice; see e.g., \cite{JT10}, \cite{TT11} and \cite{BR16}. Third, infinite activity jumps in the price process, see e.g., \cite{AJ11}, and microstructure noise in the price observations, see e.g., \cite{Hansen06}, further make it difficult to estimate volatility in the first place.   
In fact, as shown in \citet{JR14}, the optimal nonparametric rate for estimating volatility from high-frequency data in the presence of jumps depends on the degree of jump activity, with the rate becoming significantly worse (and approaching zero) for higher degree of jump activity. At the same time, the presence of market microstructure noise in observed prices further slows down the rate at which volatility can be estimated, making it difficult to evaluate its behavior over small time scales. For example, the optimal rate of convergence for estimating integrated volatility from $n$ noisy observations is $n^{1/4}$ compared to $n^{1/2}$ in the noise-free setting (\citet{reiss2011asymptotic}).  While several recent works have proposed solutions to the latency of volatility when estimating volatility roughness (see e.g.,   \citet{Bennedsen22}, \citet{FTW22}, \citet{bolko2020GMM}, \citet{chong2022statistical, chong2022statistical-b} and \cite{S24}), a robust statistical  theory for rough volatility that takes into account price jumps, volatility jumps and microstructure noise on top of estimation errors has been notably absent.

In this paper we show that, in spite of the poor rate of estimating volatility in the presence of jumps with high jump activity and market microstructure noise, one can nevertheless test for volatility roughness, with the properties  of the test unaffected by the degree of jump activity and the presence of microstructure noise. We achieve this by relying on a second---equivalent---characterization of rough volatility in terms of the autocorrelation of its increments. Mainly, volatility is rough if and only if changes in volatility are negatively correlated at high frequency (Theorem~\ref{thm:equiv}). By contrast, high-frequency volatility increments in  semimartingale volatility models are asymptotically uncorrelated, as they are locally dominated by the martingale component of volatility. Note that rough volatility  concerns the negative correlation of volatility moves on short time scales only and has no implication for the long range behavior of volatility. Therefore, rough volatility models are compatible with the well-documented long memory of volatility (\citet{ABDL03}), see Remark~\ref{rem:longmem}.

More specifically, we propose a test based on the sample autocovariance of increments of spot volatility estimates. The spot volatility estimates are constructed from the empirical characteristic function of price increments in local blocks of high-frequency data. By choosing values of the characteristic exponent away from zero, we mitigate the impact of finite variation jumps on the spot volatility estimators. While the latter still contain non-negligible biases due to the infinite variation jumps and microstructure noise, these biases  have increments that are asymptotically uncorrelated across different blocks. 

Therefore, even in the presence of jumps, noise and estimation errors, the first-order autocovariance of increments of spot volatility estimates---computed over non-overlapping blocks---should be zero asymptotically if volatility follows a standard semimartingale process. On the other hand, after appropriately scaling down the sample autocovariance, it converges to a strictly negative number when volatility is rough. The resulting test statistic is a self-normalized quantity which converges to a standard normal random variable under the null hypothesis of semimartingale volatility and  diverges to negative infinity under the rough volatility alternative hypothesis (Theorems~\ref{thm:H0} and \ref{thm:test-noise}).

The rest of the paper is organized as follows. We start with introducing our setup in Section~\ref{sec:setting}. The theoretical development of the test is given in Section~\ref{sec:test} and its finite sample properties are evaluated in Section~\ref{sec:mc}. Section~\ref{sec:emp} contains our empirical application. Section~\ref{sec:concl} concludes. The proofs are given in Appendix~\ref{sec:proof} with technical details deferred to Appendix~\ref{sec:proof-app} in the supplement (\citet{supp}).

\section{Model Setup}\label{sec:setting}

We denote the logarithmic asset price by $x$. We assume that $x$ is defined on a filtered probability space $(\Om,\calf,\F=(\calf_t)_{t\geq0},\P)$, with the following It\^o semimartingale dynamics:
\begin{equation}\label{eq:x0}  \begin{split}
		x_t&=x_0+\int_0^t\al_s ds +\int_0^t \si_s dW_s +\int_0^t\int_E \ga(s,z)(\mu-\nu)(ds,dz) \\
		&\quad+ \int_0^t\int_E \Ga(s,z) \mu(ds,dz),
	\end{split}
\end{equation}
where $x_0$ is an $\calf_0$-measurable random variable, $W$ is a standard $\F$-Brownian motion, $\mu$ is an $\F$-Poisson random measure on $[0,\infty)\times E$ with intensity measure $\nu(ds,dz)= \la(dz)ds$, and $\la$ is a $\si$-finite measure on an auxiliary space $E$. The drift $\al$ is a locally bounded predictable process, while  $\ga$ and $\Ga$ are predictable functions such that the integrals in \eqref{eq:x0} are well defined. We think of $\ga$ and $\Ga$  as modeling infinite variation and finite variation jumps, respectively, rather than modeling small and big jumps. This distinction is important because our assumptions on the two below will differ. In particular, if $x$ has only finite variation jumps, we can and should take $\ga\equiv0$. 

The goal of this paper is to develop a statistical test for the fine structure of the spot (diffusive) variance, $c_t=\si^2_t$. More precisely, our goal is to develop a nonparametric test to decide whether the realized path $c_t(\om)$ is the path of a semimartingale process or whether it is rough in the following sense:
\begin{definition}\label{def}
	We say that a   function $Z=(Z_t)_{t\in[0,T]}$    is \emph{rough} if 
	\begin{equation}\label{eq:ratio} 
		\frac{\mathit{RV}(Z)^{n,2}_T}{\mathit{RV}(Z)^n_T}= \frac{  \frac12\sum_{i=0}^n \bigl(Z_{\frac{i+1}{n}T}-Z_{\frac{i-1}{n}T}\bigr)^2}{ \sum_{i=1}^n \bigl(Z_{\frac in T}-Z_{\frac{i-1}{n} T}\bigr)^2} \stackrel{p}{\longrightarrow} \tau \quad \text{as } n\to\infty,
	\end{equation}
	for some  $\tau<1$. We say that a sequence $(a_n)_{n\in\N}$ converges \emph{subsequentially} to a limit $a$, denoted by $a_n\stackrel{p}{\longrightarrow} a$, if for every subsequence $(n_k)_{k\in\N}$  there is another subsequence $(n_{k_\ell})_{\ell\in\N}$ such that $a_{n_{k_\ell}}\to a$ as $\ell\to\infty$. Moreover, in \eqref{eq:ratio}, we use the convention $0/0=2$ and define $Z_t=Z_0$ for $t<0$ and $Z_t=Z_T$ for $t>T$.
\end{definition}
To paraphrase, a function $Z=(Z_t)_{t\in[0,T]}$ is rough if
the realized variance of $Z$ computed with twice the original step size, $\mathit{RV}(Z)^{n,2}_T$, is only a fraction of the realized variance of $Z$ computed with the original step size, $\mathit{RV}(Z)^{n}_T$. This notion of roughness is \emph{model-free}, with a smaller value of $\tau$ indicating less regular sample paths of $Z$. The reason we use a subsequential criterion in \eqref{eq:ratio}, instead ordinary convergence to $\tau$, is to make this definition more suitable for examining the roughness of a path of a \emph{stochastic process} $Z$. Indeed, in this case, proving convergence in probability in \eqref{eq:ratio} is sufficient to deduce the almost sure roughness of the sample paths of $Z$. With ordinary convergence in \eqref{eq:ratio}, one would have to show almost sure convergence of the realized variance of $Z$, which is usually quite involved or even unknown (e.g., if $Z$ is a general continuous martingale). 

It is easy to see that the limit $\tau$ in \eqref{eq:ratio}, if it exists, must satisfy $\tau\in[0,2]$ in all cases, and that we have $\tau=2$  if $Z$ is continuously differentiable. If $Z$ has a finite and non-zero quadratic variation (e.g., if $Z$ is the path of a semimartingale with a non-zero local martingale part), we have $\tau=1$. If $Z$ is the path of a pure-jump process, then $\tau=2$ if there is no jump on $[0,T]$ (in which case, $Z$ is a constant); otherwise,   there is at least one jump on $[0,T]$ and we have $\tau=1$. Therefore, the paths of a pure-jump process are almost surely not rough according to our definition. If $Z$ is the path of a fractional Brownian motion with Hurst parameter $H\in(0,1)$, then we have $\tau=2^{2H-1}$ and   $Z$ is rough precisely when $H\in(0,\frac12)$. 

With this definition, we can now introduce  the null and alternative hypotheses we wish to test:
\begin{equation}\label{eq:hypo}
	\begin{split} 
		H_0:	&~\om\in\Om_0= \{\om:\text{$(c_t(\om))_{t\in[0,T]}$ is the path of a semimartingale process}\},  \\
		H_1:	&~ \om\in\Om_1=\{\om:\text{$(c_t(\om))_{t\in[0,T]}$ is rough in the sense of Definition~\ref{def}}\},
	\end{split}
\end{equation}
where $T$ is some fixed number. 

\begin{remark}\label{rem:alt}
	There are other notions of roughness in the literature; see e.g., \cite{cont2022rough} and \cite{HS24}. These definitions are based on $p$-variations for different values of $p$ and produce the same  ranking as Definition~\ref{def} when used to compare the roughness of  continuous processes (for discontinuous processes, there are some differences). The main reason we use Definition~\ref{def} in this paper is because, as we show in Theorems~\ref{thm:equiv}--\ref{thm:test-noise} below, Definition~\ref{def} admits a statistical implementation that is robust to estimation errors, jumps and microstructure noise. 
\end{remark}

\begin{remark}\label{rem:longmem}
	Roughness in the sense of Definition~\ref{def} is a local property and is independent of the behavior of the underlying process as $T\to\infty$. In particular, if $(Z_t)_{t\geq0}$ is a stationary process, roughness is unrelated to the short- or long-memory properties of $Z$, which are usually defined in terms of the decay of its autocorrelation function $\rho_T = \Cov(Z_0,Z_T)$ as $T\to\infty$. As we show in Theorem~\ref{thm:equiv} below, roughness in fact is related to the autocorrelation  of the \emph{increments} of $Z$ as the sampling frequency increases but with the length of the time interval kept fixed. Of course, in   parametric models, it can happen that roughness and long/short memory properties are parametrized by a single parameter. This is, for instance, the case with models based on  fractional Brownian motion, where depending on whether the Hurst parameter $H$ is smaller or bigger than $\frac12$, one has either roughness and short memory or no roughness and long memory. The need to separate   roughness on the one hand and short versus long memory on the other hand has been recognized in previous work already; see \citet{Bennedsen22} and \citet{LSY20}.
\end{remark}

For our theoretical analysis, we need to impose some mild structural assumptions on $c_t$ (and other coefficients in \eqref{eq:x0}) under both $H_0$ and $H_1$, as detailed in the following two subsections. We note that $H_0$ and $H_1$ do not exhaust the whole  model space; e.g., $c_t$ could have a finite (possibly zero) quadratic variation without being a semimartingale process. For such a model for $c_t$, we have $\tau>1$ in \eqref{eq:ratio} and in that sense such a specification implies smoother volatility paths than our null hypothesis. As we explain in Remark~\ref{rem:smooth} below, such very smooth volatility models are not going to be rejected by our test.

Also, to keep the exposition simple, in our detailed formulation of the null and alternative hypotheses, we assume that $c_t$ is either a semimartingale process (with paths that are almost surely not rough) or a rough process (with paths that are rough almost surely). These hypotheses, and the subsequent results, can be easily extended to the case where $c_t(\om)$ is a semimartingale path on one subset of $\Om$ and rough on another.

\subsection{The Null Hypothesis}
Under $H_0$, we further assume that  $c_t$ is an It\^o semimartingale process given by
\begin{equation}\label{eq:c}\begin{split}
		c_t	&=c_0+\int_0^t\al^c_s ds + \int_0^t (\si^c_s dW_s +   \ov\si^c_s d\ov W_s) + \int_0^t\int_E \ga^c(s,z)(\mu-\nu)(ds,dz)\\&\quad + \int_0^t\int_E \Ga^c(s,z)\mu(ds,dz),
	\end{split}
\end{equation}
where $\ov W$ is a standard $\F$-Brownian motion that is independent of $W$, $c_0$ is $\calf_0$-measurable and the requirements for the coefficients of $c$ will be given later. All continuous-time stochastic volatility models that are solutions to stochastic differential equations are of this form, and \eqref{eq:c} is a frequent assumption in the financial econometrics literature, see e.g., Assumption (K-$r$) in \citet{AJ14} and \citet{JP12}. 
In particular,  $c_t$ can have jumps and is nowhere differentiable in the presence of a diffusive component. Nevertheless, volatility is smooth in squared mean in the following sense: There exists a sequence of stopping times $\tau_n$ increasing to infinity such that for all $t\geq0$,
\begin{equation}\label{eq:smooth_H0}
	\delta\mapsto \E[(c_{(t+\delta)\wedge \tau_n} - c_{t\wedge \tau_n})^2]~\textrm{is differentiable (including at $\delta=0$).}
\end{equation}

In addition, we need similar  structural assumptions on the infinite variation jumps of $x$. More precisely, we assume that  $\vp(u)_t= \int_E (e^{iu\ga(t,z)}-1-iu\ga(t,z))\la(dz) $, for every $u\in\R$, is a complex-valued It\^o semimartingale of the form
\begin{align}\nonumber
	\vp(u)_t&=\vp(u)_0+\int_0^t\al^\vp(u)_s ds + \int_0^t (\si^\vp(u)_s dW^\vp_s +   \ov\si^\vp(u)_s d\ov W^\vp_s) \\
	&\quad+ \int_0^t\int_E \ga^\vp(u;s,z)(\mu-\nu)(ds,dz) + \int_0^t\int_E \Ga^\vp(u;s,z)\mu(ds,dz),
	\label{eq:vp} \end{align}
where $W^\vp$ and $\ov W^\vp$ are independent standard $\F$-Brownian motions (jointly Gaussian with and possibly dependent on  $W$ and $\ov W$) and the coefficients of $\vp(u)$ may be complex-valued.

If $\ga$ is a deterministic function, $\vp(u)_t$ is simply the spot log-characteristic function of the infinite variation jump part of $x$. Condition \eqref{eq:vp} is a rather mild condition and is satisfied, for example, if $x$ has the same infinite variation jumps as $\int_0^t K_{s-} dL_s$, where $K$ is an It\^o semimartingale and $L$ is a time-changed L\'evy process with the time change being also an It\^o semimartingale, see Example~\ref{ex:1} below. This situation covers the vast majority of parametric jump models considered in the literature. A condition like \eqref{eq:vp} is needed in order to safeguard our test against the worst-case  scenario (which is possible in theory but perhaps less relevant in practice) where price jumps are of infinite variation  \textit{and} their intensity is much rougher than  diffusive volatility. The null hypothesis does cover situations where $x$ has infinite variation jumps with non-rough intensity and/or finite variation jumps with arbitrary degree of roughness.

Our assumption for the process $x$ under the null hypothesis is given by:

\settheoremtag{H$_0$}
\begin{Assumption}\label{ass:H0} We assume   the following conditions   under the null hypothesis:
	\begin{enumerate}
		\item[(i)] We have \eqref{eq:x0}--\eqref{eq:vp}, where $\al$, $\ga$, $\Ga$, $\al^c$, $\si^c$, $\ov \si^c$, $\ga^c$, $\Ga^c$ and $\al^\vp(u)$, $\si^\vp(u)$, $\ov\si^\vp(u)$, $\ga^\vp(u;\cdot)$ and $\Ga^\vp(u;\cdot)$ (for every $u\in\R$) are predictable and $\al^c$, $\si^c$ and $\ov\si^c$ 
		are locally bounded, 
		Moreover, $\inf\{c_s: 0\leq s\leq t\}>0$ for all $t>0$ almost surely. 
		\item[(ii)] There exist a sequence of stopping times $\tau_n$ increasing to infinity almost surely and   nonnegative measurable functions $J_n(z)$ satisfying $\int_E J_n(z)\la(dz)<\infty$ for all $n\in\N$ such that whenever   $t\leq \tau_n$, we have
		\begin{equation}\label{eq:prop1} 
			(	\lvert \ga(t,z)\rvert^2+\lvert \ga^c(t,z)\rvert^2+\lvert \Ga(t,z)\rvert+\lvert \Ga^c(t,z)\rvert)\wedge 1\leq J_n(z).
		\end{equation}
		\item[(iii)] 
		For any compact subset $\calu\subseteq\R$,	the process
		\begin{equation}\label{eq:prop2} 
			t\mapsto \sup_{\delta\in(0,1)} \sup_{u\in\calu} \bigl\{	\delta\lvert\al^\vp(u/\sqrt{\delta})_t\rvert\}
		\end{equation}
		is locally bounded, while
		the processes
		\begin{align}\label{eq:prop0-1} 
			&	t\mapsto\sup_{s\in[0, \delta]} \E[\lvert \al_{t+s}-\al_t\rvert\wedge1\mid\calf_t],\qquad t\mapsto\sup_{s\in[0, \delta]} \E[\lvert \si^c_{t+s}-\si^c_t\rvert\wedge1\mid\calf_t],\\
			&		t\mapsto  \sup_{u\in\calu} \bigl\{ \lvert\delta \si^\vp(u/\sqrt{\delta})_t\rvert+ \lvert\delta \ov\si^\vp(u/\sqrt{\delta})_t\rvert \bigr\} \label{eq:prop3} 
		\end{align}
		converge uniformly on compacts in probability to $0$
		as $\delta\to0$. Finally, for any $n\in\N$, there are constants $C^n(\delta)\in(0,\infty)$ such that $C^n(\delta)\to0$ as $\delta\to0$ and
		\begin{align}\label{eq:prop4-2} 
			\sup_{u\in\calu} \bigl\{	\lvert \delta\ga^\vp(u/\sqrt{\delta}; t,z)\rvert^2 + \lvert\delta\Ga^\vp(u/\sqrt{\delta}; t,z)\rvert\bigr\}\wedge 1 \leq C^n(\delta) J_n(z)
		\end{align}
		for all $t\leq \tau_n$ and  $z\in E$.
	\end{enumerate}
\end{Assumption}
The conditions on the coefficients of $x$, $c$ and $\vp(u)_t$ are only slightly stronger than requiring them to be It\^o semimartingales. In particular, all three processes are allowed to have jumps of arbitrary activity. The following example shows that the assumptions on $\vp(u)_t$ are mild indeed.
\begin{Example}\label{ex:1}
	Suppose that $\int_0^t\int_E \ga(s,z)(\mu-\nu)(ds,dz)=\int_0^t K_{s-} dL_s$, where $L$ is a mean-zero purely discontinuous martingale whose jump measure has $\F$-compensator $\la_{t-} dt F(dz)$, for some Lévy measure $F$ on $\R$. Further, suppose that $K$ and $\la$ are It\^o semimartingales. Then 
	\begin{equation}\label{eq:vp-ex} 
		\vp(u)_t=\int_\R (e^{iu K_{t}z}-1-iuK_{t}z)\la_t F(dz),
	\end{equation}
	and \eqref{eq:prop2}, \eqref{eq:prop3} and \eqref{eq:prop4-2} are satisfied. The proof is given in Lemma~\ref{lem:ex1}.
\end{Example}

\subsection{The Alternative Hypothesis}

Under the alternative hypothesis, the stochastic volatility process  $c$  ceases to be a semimartingale and is given by a rough process with a low degree of regularity. More precisely, we assume under the alternative hypothesis that
\begin{equation}\label{eq:rough} 
	c_t=f(v_t),\quad\text{where}\quad	v_t=v_0+\int_0^t g(t-s)(\si^v_sdW_s+\ov\si^v_sd\ov W_s)+\wt v_t,
\end{equation} 
$v_0$ is $\calf_0$-measurable and $f$ is a  $C^2$-function. The kernel $g$ has the semiparametric form 
\begin{equation}\label{eq:g} 
	g(t)=K_H^{-1} t_+^{H-1/2}+g_0(t),\qquad K_H = \Ga(H+1/2)/\sqrt{\sin(\pi H)\Ga(2H+1)},
\end{equation} 
where $H\in(0,\frac12)$, $x^a_+=x^a$ if $x>0$ and $x^a_+=0$ otherwise, and $g_0\in C^1([0,\infty))$ with $g_0(0)=0$. The normalization constant $K_H$ is chosen such that in the case $\si^v=1$, $\ov\si^v=0$ and $\wt v=0$, we have $\E[(v_{t+\Den}-v_t)^2]/\Den^{2H}\to1$ as $\Den\to0$ for any $t\in(0,\infty)$. The specification in \eqref{eq:rough} contains the rough volatility models considered in \citet{gatheral2018volatility}, \citet{Bennedsen22} and \citet{WANG2021} as special cases.

Since $H<\frac12$ and $g$ has the same behavior as the fractional kernel $K_H^{-1}t^{H-1/2}$ at $t=0$, the process $v$ has the same small time scale behavior, and thus the same degree of roughness, as a fractional Brownian motion $B^H$ with Hurst parameter $H$. However, due to the presence of $g_0$, there are no restrictions on the asymptotic behavior of $g$ and hence of $v$ as $t\to\infty$. In particular, $v$ can have  short-memory behavior or long-memory behavior. As explained in Remark~\ref{rem:longmem}, the separation of roughness  and   memory properties of $v$ has been found important in practice; see \citet{Bennedsen22} and \citet{LSY20}. In addition to a fractional component, $v$ may also have a regular part $\wt v$, which can include a (possibly discontinuous) semimartingale component.

We can compare the smoothness of $c_t$ under the alternative hypothesis with that under the null hypothesis. We have the following result for $c_t$ in (\ref{eq:rough}): There exists a sequence of stopping times $\tau_n$ increasing to infinity such that for all $t\geq0$, 
\begin{equation}\label{eq:smooth_H1}
	\E[(c_{(t+\delta)\wedge \tau_n} - c_{t\wedge \tau_n})^2] \asymp \delta^{2H}~\textrm{as $\delta\to0$. }
\end{equation}
In particular, \eqref{eq:smooth_H1} is not differentiable at $\delta=0$. This leads to paths of $c$ of infinite quadratic variation. The parameter $H$ governs the roughness of the volatility path, with lower values corresponding to  rougher volatility dynamics. The limiting case $H=1/2$ corresponds to the standard semimartingale volatility model under the null hypothesis. Figure~\ref{fig:vol_path} below visualizes the difference between a rough volatility path and a semimartingale volatility path.

Our assumption for the process $x$ under the alternative hypothesis is given by:

\settheoremtag{H$_1$}
\begin{Assumption}\label{ass:H1} We have the following setup  under the alternative hypothesis:
	\begin{enumerate}
		\item[(i)]  We have \eqref{eq:x0}  and \eqref{eq:rough}, where $f$ and $g$ are as specified above and $\wt v$, $\al$, $\ga$, $\Ga$, $\si^v$ and $\ov \si^v$ are predictable processes. Moreover,  $\al$, $\wt v$, $\si^v$ and $\ov\si^v$ are locally bounded and $\inf\{(f'(v_s))^2[(\si^v_s)^2+(\ov \si^v_s)^2]: 0\leq s\leq t\}>0$ for all $t>0$ almost surely. 
		\item[(ii)] 	  There exist  stopping times $\tau_n$ increasing to infinity almost surely and   nonnegative measurable functions $J_n(z)$ satisfying $\int_E J_n(z)\la(dz)<\infty$ for all $n\in\N$ such that whenever   $t\leq \tau_n$, we have
		\begin{equation}\label{eq:prop1-alt} 
			(	\lvert \ga(t,z)\rvert^2+\lvert \Ga(t,z)\rvert)\wedge 1\leq J_n(z).
		\end{equation}
		\item[(iii)]  For any compact subset $\calu$ of $\R$, the process 
		\begin{equation}\label{eq:prop3-alt} 
			t\mapsto \sup_{\delta\in(0,1)} \sup_{u\in\calu} \bigl\{	\delta\lvert\vp(u/\sqrt{\delta})_t\rvert\} 
		\end{equation}
		is locally bounded. Moreover, 
		for any $n\in\N$, there is   $C_n\in(0,\infty)$ such that
		\begin{align}\label{eq:prop2-alt} 
			&	\sup_{t\geq0}	\E[\lvert\wt v_{t+\delta}-\wt v_t\rvert^2\bone_{\{t+\delta\leq \tau_n\}}]^{1/2} \leq C_n \delta^H h(\delta),\\
			\label{eq:prop4-alt} 
			&	\sup_{t\geq0} \sup_{u\in\calu}\E\Bigl[ \lvert \delta\vp(u/\sqrt{\delta})_{t+\delta}-\delta\vp(u/\sqrt{\delta})_t\rvert^2\bone_{\{t+\delta\leq \tau_n\}} \Bigr]^{1/2} \leq C_n \delta^Hh(\delta),\\
			&\sup_{t\geq0}\E[ (\lvert\si^v_{t+\delta}-\si^v_t\rvert^2+\lvert\ov\si^v_{t+\delta}-\ov\si^v_t\rvert^2)\bone_{\{t+\delta\leq \tau_n\}}  ]^{1/2} \leq C_nh(\delta)\label{eq:prop5-alt}
		\end{align}
		for all $\delta\in(0,1)$ and some $h$ satisfying $h(t)\to0$ as $t\to0$.
	\end{enumerate}
\end{Assumption}

Due to Condition~\eqref{eq:prop2-alt}, $\wt v$ can be any predictable process that is marginally smoother than $v$. In particular, $\wt v$ can be another fractional process with $H'>H$ or an It\^o semimartingale (possibly with jumps) or a combination thereof.  
\begin{Example}\label{ex:2}
	Consider the same setting  as in Example~\ref{ex:1}, except that $K$ and $\la$ can be any predictable locally bounded process satisfying
	\begin{equation}\label{eq:cond} 
		\sup_{s,t\geq0 }	\E[\lvert K_t- K_s\rvert^2\bone_{\{s,t\leq \tau_n\}}]^{1/2}+\sup_{s,t\geq0 }	\E[\lvert \la_t- \la_s\rvert^2\bone_{\{s,t\leq \tau_n\}}]^{1/2} \leq C_n \lvert t-s\rvert^H,
	\end{equation}
	for all $n\in\N$ and where $\tau_n$ and $C_n$ are as in Assumption~\ref{ass:H1}. Then, we still have \eqref{eq:vp-ex} and both \eqref{eq:prop3-alt} and \eqref{eq:prop4-alt} are satisfied. This is shown in Lemma~\ref{lem:ex2}. In particular, this example covers the case where $\la_t= \ell(c_t)$ for some $C^1$-function $\ell$.
\end{Example}

\section{Testing for Rough Volatility}\label{sec:test}

If volatility were directly observable, testing for rough volatility would be a straightforward matter based on the variance ratio in \eqref{eq:ratio}; see \citet{barndorff2011multipower}. Due to the latency of volatility, however,    examining the realized variance of spot volatility estimates at different frequencies is problematic as a testing strategy, as estimation errors can produce spurious roughness and jumps in price and volatility as well as microstructure noise in the data can lead to significant biases. To avoid such complications, we base our test  on an equivalent---but more robust---characterization of rough functions. 

To this end, we expand the square in the numerator of \eqref{eq:ratio} and obtain
\begin{align*} 
	\frac{\mathit{RV}(Z)^{n,2}_T}{\mathit{RV}(Z)^n_T} &=\frac{ \frac12 \sum_{i=0}^{n} \bigl(Z_{ \frac{i+1}n T}-Z_{\frac in T}\bigr)^2+\frac12\sum_{i=0}^n  \bigl(Z_{\frac in T}-Z_{\frac {i-1}n T}\bigr)^2}{ \sum_{i=1}^n \bigl(Z_{\frac in T}-Z_{\frac {i-1}n T}\bigr)^2}\\
	&\quad+\frac{ \sum_{i=0}^n \bigl(Z_{ \frac{i+1}n T}-Z_{\frac in T}\bigr)\bigl(Z_{\frac in T}-Z_{\frac {i-1}n T}\bigr)}{ \sum_{i=1}^n \bigl(Z_{\frac in T}-Z_{\frac {i-1}n T}\bigr)^2}\\
	&=1+\frac{ \sum_{i=1}^{n-1} \bigl(Z_{ \frac{i+1}n T}-Z_{\frac in T}\bigr)\bigl(Z_{\frac in T}-Z_{\frac {i-1}n T}\bigr)}{ \sum_{i=1}^n \bigl(Z_{\frac in T}-Z_{\frac {i-1}n T}\bigr)^2}
\end{align*}
in the case where $\mathit{RV}(Z)^n_T>0$. 
Recognizing  the last term as  the  first-order sample autocorrelation of the increments of $Z$, we conclude that $Z$ is rough precisely when its increments are negatively correlated at high frequency.
\begin{theorem}\label{thm:equiv}
	A  function $Z=(Z_t)_{t\in[0,T]}$ is rough in the sense of Definition~\ref{def} if and only if  
	\begin{equation}\label{eq:neg} 
		\rho(\Delta Z)^n_T=\frac{\sum_{i=1}^{n-1} \bigl(Z_{ \frac{i+1}n T}-Z_{\frac in T}\bigr)\bigl(Z_{\frac in T}-Z_{\frac {i-1}n T}\bigr)}{\sum_{i=1}^n \bigl(Z_{\frac in T}-Z_{\frac {i-1}n T}\bigr)^2} \stackrel{p}{\longrightarrow} \rho 
	\end{equation}
	for some $ \rho<0$. In \eqref{eq:neg}, we let $0/0=1$. 
\end{theorem}

With this insight, we construct a rough volatility test based on the sample autocovariance of increments of spot volatility estimates. Such a covariance-based test has one important advantage over tests relying on variances or $p$-variations with different values of $p$. Mainly, while statistical errors in estimating spot volatility and biases from jumps and microstructure noise can heavily distort variance or $p$-variation estimates, they remain largely uncorrelated across non-overlapping blocks of data and, as a result, have a much smaller impact on autocovariances. For example, our test  achieves robustness to jumps and microstructure noise without  any procedure to remove them.


\subsection{Formulation of the Test}\label{sec:form}

We assume that we have high-frequency observations of the log-asset price process $x$ at  time points $0,\Delta_n,2\Delta_n,\dots$ up to time $T$, where $\Delta_n\rightarrow 0$ as $n\rightarrow\infty$ and $T\in(0,\infty)$, the number of trading days, is fixed.   
We split the data into blocks of $p_n$ increments, where $p_n$ is a positive integer increasing to infinity. Our test statistic is based on estimating   spot log-variance (i.e., $\log \si_t^2$) on each of the blocks and then computing  the first-order autocovariance of the increments of these volatility estimates. More specifically, following \citet{JT14} and \cite{LLL18}, we estimate the spot variance using the local empirical characteristic function of the price increments within each block. These local volatility estimators are given by 
\begin{equation}\label{eq:est}
	\wh c^n_j(u)=-\frac{2}{u^2}\log \lvert\wh L^n_j(u)\rvert, \quad\text{where}\quad \widehat{L}_j^n(u) = \frac{1}{k_n}\sum_{i=(j-1)p_n+1}^{(j-1)p_n+k_n}e^{iu  {\Delta_i^nx}/{\sqrt{\Den}}},
\end{equation}
for $j=1,\dots,  \lfloor T/(p_n\Den)\rfloor $, some $1<k_n\leq p_n$ and some $u\in\R\setminus\{0\}$. For a generic process $X$, we write $\Delta^n_i X=  X_{i\Den}-X_{(i-1)\Den}$. If $k_n<p_n$, we have time gaps between the blocks used in computing successive $\widehat{L}_j^n(u)$. This is done  to minimize the impact from potential dependent market microstructure noise in the observed price on the statistic, see Section~\ref{sec:noise} below. 

Denoting $\wh\cf^n_j(u)=\log \wh c^n_j(u)$, we next form increments between blocks, that is, $\Delta^n_j \wh \cf(u)=\wh \cf^n_j(u)-\wh \cf^n_{j-1}(u)$. We further difference $\Delta^n_j \wh \cf(u)$ between consecutive days and denote this difference by 
\begin{equation}\label{eq:c_diff_diff}
	\un\Delta^n_j \wh \cf(u,u')=\Delta^n_j \wh \cf(u)-\Delta^n_{j-\mathds{1}} \wh \cf(u'),~\textrm{for $u,u'\in\R\setminus\{0\}$}.
\end{equation} 
To simplify notation, we assume that $T$  and  $\mathds{1} =(p_n\Den)^{-1}$ are both even integers. In this case, $\mathds{1} $ is the number of blocks within one trading day, which is asymptotically increasing. The differencing across days  is done to mitigate the effect on the statistic from the presence of a potentially erratic periodic intraday volatility pattern, which can be quite pronounced as previous studies have found (see e.g., \citet{WMO85, H86, ASTZ24}). Indeed, if volatility is a product of a deterministic function of time of the trading day and a regular stochastic semimartingale, then the intraday periodic component of volatility gets canceled out in the asymptotic limit of $\un\Delta^n_j \wh \cf(u,u')$. Importantly, the daily differencing in (\ref{eq:c_diff_diff}) has no impact on the power of the test.

Our test statistic is then based on the sample autocovariance of $\un\Delta^n_j \wh \cf(u,u')$. More specifically, it is given by
\begin{equation}\label{eq:T3}
	\widehat{T}^n = \frac{\textstyle\sum_{j\in\calj_n}\un\Delta^n_{2j}\widehat{\cf}( {u}^n_j, {u}^n_{j-\dsone/2})\un\Delta^n_{2j-2}\widehat{\cf}( {u}^n_{j-1},{u}^n_{j-1-\dsone/2})}{\textstyle\sqrt{\sum_{j\in\calj_n}(\un\Delta^n_{2j}\widehat{\cf}( {u}^n_j, {u}^n_{j-\dsone/2})\un\Delta^n_{2j-2}\widehat{\cf}( {u}^n_{j-1},{u}^n_{j-1-\dsone/2}))^2}},
\end{equation} 
where $\calj_n=\bigcup_{k=1}^{T/2} \{2+ (2k-1)/(2p_n\Den) ,\dots,  2k/(2p_n\Den) \}$. This choice of $\calj_n$ avoids including the covariance between the last volatility increment of a trading day and the first one on the next trading day. This way, the summands in the numerator of $\widehat{T}^n$ are asymptotically uncorrelated under the null hypothesis and this makes the construction of a feasible estimator of the asymptotic variance straightforward. 

We note that one can easily derive the asymptotic distribution for higher-order sample autocovariances and include them in the formulation of the test as they should be all asymptotically zero under the null hypothesis and all asymptotically negative under the alternative hypothesis. We do not do this here because one can show that the first-order autocovariance is typically much higher in magnitude than the higher-order ones under the alternative hypothesis.  

The arguments $u^n_j$ of the empirical characteristic functions  used in $\widehat{T}^n$ are chosen in a data-driven way using information from preceding blocks. Their specification is given in the following assumption: 
\settheoremtag{U}
\begin{Assumption} \label{ass:U} We have
	\begin{equation}\label{eq:unj}
		{u}^n_j =  {\theta}/{\sqrt{\eta^n_j}},\qquad j=1,\dots,\lfloor T/(2p_n\Den)\rfloor,
	\end{equation}
	where  $\theta>0$  
	and $\eta^n_j$ is an $\calf_{(2j-2)p_n\Den}$-measurable random variable with the following properties: There is an adapted and stochastically continuous process $\eta_t$  such that $\inf\{ \eta_s:0\leq s\leq t\}>0$ almost surely for all $t>0$ and, for every $m\in\N$,
	\begin{equation}\label{eq:U-1} 
		\sup_{j=1,\dots,\lfloor T/(2p_n\Den)\rfloor}\E[\lvert\eta^n_j-\eta_{(2j-2)p_n\Den}\rvert\bone_{\{(2j-2)p_n\Den\leq \tau_m\}}] \to 0
	\end{equation}
	as $n\to\infty$,
	where $(\tau_m)_{m\in\N}$ is the sequence of stopping times from Assumption~\ref{ass:H0} or \ref{ass:H1}.
	Moreover, for every $m\in\N$, there is are   constants $0<\eta_m^-<\eta_m^+<\infty$ such that 
	\begin{equation}\label{eq:U-2} 
		\lim_{n\to\infty}	\P\Bigl( \eta_m^-\leq \eta^n_j \leq \eta_m^+\text{ for all }  j=1,\ldots,\lfloor T/(2p_n\Den)\rfloor\text{ with } (2j-2)p_n\Den\leq \tau_m\Bigr) = 1.
	\end{equation}
\end{Assumption}

\begin{Example}\label{ex:3} As we show in Lemma~\ref{lem:ex3}, a natural choice is  $\eta_t =\si^2_t$ together with
	\begin{equation}\label{eq:ups} 
		\eta^n_j= {\frac1{\ell_2-\ell_1+1}\sum_{\ell=\ell_1}^{\ell_2}\widehat{c}_{2j-\ell}^n},\qquad j=1,\ldots,\lfloor T/(2p_n\Den)\rfloor,
	\end{equation} for some fixed $3\leq \ell_1\leq\ell_2$ and where 
	\begin{equation}\label{eq:bp} 
		\wh c^n_j=\frac{\pi}{2k_n\Den}\sum_{i=(j-1)p_n+1}^{(j-1)p_n+k_n} \lvert\Delta^n_i x\Delta^n_{i-1} x\rvert
	\end{equation}
	is the bipower spot volatility estimator of \citet{barndorff2004power}. 
\end{Example}

\subsection{Asymptotic Behavior of the Test}

Let $\calf_\infty=\bigvee_{t\geq0}\calf_t$ and use $\stackrel{\call-s}{\longrightarrow}$ to denote $\calf_\infty$-stable convergence in law, see page 512 of \citet{JS03} for the   definition. The following theorem characterizes the asymptotic behavior of $\widehat{T}^n$ under the null and alternative hypotheses. 

\begin{theorem}\label{thm:H0}
	Suppose that
	\begin{equation}\label{eq:rates-0} 
		{k_n\sqrt{\Den}}\to0,\quad k_n\Den^{1/2-\iota}\to\infty\quad \text{for all } \iota>0,\quad  
		p_n/k_n \to \kappa \in[1,\infty),
	\end{equation}
	and consider the test statistic $\wh T^n$ as defined in \eqref{eq:T3}, where $u^n_j$ satisfies Assumption~\ref{ass:U}.
	\begin{enumerate}			\item[(i)] Under the null hypothesis set forth in Assumption \ref{ass:H0}, we have
		\begin{equation}\label{eq:T-H0}
			\widehat{T}^n\stackrel{\call-s}{\longrightarrow}N(0,1).
		\end{equation}
		\item[(ii)]  Under the alternative hypothesis set forth in Assumption \ref{ass:H1}, we have
		\begin{equation}\label{eq:T-alt}
			\widehat{T}^n\stackrel{\P}{\longrightarrow}-\infty \quad\text{at a rate of } \sqrt{ 1/{(p_n\Den)}}. 
		\end{equation}
	\end{enumerate}
	In particular, a test with critical  region
	\begin{equation}\label{eq:crit} 
		\calc^n = \{\wh T^n < \Phi^{-1}(\al)\},~~\alpha\in(0,1),
	\end{equation}
	where $\Phi$ is the  cumulative distribution function of the standard normal distribution, has asymptotic size $\al$ under Assumption~\ref{ass:H0} and is   consistent under Assumption~\ref{ass:H1}.
\end{theorem}

The requirements for the number of elements in a block, $k_n$, in (\ref{eq:rates-0}) is a standard condition that essentially balances asymptotic bias and variance in the spot volatility estimation. The limit behavior of the statistic under the null hypothesis is driven by the diffusive component of $x$. While infinite variation jumps are of higher asymptotic order, they can nevertheless have nontrivial effect in finite samples. The estimates of the asymptotic standard error of the autocovariance used in $\widehat{T}^n$ should automatically account for such higher order terms. Therefore, we expect good finite sample behavior of the test statistic even in situations with high jump activity. We confirm this in the Monte Carlo study in Section~\ref{sec:mc}. This is not the case for nonparametric estimates of diffusive volatility, which are known to have poor properties  in the presence of infinite variation jumps. Finally, the rate of divergence of the statistic under the alternative hypothesis is determined by the length of the interval over which local volatility estimates are formed. 

\subsection{Robustness to Microstructure Noise}\label{sec:noise}

Volatility estimators such as realized variance can be severely impacted by microstructure noise in  observed asset prices. By contrast, we show in this section that the test statistic $\wh T^n$ from \eqref{eq:T3} is \emph{naturally} robust to many forms of microstructure noise considered in the literature, without the need to employ classical noise-reduction techniques such as two-scale estimation (\citet{Zhang05}), realized kernels (\citet{barndorff2008designing}) or pre-averaging (\citet{Jacod09}). The reason for this is similar to that for the robustness towards infinite variation jumps. Mainly, the contribution of the noise to the empirical characteristic function gets canceled out in the differencing of the consecutive volatility estimates in $\Delta^n_j \wh \cf(u)$. 

We start with stating the counterparts to Assumptions~\ref{ass:H0}, \ref{ass:H1} and \ref{ass:U} in the noisy setting.

\settheoremtag{H$'_0$}
\begin{Assumption}\label{ass:H0-noise} The observed log-prices are given by
	\begin{equation}\label{eq:noise} 
		y_{i\Den}=x_{i\Den} + \eps^n_i,\qquad i=0,\dots,\lfloor T/\Den\rfloor,
	\end{equation}	
	where   the latent efficient price $x_{t}$ follows \eqref{eq:x0} and satisfies Assumption~\ref{ass:H0}, while the noise variables $\eps^n_i$ are of the form
	\begin{equation}\label{eq:eps} 
		\eps^n_i = \Den^\varpi \rho_{i\Den} \chi_i.
	\end{equation}
	Here, $\varpi\in[0,\infty)$ determines the size of the noise, $\rho$ is an adapted locally bounded process and $(\chi_i)_{i\in\Z}$ is a strictly stationary $m$-dependent
	sequence of random variables (for some $m\geq0$),  independent of $\calf_\infty$  and with mean $0$, variance $1$ and finite moments of all orders. 
	\begin{enumerate}			\item[(i)]  If $\varpi\leq \frac34$, we further assume that $\rho$ is an It\^o semimartingale of the form
		\begin{align} 
			\rho_t	&=\rho_0+\int_0^t\al^\rho_s ds + \int_0^t (\si^\rho_s dW_s +   \wt\si^\rho_s d\wt W_s) + \int_0^t\int_E \ga^\rho(s,z)(\mu-\nu)(ds,dz)\nonumber\\
			&\quad + \int_0^t\int_E \Ga^\rho(s,z)\mu(ds,dz),\label{eq:pi}
		\end{align}
		where $\wt W$ is a standard $\F$-Brownian motion that is independent of $W$ (and jointly Gaussian and potentially correlated with $\ov W$, $W^\vp$ and $\ov W^\vp$), $\rho_0$ is an $\calf_0$-measurable random variable, $\al^\rho$, $\si^\rho$ and $\wt\si^\rho$ are locally bounded adapted processes and $\ga^\rho$ and $\Ga^\rho$ are predictable functions. Moreover, as $\delta\to0$, the process $
		t\mapsto \sup_{s\in[0,\delta]}\E[\lvert \si^\rho_{t+s}-\si^\rho_t\rvert\wedge 1 \mid \calf_t]$
		converges uniformly on compacts in probability to $0$ and we have $(\lvert\ga^\rho(t,z)\rvert^2 + \lvert \Ga^\rho(t,z)\rvert) \wedge1\leq J_n(z)$ whenever $t\leq \tau_n$ (where $\tau_n$ is the localizing sequence from Assumption~\ref{ass:H0}).  
		\item[(ii)] If $\varpi\leq \frac12$, we have   $\inf\{\rho_t: 0\leq s\leq t\}>0$ for all $t>0$ almost surely. If $\varpi\leq \frac12$ or if $\theta_0=0$ in Assumption~\ref{ass:U-noise} below, there is $\varepsilon>0$ such that \eqref{eq:prop1} can be strengthened to 
		\begin{equation}\label{eq:prop1-noise} 
			(	\lvert \ga(t,z)\rvert^{2-\varepsilon}+\lvert \ga^c(t,z)\rvert^{2-\varepsilon}+\lvert \Ga(t,z)\rvert^{1-\varepsilon}+\lvert \Ga^c(t,z)\rvert^{1-\varepsilon})\wedge 1\leq J_n(z)
		\end{equation}
		and the processes in \eqref{eq:prop0-1} and \eqref{eq:prop3} as well as the constant $C^n(\delta)$ from \eqref{eq:prop4-2} still converge to $0$ if divided by $\delta^\varepsilon$.
	\end{enumerate}
\end{Assumption}

\settheoremtag{H$'_1$}
\begin{Assumption}\label{ass:H1-noise} We have \eqref{eq:noise}, where $x_t$ satisfies Assumption~\ref{ass:H1} and the noise variables $\eps^n_i$ satisfy \eqref{eq:eps}, where $\varpi$, $\rho$  and $(\chi_i)_{i\in\Z}$ have the properties listed in the sentence immediately after \eqref{eq:eps}.
	\begin{enumerate}
		\item[(i)] If $\varpi\leq\frac12+\frac{1}{2}H$,   we further assume that  
		\begin{equation}\label{eq:v-pi} 
			\rho_t=F(w_t),\quad	w_t = w_0 + \int_0^t G(t-s)(\si^{w}_sdW_s+\ov \si^w_sd\ov W_s+\wh\si^{w}_s d\wh W_s)+ \wt w_t,
		\end{equation} 
		where $F$ is a $C^2$-function, $w_0$ is   $\calf_0$-measurable and $G(t)=K_{H_\rho}^{-1} t_+^{H_\rho-1/2}+G_0(t)$ for some
		$H_\rho\in(0,\frac12)$ and $G_0\in C^1([0,\infty))$ with $G_0(0)=0$. Moreover, $\wh W$ is a standard $\F$-Brownian motion that is independent of $W$ and $\ov W$ and jointly Gaussian (and possibly correlated) with $W^\vp$, $\ov W^\vp$ and $\wt W$. The processes $\wt w$, $\si^w$, $\ov\si^w$ and $\wh\si^w$   are adapted and locally bounded,  and for any $n\in\N$ (and with $\tau_n$ as in \ref{ass:H0} and $C_n$ and $h$ as in \ref{ass:H1}), we have
		\begin{align}\label{eq:wt-v-pi} 
			&	\sup_{t\geq0 }	\E[\lvert\wt w_{t+\delta}-\wt w_t\rvert^2\bone_{\{t+\delta\leq \tau_n\}} ]^{1/2} \leq C_n\delta^{H_\rho}h(\delta),\\
			&	\sup_{t\geq0}\E[ (\lvert\si^w_{t+\delta}-\si^w_t\rvert^2+\lvert\ov\si^w_{t+\delta}-\ov\si^w_t\rvert^2+\lvert\wh\si^w_{t+\delta}-\wh\si^w_t\rvert^2)\bone_{\{t+\delta\leq \tau_n\}} ]^{1/2} \leq C_nh(\delta).\label{eq:wt-v-pi2}  
		\end{align}
		\item[(ii)]  If $\varpi\leq \frac12$ or if $\theta_0=0$ in Assumption~\ref{ass:U-noise} below, we have $h(t)=O(t^\vareps)$ as $t\to0$ for some $\varepsilon>0$. 
		\item[(iii)] If $\varpi<\frac12+\frac12H$ and $(p_n\Den)^H \Den^{(1-2\varpi)_+} = o((p_n\Den)^{H_\rho}\Den^{(2\varpi-1)_+})$,  then almost surely $\inf\{(\rho_sF'(w_s))^2[ (\si^w_s)^2+(\ov\si^w_s)^2+(\wh\si^w_s)^2]: 0\leq s\leq t\}>0$ for all $t>0$. If  $\varpi=\frac12$ and $H=H_\rho$, then almost surely $\inf\{z_s: 0\leq s\leq t\}>0$ for all $t>0$, where 
		\begin{align*}
			z_t&=(\tfrac12 f'(v_t)\si^v_t + \E[(\Delta\chi_1)^2]\rho_tF'(w_t)\si^w_t)^2\\
			&\quad+(\tfrac12 f'(v_t)\ov\si^v_t + \E[(\Delta\chi_1)^2]\rho_tF'(w_t)\ov\si^w_t)^2+( \E[(\Delta\chi_1)^2]\rho_tF'(w_t)\wh\si^w_t)^2
		\end{align*}
		and $\Delta\chi_i = \chi_i-\chi_{i-1}$.
	\end{enumerate}
\end{Assumption}

We consider the case when the noise is asymptotically shrinking ($\varpi>0$) and when this is not the case ($\varpi=0$). Naturally, the requirements for the process $\rho$ are somewhat stronger for lower values of $\varpi$ (which correspond to asymptotically bigger noise).  A typical example where the conditions on $\rho$ are met, irrespective of  the size of the noise, is when $\rho$ is a function of volatility, in which case we have $w=v$ and $H_\rho=H$. Given prior evidence for a strong relationship between noise intensity and volatility, this is also the empirically relevant case.

To test \ref{ass:H0-noise} against \ref{ass:H1-noise}, we use the  same test statistic $\wh T^n$ from \eqref{eq:T3} as before, except that now
\begin{equation}\label{eq:est-2}
	\widehat{L}_j^n(u) = \frac{1}{k_n}\sum_{i=(j-1)p_n+1}^{(j-1)p_n+k_n}e^{iu \Delta_i^ny/\sqrt{\Den}},
\end{equation}
which uses the observed prices $y$ instead of the latent efficient price $x$.  We do not change the notation of $\wh T^n$, $\wh \cf^n_j(u)$, $\wh c^n_j(u)$ and $\widehat{L}_j^n(u)$ to reflect the fact that \eqref{eq:est} is a special case of \eqref{eq:est-2}, namely when  microstructure noise is absent. More importantly, in practice, one does not have to know or decide whether noise is present or not. We do have to adjust Assumption~\ref{ass:U} in the presence of noise:
\settheoremtag{U$'$}
\begin{Assumption}\label{ass:U-noise}
	We have Assumption~\ref{ass:U} with the following modifications: 
	\begin{enumerate}			\item[(i)] In \eqref{eq:unj}, $\theta=\theta_n$  is a sequence of positive numbers such that $\theta_n\to\theta_0\in[0,\infty)$. If $\varpi\leq \frac12$, we have $\theta_0=0$. The variables  $\eta^n_j$ are $\calf_{((2j-2)p_n-m-1)\Den}$-measurable. 
		\item[(ii)]  In \eqref{eq:U-1} and \eqref{eq:U-2}, $\eta^n_j$ is replaced by $\eta^n_j/\Den^{(2\varpi-1)\wedge 0}$, and \eqref{eq:U-1} still holds if the left-hand side is divided by $\Delta_n^\varepsilon$. The numbers   $\varpi$ and $\vareps$ are the same numbers as in \eqref{eq:eps} and \eqref{eq:prop1-noise}.
		\item[(iii)]  The process $
		t\mapsto \sup_{s\in[0,\delta]}\E[\lvert \eta_{t+s}-\eta_t\rvert\wedge 1 ]/\delta^\vareps$ converges uniformly on compacts in probability to $0$ as $\delta\to0$.
	\end{enumerate} 
\end{Assumption}

\begin{theorem}\label{thm:test-noise}
	Suppose that Assumption~\ref{ass:U-noise} is satisfied. If 
	\begin{equation}\label{eq:rates} \begin{split}
			&\frac{k_n\sqrt{\Den}}{\theta_n^4}\to0, \qquad k_n\Den^{1/2-\iota}\to\infty\quad \text{for all } \iota>0,\\  
			&p_n/k_n \to \kappa \in[1,\infty),\qquad k_n\leq p_n-m-1,\end{split}
	\end{equation}
	then the behavior of the test statistic $\wh T^n$ remains the same as described in \eqref{eq:T-H0} or \eqref{eq:T-alt} depending on whether we have, respectively, Assumption~\ref{ass:H0-noise} or \ref{ass:H1-noise}. A test based on the critical region \eqref{eq:crit} has asymptotic size $\al$ under Assumption \ref{ass:H0-noise} and is consistent under \ref{ass:H1-noise}.
\end{theorem}

The rate conditions in (\ref{eq:rates}) are slightly stronger than those in (\ref{eq:rates-0}) corresponding to the case of no market microstructure noise. In particular, the last condition in (\ref{eq:rates}) is necessary to avoid the noise having an asymptotic effect on the statistic. Since one does not typically know a priori the degree of dependence in the noise, it is better to pick $\kappa$ slightly above one.  

In the noisy case, we need $\theta=\theta_n\to0$ if $\varpi\leq \frac12$. The consequence of that is that we only need to assess the characteristic function of the noise for values of the exponent around zero. Setting $\rho\equiv0$ in Theorem~\ref{thm:test-noise}, we can verify that Theorem~\ref{thm:H0} remains true even if we choose $\theta=\theta_n\to0$ (subject to \eqref{eq:rates} and the strengthened conditions for $\theta_0=0$). Similarly, as shown in Lemma~\ref{lem:ex3}, Example~\ref{ex:3} can be extended to the noisy case with $\eta_t=\si^2_t\bone_{\{\varpi\geq\frac12\}}+K(\chi)\rho_t^2\bone_{\{\varpi\leq \frac12\}}$, where $K(\chi)=\frac\pi2 \E[\lvert \Delta\chi_2\Delta\chi_1\rvert]$. Importantly from an applied point of view, in order to implement the test, we do not need to decide  whether microstructure noise is present or not (or know the value of $\varpi$). 

Finally, the rate of convergence of the statistic is determined either by the noise or by the diffusive component of the price. This depends on the value of $\varpi$, i.e., on how big in asymptotic sense the noise is. The user does not need to know this in implementing the test. 

\begin{remark} We conjecture that in Theorem~\ref{thm:test-noise}, the $m$-dependence assumption on $\chi$  can be weakened by requiring $\chi$ to be $\iota$-polynomially $\rho$-mixing for some $\iota>1$ as in \citet{jacod2017statistical}, \citet{DX21} and \citet{li2022remedi}.
\end{remark}

\begin{remark}\label{rem:weak}
	As \citet{LPSY22} and \citet{SY23} show, there is a weak identification issue in estimating discrete-time rough volatility  models. For example, within the class of ARFIMA models, it is asymptotically impossible to distinguish a near-stationary long-memory model from a rough model with near unit root dynamics. As our theoretical analysis shows, this weak identification issue is a feature of discrete-time volatility models only. In fact, roughness of a function is intrinsically a continuous-time concept and one can achieve identification of the roughness of the volatility path by considering an infill asymptotic setting. Therefore, our test can identify rough volatility irrespective of whether volatility has short or  long memory behavior.
\end{remark}

\begin{remark}\label{rem:smooth}
	As we mentioned above, the hypotheses $H_0$ and $H_1$ in \eqref{eq:hypo} encompass most continuous-time volatility models considered in prior work. That said, there are specifications for $c_t$ that do not belong to  either $H_0$ or $H_1$. The most notable such specification is one in which $c_t$ is a non-semimartingale process with finite quadratic variation. This is for instance the case if $c_t$ is a fractional process of the form \eqref{eq:rough} but with $H>\frac12$ in \eqref{eq:g} (and  $H_\rho>\frac12$ if there is noise). In this case, the volatility paths will have zero quadratic variation and the high-frequency increments of volatility will have  positive autocorrelations.  In fact, denoting $H_\ast= (H+(1-2\varpi)_+)\wedge (H_\rho+(2\varpi-1)_+)$, one can show  that  if $H_\ast\in(\frac12,\frac34)$, then   $\wh T^n\stackrel{\P}{\longrightarrow}+\infty$. This is because the positive autocovariance of volatility (or noise intensity) increments dominates in this case. If $H_\ast\in(\frac34,1)$, then estimation errors  dominate and we have $\wh T^n\limst N(0,1)$. In summary, the asymptotic rejection rate of a one-sided test based on the critical region in \eqref{eq:crit} will not exceed  the nominal significance level under  such very smooth specification for $c_t$. 
	As we do not observe any significant   positive values of the test statistic in the empirical application, we omit a formal proof of the aforementioned fact.
\end{remark}



\section{Monte Carlo Study}\label{sec:mc}

In this section, we evaluate the performance of the test on simulated data.
\subsection{Setup}
We use the following model for $x$:
\begin{equation}
	dx_t= \sqrt{V_t}dW_t+\int_{\mathbb{R}}x(\mu(dt,dx)-dt\nu_t(dx)),
\end{equation}
where $W_t$ is a standard Brownian motion, $\mu$ is an integer-valued random measure counting the jumps in $x$ with compensator measure $dt\nu_t(dx)$ and $\nu_t(dx)$ is given by
\begin{equation}
	\nu_t(dx) =  c\frac{e^{-\lambda|x|}}{|x|^{\alpha+1}}V_t dx.
\end{equation}
The volatility under the null hypothesis (corresponding to $H=1/2$) follows the standard Heston model:
\begin{equation}
	V_t = V_0 + \int_0^t\left(\kappa(\theta-V_s)ds + \nu\sqrt{V_s}dB_s\right),
\end{equation}
where $B_t$ is a Brownian motion with $\textrm{corr}(dW_t,dB_t) = \rho dt$. The volatility under the alternative hypothesis follows the rough Heston model of \citet{jaisson2016rough}: 
\begin{equation}
	V_t = V_0 + \frac{1}{\Gamma(H+1/2)}\int_0^t(t-s)^{H-1/2}\left(\kappa(\theta-V_s)ds + \nu\sqrt{V_s}dB_s\right),
\end{equation}
where again $B_t$ is a Brownian motion with $\textrm{corr}(dW_t,dB_t) = \rho dt$.

In the above specification of $x$, we allow both for stochastic volatility, which can exhibit rough dynamics, and jumps. The jumps are modeled as a time-changed tempered stable process, with the time change  determined by the diffusive volatility as in \citet{DPS00}. This is consistent with earlier empirical evidence for jump clustering. The parameters $\lambda$ and $\alpha$ control the behavior of the big and small jumps, respectively. In particular, $\alpha$ coincides with the Blumenthal--Getoor index of $x$ controlling the degree of jump activity. We fix the value of $\lambda$ throughout and consider different values of $\alpha$. For given $\lambda$ and $\alpha$, we set the value of the scale parameter $c$ to 
\begin{equation}
	c = 0.1\times \frac{\lambda^{2-\alpha}}{\Gamma(2-\alpha)}.
\end{equation} 
With this choice of $c$, we have $\int_{\mathbb{R}}x^2\nu_t(dx) = 0.2\times V_t$, which is roughly consistent with earlier nonparametric evidence regarding the contribution of jumps to asset price variance. 

Turning next to the specification of the volatility dynamics, we set the mean of the diffusive variance to $\theta = 0.02$ (unit of time corresponds to one year) and the mean reversion parameter to $\kappa = 8$, which corresponds approximately to half-life of a shock to volatility of one month. The volatility of volatility parameter in the Heston model is set to $\nu = 0.45$ (recall that the Feller condition puts an upper bound on $\nu$ of $\nu<\sqrt{2\kappa\theta}$). We then set the value of $\nu$ for the different rough specifications (i.e., for the different values of $H<1/2$) so that the unconditional second moment of $V_t$ is the same across all models. 

Finally, we assume that the observed log-price is contaminated with noise, i.e., instead of observing $x$, we observe
\begin{equation}
	y_{i\Delta_n} = x_{i\Delta_n}+\sigma_{\textrm{noise}}\sqrt{V_{i\Delta_n}}\varepsilon_i,\quad i=1,\dots,n,
\end{equation} 
where $\{\varepsilon_i\}_{i=1,\dots,n}$ is an i.i.d.\ sequence of standard normal random variables defined on a product extension of the sample probability space and independent from $\mathcal{F}$. We set $\sigma_{\textrm{noise}}^2 =  0.5\times \frac{1}{252} \times (1.0548-1)/(4620-77\times 1.0548) = 2.40\times 10^{-8}$. The value $1.0548$ corresponds to the median of the ratio of daily realized volatility from five-second returns over the daily realized volatility from five-minute returns in the data set used in the empirical application. The numbers $4620$ and $77$ correspond, respectively, to the number of five-second and five-minute returns for our daily estimation window, see the discussion below for our sampling scheme. The above calibration of the noise parameter is based on the fact that in our model, the ratio of daily realized variances at five seconds versus five minutes is approximately equal to $\frac{1/252+2\times 4620\times\sigma_{\textrm{noise}}^2}{1/252+2\times 77\times\sigma_{\textrm{noise}}^2}$, with a business time convention of $252$ trading days per year (which we use as our unit of time here).

The parameter values for all specifications used in the Monte Carlo are given in Table~\ref{table:pars}. We consider two different values of the Blumenthal--Getoor index: one corresponds to finite variation jumps ($\alpha = 1/2$) and one to infinite variation jumps ($\alpha = 3/2$). For the volatility specification, we consider three values of the roughness parameter: $H=0.5$ (this is the standard Heston model), $H=0.3$ and $H=0.1$, with the first value corresponding to the null hypothesis and the last two to the alternative hypothesis. Recall that lower values of $H$ imply rougher volatility paths and the rough volatility literature argues for value of $H$ around $0.1$ (\citet{gatheral2018volatility}).

\begin{table}[ht!]
	\setlength{\tabcolsep}{0.20cm}
	\begin{center}\small 
		\caption{{\bf Parameter Setting for the Monte Carlo}\label{table:pars}}
		\begin{tabularx}{0.95\textwidth}{Xccccccccccc}
			\toprule
			Case    & \multicolumn{5}{c}{Variance Parameters} & & \multicolumn{3}{c}{Jump Parameters} & & \multicolumn{1}{c}{Noise Parameter}\\ 
			\cline{2-6} \cline{8-10} 
			& $H$ & $\theta$ & $\kappa$ & $\nu$ & $\rho$ & & $\alpha$ & $\lambda$ & $c$  & & $\sigma_{\textrm{noise}}$\\
			\hline
			V1-J1 & $0.1$ & $0.02$ & $8$ & $0.10$ & $-0.7$ & &  $0.5$ & $500$ & $1262$ & & $1.55\times 10^{-4}$\\ 
			V1-J2 & $0.1$ & $0.02$ & $8$ & $0.10$ & $-0.7$ & &  $1.5$ & $500$ & $1.26$ & & $1.55\times 10^{-4}$\\[+0.5ex]	
			V2-J1 & $0.3$ & $0.02$ & $8$ & $0.22$ & $-0.7$ & & $0.5$ & $500$ & $1262$ & & $1.55\times 10^{-4}$\\
			V2-J2 & $0.3$ & $0.02$ & $8$ & $0.22$ & $-0.7$ & & $1.5$ & $500$ & $1.26$ & & $1.55\times 10^{-4}$\\[+0.5ex]	
			V3-J2 & $0.5$ & $0.02$ & $8$ & $0.45$ & $-0.7$ & & $0.5$ & $500$ & $1262$ & & $1.55\times 10^{-4}$\\					
			V3-J2 & $0.5$ & $0.02$ & $8$ & $0.45$ & $-0.7$ & & $1.5$ & $500$ & $1.26$ & & $1.55\times 10^{-4}$\\
			\bottomrule
		\end{tabularx}
	\end{center}
\end{table}

We simulate all models using a standard Euler scheme. Since the rough volatility models are non-Markovian and because the asymptotic distribution of the test statistic is determined by the Brownian motion $W_t$ driving the price $x$, we independently generate  blocks of seven days of daily high-frequency data in order to save computational time. The first five days are used for determining the value of the characteristic exponent $\eta_j^n$ on the last two days as will become clear below. For each independent block of data we then keep the products of the differenced volatility increments over the last day.  

The starting value of volatility is set to its unconditional mean. The sampling frequency is five seconds. 
As a result, on each day, we have $4680$ five-second price increments in a $6.5$ hour trading day. This matches exactly the number of high-frequency daily observations in our empirical application. To further match what we do in the application, we drop the first $5$ minutes on each day in the Monte Carlo. This leads to a total of $4620$ five-second return observations per day that we use in the construction of the test.       

In the top left panel of Figure~\ref{fig:vol_path}, we plot a simulated path from the rough volatility specification V1-J1 with $H=0.1$ over one day. One can clearly notice the sizable short-term oscillations of the spot volatility which is due to the roughness of the volatility path. In the top right panel of Figure~\ref{fig:vol_path}, we display the integrated volatility over intervals of  five minutes. Our local volatility estimators $\widehat{c}_j^n(u)$ can be thought of as estimators of such integrated volatilities over short windows.  The oscillations in the five-minute integrated volatility series are much smaller than in the spot volatility one. This is not surprising and illustrates  the difficulty of the testing problem at hand. Nevertheless, even when volatility is integrated to five-minute intervals, one can notice frequent  short-term volatility reversals. These reversals lead to negative autocorrelation in the first difference of the five-minute integrated volatility and will effectively provide the power of our test.    
\begin{figure}[t!]
	\begin{center}
		\includegraphics[width=\linewidth]{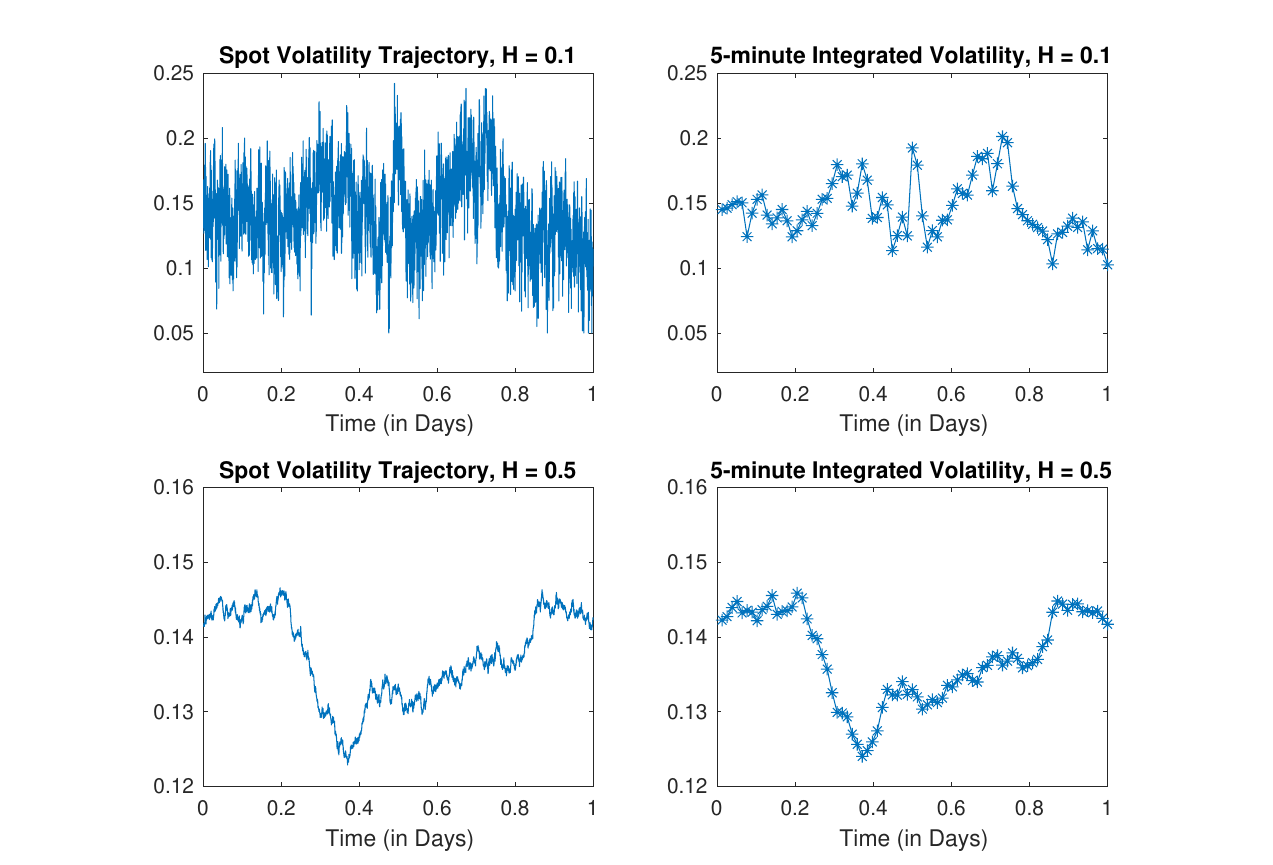}
		\caption{All time series are reported in annualized volatility units. The simulated volatility paths are from specification V1-J1 (left panels) and V3-J1 (right panels).}\label{fig:vol_path} 
	\end{center}
\end{figure}
In the bottom panels of Figure~\ref{fig:vol_path}, we plot the spot and five-minute integrated volatilities for a simulation from the specification V3-J1, which corresponds to the null hypothesis with $H = 0.5$. The oscillations in spot volatility are now orders of magnitude smaller. Furthermore, the five-minute integrated volatility from specification V3-J1 does not exhibit any distinguishable short-term reversals. 

The contrast between the spot volatility and the integrated volatility in the two top panels of Figure~\ref{fig:vol_path} further illustrate the difficulty of estimating spot volatility in a rough setting. For example, the relative absolute bias of five-minute integrated volatility as a proxy for the spot volatility at the beginning of the interval is around 20\% in scenario V1-J1 (corresponding to $H=0.1$) and only 1\% in scenario V3-J1 (corresponding to $H=0.5$). 

We finish this section with explaining our choice of tuning parameters. For computing the statistic, we use a block size of $p_n = 60$, which leads to $ 4620/p_n = 77$ blocks per day. On each local block, we use the first four minutes to estimate volatility, i.e., we set $k_n = 48$. This implies a one-minute gap between blocks of increments used in computing the empirical characteristic functions. Finally, for each day, the characteristic exponent parameter $\theta_n$ is set to 
\begin{equation}
	\theta_n = \sqrt{-2\log(\mathfrak{L}_n)},~~0<\mathfrak{L}_n<1,
\end{equation}
and the data dependent scale parameter $\eta_j^n$ is set to 
\begin{equation}
	\eta_{j}^n = \frac{1}{5}\sum_{k=1}^5\widehat{c}_{l_{j,k}},
\end{equation}
where $l_{j,k}$ correspond to the indices of the blocks of increments over the past five trading days covering the same time of day as the $(2j-1)$-th and  $2j$-th blocks. (Extending Example~\ref{ex:3} to the case where gaps are permitted, one can easily show that this specification of $\eta^n_j$ satisfies Assumptions~\ref{ass:U} and \ref{ass:U-noise}.) With this choice of tuning parameters, $\widehat{L}_j^n(u_j^n)$ has a norm of approximately $\mathfrak{L}_n$. We experiment with three different values of $\mathfrak{L}_n$ of $0.95$, $0.75$ and $0.50$. These correspond to $\theta_n = 0.32$, $0.76$ and $1.18$, respectively.

\subsection{Results}

The Monte Carlo results are reported in Table~\ref{tb:mc}. We can draw several conclusions from them.  First, the test appears correctly sized with only minor deviations from the nominal size level across the different configurations. Second, the test has very good power against volatility specification with very rough paths, i.e., against the specification V1 corresponding to $H=0.1$. Earlier work arguing for presence of rough volatility has found the rough parameter $H$ to be close to zero and below $0.1$. As seen from the reported simulation results, our test can easily separate such an alternative hypothesis from the null hypothesis of smooth It\^o semimartingale volatility dynamics. Indeed, the empirical rejection rates of the test in scenarios V1-J1 and V1-J2 is 100\%. Third, the power of the test is significantly lower against rough volatility with $H=0.3$. This is not surprising because this case corresponds to significantly less roughness in the volatility paths and  illustrates the difficulty of the testing problem. Fourth, the power of the test decreases slightly with the decrease in $\mathfrak{L}_n$, with the performance for $\mathfrak{L}_n=0.95$ and $0.75$ being very similar. Naturally, the power of the test increases when including more data, i.e., when going from one to four years of data. Finally, we note that neither the size nor the power properties of the test are affected by the level of jump activity, which is consistent with our theoretical derivations.     

\begin{table}[h!]                                                                                                   
	\centering                                                                                                                              
	\caption{Monte Carlo Results}\label{tb:mc}                                                                                                                    
	\begin{tabularx}{0.9\textwidth}{X|ccc|ccc}                                                                                             
		\toprule                                                                                                                                
		Case 	& \multicolumn{3}{c}{\textbf{$n = 4680\times 250$}} & \multicolumn{3}{c}{\textbf{$n = 4680\times 1000$}}\\[+0.5ex]	
		\hline                                                                                                                             
		& $\mathfrak{L}_n = 0.95$ & $\mathfrak{L}_n = 0.75$ & $\mathfrak{L}_n = 0.50$ & $\mathfrak{L}_n = 0.95$ & $\mathfrak{L}_n = 0.75$ & $\mathfrak{L}_n = 0.50$ \\ 
		\hline  
		V1-J1 & 1.0000 & 1.0000 & 1.0000 & 1.0000 & 1.0000 & 1.0000 \\                             
		V1-J2 & 1.0000 & 1.0000 & 1.0000 & 1.0000 & 1.0000 & 1.0000  \\	 [+0.5ex]	   
		
		V2-J1 & 0.1640 & 0.1640 & 0.1240 & 0.3280 & 0.3160 & 0.2760 \\                             
		V2-J2 & 0.1360 & 0.1120 & 0.0800 & 0.3480 & 0.3360 & 0.2800  \\ [+0.5ex]

		V3-J1 & 0.0570 & 0.0600 & 0.0490 & 0.0510 & 0.0440 & 0.0430 \\                             
		V3-J2 & 0.0580 & 0.0550 & 0.0540 & 0.0480 & 0.0580 & 0.0680  \\
		\bottomrule                                                                                                                                                                                                                                                             
	\end{tabularx} 
	\smallskip
	\begin{minipage}{\textwidth}                                                                                                   
		\emph{Note}: Reported results are empirical rejection rates of the test with nominal size of $0.05$ based on $1{,}000$ Monte Carlo replications for the scenarios corresponding to the null hypothesis and on $250$ Monte Carlo replications for the scenarios corresponding to the alternative hypothesis. 
	\end{minipage}                                                                                                                          
\end{table}

\section{Empirical Application}\label{sec:emp}

We apply the test for volatility roughness to high-frequency data for the SPY exchange traded fund tracking the S\&P 500 market index over the period 2012--2022. On each trading day, we sample the SPY price at five-second frequency from 9:35 AM ET until 4:00 PM ET. We drop days with more than $20\%$ zero five-minute returns over the entire trading period of 6.5 hours. This removes from the analysis mostly half-trading days around holidays when liquidity tends to be lower. We further exclude days with FOMC announcements. The choice of $p_n$ and $k_n$ as well as of the characteristic exponents is  done  exactly as in the Monte Carlo.  

The test results are reported in Table~\ref{tb:emp}. They provide evidence for existence of rough volatility. Indeed, when conducting the test over the different calendar years in our sample, we reject the null hypothesis in 5 to 7 (depending on the level of $\mathfrak{L}_n$) out of a total of 11 years at a significance level of $5\%$. If we conduct the test over periods equal to or exceeding 3 years, then we always reject the null hypothesis at the same significance level of $5\%$. Comparing the performance of the test on the data and in the simulations, we see that the rejection rates on the data are lower than for the case $H=0.1$ in the simulations but higher than those for the case $H=0.3$.     

\begin{table}[h!]                                                                                                   
	\centering                                                                                                                              
	\caption{Empirical Test Results}\label{tb:emp}                                                                                                                    
	\begin{tabularx}{0.97\textwidth}{c|ccc||c|ccc}                                                                                             
		\toprule                                                                                                                                
		Year 	& \multicolumn{3}{c||}{Test Statistic} & Year & \multicolumn{3}{c}{Test Statistic} \\[+0.5ex]	
		\hline                                                                                                                             
		&  $\mathfrak{L}_n = 0.95$ & $\mathfrak{L}_n = 0.75$ & $\mathfrak{L}_n  = 0.50$ & & $\mathfrak{L}_n = 0.95$ & $\mathfrak{L}_n = 0.75$ & $\mathfrak{L}_n = 0.50$\\ 
		\hline  
		2012 & ~0.26 & ~0.05 & -0.12 & \multirow{4}{*}{2012--2015} & \multirow{4}{*}{-2.34} & \multirow{4}{*}{-1.81} & \multirow{4}{*}{-1.45}\\ 
		2013 & -2.88 & -1.93 & -1.63 &  &  & & \\ 
		2014 & -1.62 & -1.24 & -0.82 &  &  &  & \\    
		2015 & -0.17 & -0.59 & -0.95 &  &  &  & \\ [+0.75ex]
		2016 & -1.04 & -1.66 & -1.75 & \multirow{4}{*}{2016--2019} & \multirow{4}{*}{-3.67} & \multirow{4}{*}{-3.18} & \multirow{4}{*}{-2.59}\\ 
		2017 & -2.02 & -2.08 & -2.61 &  &  &  & \\ 
		2018 & -1.61 & -1.40 & -1.87 &  &  &  & \\    
		2019 & -3.10 & -2.31 & -0.79 &  & &  & \\[+0.75ex]
		2020 & -1.92 & -2.07 & -2.03& \multirow{3}{*}{2020--2022} & \multirow{3}{*}{-2.17} & \multirow{3}{*}{-3.01} & \multirow{3}{*}{-4.06} \\ 
		2021 & -1.74 & -2.47 & -3.44 &  &  &  & \\ 
		2022 & -1.17 &-1.74 & -1.84 &  & &  & \\[+0.75ex]  
		&             &             &             & \multirow{1}{*}{2012--2022} & -5.27 & -5.04 & -4.57\\                 
		\bottomrule                                                                                                                                                                                                                                                             
	\end{tabularx}                                                                                                        
\end{table}  

In Figure~\ref{fig:auto}, we plot the autocorrelation of the differenced volatility increments over the entire sample period 2012--2022. Under the null hypothesis of no rough volatility, the autocorrelations should be all zero asymptotically. Under the alternative of rough volatility, these autocorrelations should be negative asymptotically, with the highest in magnitude being the first one. The reported autocorrelations up to lag 7 are all negative, with the highest in magnitude being the first one. This is consistent with existence of rough volatility. We note that the reported autocorrelations are small in absolute value. This is at least in part due to the nontrivial measurement error in the volatility estimates.  

\begin{figure}[h!]
	\begin{center}
		\includegraphics[scale=0.6]{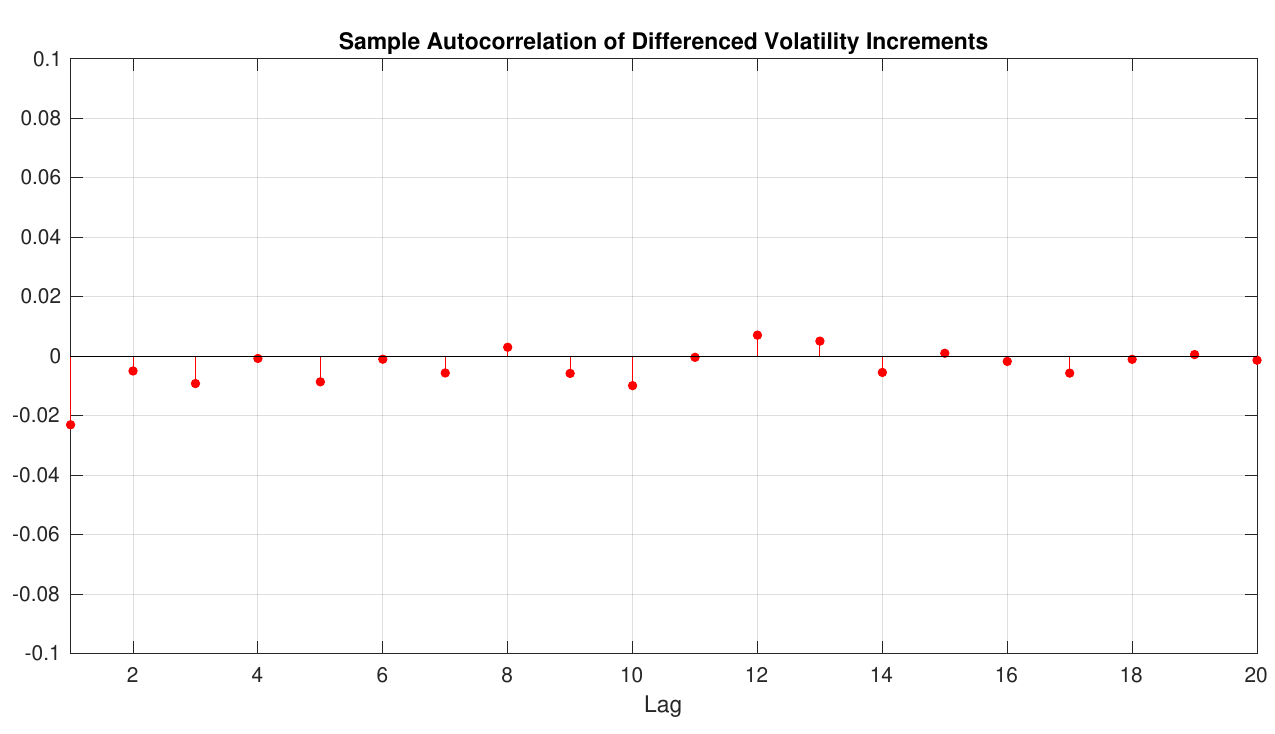}
		\caption{The plot displays autocorrelation in $\un\Delta^n_{2j}\widehat{\cf}( {u}^n_j, {u}^n_{j-\dsone/2})$ for the S\&P 500 index.}\label{fig:auto} 
	\end{center}
\end{figure}

\section{Concluding Remarks}\label{sec:concl}

We develop a nonparametric test for rough volatility based on
high-frequency observations of the underlying asset price process
over an interval of fixed length. The test is based on the sample autocovariance of increments of local volatility estimates formed from blocks of high-frequency observations with asymptotically shrinking time span. The autocovariance, after suitable normalization, converges to a standard normal random variable 
under the null hypothesis if volatility is not rough and to negative infinity
in the case of rough volatility. The proposed test is robust to the presence of price and volatility jumps with arbitrary degree of jump activity and to observation errors in the underlying process. 

Implementing the test on SPY transaction data, we find  evidence  in support of rough volatility throughout the past eleven years. As a consequence of this finding,  nonparametric estimation of spot volatility is much more difficult in reality than what is implied by standard models  with It\^o semimartingale volatility dynamics used in economics and finance. Indeed, the rate of convergence for estimating  spot variance $c_t$ at a fixed time point $t$ becomes arbitrarily slow  as volatility roughness increases. In spite of this observation, classical spot variance estimators   based on sums of squared increments retain their usual  rates of convergence   in a rough volatility setting    when viewed as estimators of \emph{local averages} $\frac1{\Delta}\int_t^{t+\Delta} c_s ds$ of spot variance, where $\Delta$ is the length of the estimation window. Therefore, regardless of how rough volatility is, practically feasible inference is possible for local averages of spot variance rather than spot variance itself. 

That said, many estimators in high-frequency financial econometrics, other than the standard realized variance, rely on volatility being a semimartingale process, see e.g., the book of \citet{JP12}. We leave a detailed investigation of the implications of rough volatility on the properties of such estimators for future work.

\begin{appendix}
	\section{Proof of Main Results}\label{sec:proof}
	
	\subsection{Proof of Theorems~\ref{thm:H0} and \ref{thm:test-noise}}
	
	Since the noise-free case is obtained from the noisy one by choosing $\varpi=1$, $\rho=0$ and $m=-1$ (which we formally allow in the proof below), Theorem~\ref{thm:H0} is just a special case of  Theorem~\ref{thm:test-noise}.
	We sketch the main steps in the proof of the latter, with technical details deferred to Appendix~\ref{sec:proof-app} in the supplement (\citet{supp}). 
	By a classical localization procedure (see e.g., \citet[Chapter 4.4.1]{JP12}), it suffices to prove Theorem~\ref{thm:test-noise} under the following Assumptions \ref{ass:H0-2}, \ref{ass:H1-2} and \ref{ass:U-2}, which are strengthened versions of Assumptions \ref{ass:H0-noise}, \ref{ass:H1-noise} and \ref{ass:U-noise}, respectively. If $\theta_0=0$ in Assumption~\ref{ass:U-noise}, we  write $a_n=o'(b_n)$ to denote $a_n/(b_n\Delta_n^\iota)\to0$ for some $\iota>0$ (that can vary from one place to another); if $\theta_0>0$, $a_n=o'(b_n)$  shall mean the same as $a_n=o(b_n)$. 
	
	\settheoremtag{SH$'_0$}
	\begin{Assumption}\label{ass:H0-2} For any compact subset $\calu$ of $\R$ and with the notation $0/0=0$, the following holds in addition to Assumption~\ref{ass:H0-noise} (the assumptions on $\al^\rho$, $\si^\rho$, $\wt\si^\rho$, $\ga^\rho$ and $\Ga^\rho$  only apply if $\varpi\leq \frac34$):
		\begin{enumerate}		\item[(i)] We have $\Ga^c\equiv \Ga^\rho\equiv  0$ and  the processes $x$, $\al$, $c$, $\al^c$, $\si^c$, $\ov\si^c$,  $\rho$, $\al^\rho$, $\si^\rho$, $\wt \si^\rho$ and
			\begin{equation}\label{eq:prop2-2} 
				t\mapsto\sup_{\delta\in(0,1)} \sup_{u\in\calu} \bigl\{	\delta\lvert\al^\vp(u/\sqrt{\delta})_t\rvert\} 
			\end{equation}
			are uniformly bounded by a constant $K\in(0,\infty)$. 
			Moreover,  $c_t\geq K^{-1}$ and, if $\varpi\leq \frac12$, we have $\lvert\rho_t\rvert\geq K^{-1}$ for all $t>0$ almost surely.
			\item[(ii)] There is a nonnegative measurable function $J(z)$ satisfying $J(z)\leq K$ as well as $\int_E J(z)\la(dz)\leq K$ such that for all $z\in E$,
			\begin{equation}\label{eq:prop1-2} 
				\sup_{\om\in\Om}	\sup_{t\in[0,\infty)}	\bigl\{	\lvert \ga(t,z)\rvert^2+\lvert \ga^c(t,z)\rvert^2+\lvert \ga^\rho(t,z)\rvert^2 + \lvert \Ga(t,z)\rvert\bigr\}\leq J(z).
			\end{equation}
			\item[(iii)]   
			As $\delta\to0$, the following   sequences all converge to $0$:
			\begin{align}\label{eq:prop0-2}
				w(\delta)&=\begin{aligned}[t]
					&\sup_{\om\in\Om,0\leq s\leq t,\lvert t-s\rvert \leq \delta} \bigl\{\E[\lvert \al_t-\al_s\rvert\mid\calf_s]\\
					&\qquad\qquad\qquad\qquad+ \E[\lvert \si^c_t-\si^c_s\rvert\mid\calf_s]+ \E[\lvert \si^\rho_t-\si^\rho_s\rvert\mid\calf_s]  \bigr\},\end{aligned}\\
				C_1(\delta)&=\sup_{\om\in\Om}	\sup_{t\in[0,\infty)}  \sup_{u\in\calu} \bigl\{ \lvert \delta \si^\vp(u/\sqrt{\delta})_t\rvert^2+\lvert \delta\ov\si^\vp(u/\sqrt{\delta})_t\rvert^2  \bigr\},	\label{eq:prop3-2} \\
				\label{eq:prop4} 
				C_2(\delta)&=\sup_{\om\in\Om}	\sup_{t\in[0,\infty)}\sup_{z\in E}	 \sup_{u\in\calu}  \bigl\{	\lvert \delta \ga^\vp(u/\sqrt{\delta}; t,z)\rvert^2/J(z) \bigr\},\\
				C_3(\delta)&=	\sup_{\om\in\Om}\sup_{t\in[0,\infty)} \sup_{z\in E} \sup_{u\in\calu} \bigl\{ \lvert \delta\Ga^\vp(u/\sqrt{\delta};t,z)\rvert / J(z)\bigr\}. \label{eq:prop7-2} 
			\end{align}
			\item[(iv)] If $\varpi\leq\frac12$ or if $\theta_0=0$,  we have
			\begin{equation}\label{eq:prop1-2-noise} 
				\sup_{\om\in\Om}	\sup_{t\in[0,\infty)}	\bigl\{	\lvert \ga(t,z)\rvert^{2-\varepsilon}+\lvert \ga^c(t,z)\rvert^{2-\varepsilon}  + \lvert \Ga(t,z)\rvert^{1-\varepsilon}\bigr\}\leq J(z),\qquad z\in E,
			\end{equation}
			and   the constants in \eqref{eq:prop0-2}--\eqref{eq:prop7-2} still converge to $0$ after division by $\delta^\varepsilon$.
		\end{enumerate}
	\end{Assumption} 
	
	\settheoremtag{SH$'_1$}
	\begin{Assumption}\label{ass:H1-2}  The following holds in addition to Assumption~\ref{ass:H1-noise} (the assumptions on $\wt w$, $\si^w$, $\ov\si^w$, $\wh\si^w$, $F$ and $G_0$ only apply if $\varpi\leq \frac12+\frac12H$):
		\begin{enumerate}		\item[(i)] We have     $g_0\equiv G_0\equiv0$. 
			The functions $f$ and $F$ are bounded and twice differentiable with  uniformly continuous and bounded derivatives and the processes $x$, $\al$, $v$,  $\si^v$, $\ov\si^v$, $\wt v$, $w$, $\si^w$, $\ov \si^w$ and $\wh \si^w$ and $\wt w$ are uniformly bounded by a constant $K\in(0,\infty)$. 	
			\item[(ii)] We have  $\lvert \ga(t,z)\rvert^2 + \lvert \Ga(t,z)\rvert\leq J(z)$, and if $\varpi\leq\frac12$ or if $\theta_0=0$, we have $\lvert \ga(t,z)\rvert^{2-\varepsilon}+ \lvert \Ga(t,z)\rvert^{1-\varepsilon}\leq J(z)$	for all $t\geq0$ and $z\in E$ (with $\vareps$ and $J$ as in Assumptions~\ref{ass:H0-noise} and \ref{ass:H0-2}).
			\item[(iii)] For any compact subset $\calu$ of $\R$, we have  
			\begin{equation}\label{eq:prop3-alt-2} 
				\sup_{t\in(0,\infty)}\sup_{\delta\in(0,1)} \sup_{u\in\calu} \bigl\{	\delta\lvert\vp(u/\sqrt{\delta})_t\rvert\} <\infty.
			\end{equation}
			With the same $h$ as in Assumption~\ref{ass:H1-noise} and for all $t\geq0$ and $\delta>0$,  we have
			\begin{align}\label{eq:prop2-alt-2} 
				&	\E[\lvert\wt v_{t+\delta}-\wt v_t\rvert^2 ]^{1/2}\leq \delta^H h(\delta),\quad \E[\lvert\wt w_{t+\delta}-\wt w_t\rvert^2 ]^{1/2}\leq \delta^{H_\rho} h(\delta), \\
				\label{eq:prop4-alt-2} 
				&\sup_{u\in\calu}  \E [\lvert \delta\vp(u/\sqrt{\delta})_{t+\delta}-\delta\vp(u/\sqrt{\delta})_t\rvert^2  ]^{1/2} \leq \delta^Hh(\delta),\\
				\begin{split}&\E[\lvert\si^v_{t+\delta}-\si^v_t\rvert^2+\lvert\ov\si^v_{t+\delta}-\ov\si^v_t\rvert^2\\
					&\qquad\qquad 	 +\lvert\si^w_{t+\delta}-\si^w_t\rvert^2+\lvert\ov\si^w_{t+\delta}-\ov\si^w_t\rvert^2+\lvert\wh\si^w_{t+\delta}-\wh\si^w_t\rvert^2 ]^{1/2}\leq h(\delta).\end{split} \label{eq:prop5-alt-2}
			\end{align}  
		\end{enumerate}
	\end{Assumption} 
	
	We can assume $g_0\equiv0$ because $\int_0^t g_0(t-s)( \si^v_sdW_s+\ov \si^vd\ov W_s)=\int_0^t\int_0^s g'_0(r-s)(\si^v_sdW_s+\ov\si^v_sd\ov W_s)dr$ is just a drift that can be absorbed into $\wt v$. A similar argument shows that $G_0$ only contributes a drift that can be absorbed into $\wt w$, if $\varpi\leq \frac12+\frac12 H$.
	
	\settheoremtag{SU$'$}
	\begin{Assumption}\label{ass:U-2} In addition to Assumption~\ref{ass:U-noise}, the following conditions are satisfied:
		\begin{enumerate}
			\item[(i)] There is $K\in(0,\infty)$ such that $\eta_t\geq K^{-1}$ for all $t\geq0$.
			\item[(ii)] We have
			\begin{equation}\label{eq:U-1'} 
				\sup_{j=1,\dots,\lfloor T/(2p_n\Den)\rfloor}\E[\lvert\eta^n_j/\Den^{(2\varpi-1)\wedge0}-\eta_{(2j-2)p_n\Den}\rvert] = o'(1)
			\end{equation}
			\item[(iii)] We have 
			\begin{equation}\label{eq:eta-cont} 
				\sup_{t\geq0} \E[\lvert \eta_{t+\delta}-\eta_t\rvert] = o'(1).
			\end{equation}
			\item[(iv)] There are constants $0<\eta_0^-<\eta_0^+<\infty$ such that 
			\begin{equation}\label{eq:U-3} 
				\lim_{n\to\infty}	\P\Bigl(\eta_0^-\leq \eta^n_j /\Den^{(2\varpi-1)\wedge0}\leq \eta_0^+\text{ for all }  j=1,\ldots,\lfloor T/(2p_n\Den)\rfloor\Bigr) = 1.
			\end{equation}
		\end{enumerate}
	\end{Assumption}

	We first consider the behavior of the test statistic under the null hypothesis and suppose that Assumptions~\ref{ass:H0-2} and \ref{ass:U-2} are satisfied.
	Using the notations
	\begin{align*}
		&\wt u^n_j=\frac{\theta_n}{ \sqrt{\wt\eta^n_j}}, \quad \wt \eta^n_j =\frac{\eta^n_j}{\Den^{(2\varpi-1)\wedge0}},\quad \wt L_j^n(u) = \frac{1}{k_n}\sum_{i=(j-1)p_n+1}^{(j-1)p_n+k_n}e^{iu \Delta_i^ny/\Den^{\varpi\wedge1/2}}, \\
		&\wt c^n_j(u)=- \log \lvert\wt L^n_j(u)\rvert,\quad \wt \cf^n_j(u)=\log \wt c^n_j(u),\quad \Delta^n_j \wt \cf(u)=\wt \cf^n_j(u)-\wt \cf^n_{j-1}(u),\\
		&  \un\Delta^n_j \wt \cf(u,u')= \Delta^n_j \wt \cf(u)- \Delta^n_{j-\dsone} \wt \cf(u'),
	\end{align*}
	we can write the test statistic $\wh T^n$ as
	\begin{equation}\label{eq:T3-2}
		\widehat{T}^n = \frac{\textstyle\sum_{j\in\calj_n}\un\Delta^n_{2j} \wt \cf( \wt{u}^n_j,\wt{u}^n_{j-\dsone/2}) \un\Delta^n_{2j-2} \wt \cf(\wt u^n_{j-1},\wt u^n_{j-1-\dsone/2})}{\textstyle\sqrt{\sum_{j\in\calj_n}\bigl(\un\Delta^n_{2j} \wt \cf( \wt{u}^n_j,\wt{u}^n_{j-\dsone/2}) \un\Delta^n_{2j-2} \wt \cf(\wt u^n_{j-1},\wt u^n_{j-1-\dsone/2})\bigr)^2}}=\frac{  V^n}{\sqrt{  W^n}},
	\end{equation} 
	where 
	\begin{align}
		\label{eq:V} 
		V^n&=\frac{ k_n }{\sqrt{\lvert\calj_n\rvert }}\sum_{j\in\calj_n}\un\Delta^n_{2j} \wt \cf( \wt{u}^n_j,\wt{u}^n_{j-\dsone/2}) \un\Delta^n_{2j-2} \wt \cf(\wt u^n_{j-1},\wt u^n_{j-1-\dsone/2}),\\
		\label{eq:denom} 
		W^n&= \frac{ k_n^2 }{\lvert\calj_n\rvert }\sum_{j\in\calj_n}\bigl(\un\Delta^n_{2j} \wt \cf( \wt{u}^n_j,\wt{u}^n_{j-\dsone/2}) \un\Delta^n_{2j-2} \wt \cf(\wt u^n_{j-1},\wt u^n_{j-1-\dsone/2})\bigr)^2.
	\end{align}
	Since $\log \lvert z\rvert =\Re(\Log z)$, where $\Log$ is the principal branch of the complex logarithm, a 
	first-order expansion gives 
	\begin{equation}\label{eq:firstorder} 
		\Delta^n_{2j}\wt \cf(\wt u^n_j)=\Re\bigl\{\Delta^n_{2j} \wt L(\wt u^n_j)/\Lf(\wt L^n_{2j-1}(\wt u^n_j))\bigr\} + \text{higher-order terms},
	\end{equation} 
	where $\Delta^n_i \wt L(u)= \wt L^n_j(u)-\wt L^n_{j-1}(u)$ and $\Lf(z)=z\log \lvert z\rvert$ for $z\in\mathbb{C}$. As  $\wt L^n_j(u)$   is a local estimator of the conditional characteristic function of $\Delta^n_i y/\Den^{\varpi\wedge 1/2}$, we can decompose
	\begin{equation}\label{eq:split} 
		\wt{L}_j^n(u) = \underbrace{\wt{L}_j^{n,\mathrm{v}}(u)}_{\text{variance}} + \underbrace{\wt{L}_j^{n,\mathrm{b}}(u)}_{\text{bias}}+\underbrace{\wt{L}_j^{n,\mathrm{s}}(u)}_{\text{signal}},
	\end{equation} 
	where
	\begin{align*}
		\wt{L}_j^{n,\mathrm{v}}(u) &= \frac{1}{k_n}\sum_{i=(j-1)p_n+1}^{(j-1)p_n+k_n} \bigl\{e^{i u_n \Delta_i^ny}  -\E [e^{i u_n\Delta_i^ny} \mid\calf_{(i-1)\Den} ]\bigr \},\\
		\wt{L}_j^{n,\mathrm{b}}(u) &=\frac{1}{k_n}\sum_{i=(j-1)p_n+1}^{(j-1)p_n+k_n}\E [e^{i u_n\Delta_i^ny}-e^{i u_n\Delta_{i,i-1}^ny}\mid\calf_{(i-1)\Den} ],\\
		\wt{L}_j^{n,\mathrm{s}}(u) &=\frac{1}{k_n}\sum_{i=(j-1)p_n+1}^{(j-1)p_n+k_n}\E [e^{i u_n\Delta_{i,i-1}^ny}\mid\calf_{(i-1)\Den} ],\qquad u_n=u/\Den^{\varpi\wedge1/2},
	\end{align*}
	and $\Delta^n_{i,j} y=\Delta^n_{i,j} x+\Delta^n_{i,j} \eps$ with
	\begin{align*} 
		\Delta^n_{i,j} x&= \al_{j\Den}\Den + \si_{j\Den}\Delta^n_i W + \int_{(i-1)\Den}^{i\Den}\int_E \ga(j\Den,z) (\mu-\nu)(ds,dz)\\
		&	\quad+\int_{(i-1)\Den}^{i\Den} \int_E\Ga(j\Den,z) \mu(ds,dz)
	\end{align*} 
	and $	\Delta^n_{i,j}\eps= \Den^\varpi\rho_{j\Den} \Delta \chi_i$.
	The next lemma is
	a key technical step in the proofs and  shows that in the limit as $n\to\infty$, only the variance term in $\Delta^n_{2j}\wt L(\wt u^n_j)$ and the signal term in $\wt L^{n}_{2j-1}(\wt u^n_j)$ have to be retained in \eqref{eq:firstorder} for the asymptotic analysis of $V^n$. In particular, the signal part of $\Delta^n_{2j}\wt L(\wt u^n_j)$ has no asymptotic contribution to $V^n$. To understand the intuition behind, consider the noise-free case where the signal part of $\Delta^n_{2j}\wt L(\wt u^n_j)$ is essentially given by an increment of $\exp(-\frac12(\wt u^n_j)^2c_t+\Den \vp(\wt u^n_j/\Den)_t)$. Since this is an It\^o semimartingale in $t$ under the null hypothesis, it has asymptotically uncorrelated increments with variances dominated by $\Delta^n_{2j}\wt L^{\rm v}(\wt u^n_j)$ because of \eqref{eq:rates}.

	\begin{lemma}\label{lem:linearize}
		Under Assumptions~\ref{ass:H0-2} and \ref{ass:U-2}, we have  
		$      V^n -\wt V^n  \stackrel{\P}{\longrightarrow} 0$ as $n\to\infty$, where
		\begin{equation}\label{eq:Vn} \begin{split}
				\wt V^n&=\frac{k_n}{\sqrt{\lvert \calj_n\rvert}}\sum_{j\in\calj_n}\Re\biggl\{\frac{ \Delta^n_{2j-2} \wt L^{\mathrm{v}} (\wt u^n_{j-1})}{\Lf(\call^n_{j-1}(\wt u^n_{j-1}))}-\frac{ \Delta^n_{2j-2-\dsone} \wt L^{\mathrm{v}} (\wt u^n_{j-1-\dsone/2})}{\Lf(\call^n_{j-1-\dsone/2}(\wt u^n_{j-1-\dsone/2}))}\biggr\}\\
				&\qquad\qquad\qquad\qquad\qquad\qquad\times\Re\biggl\{\frac{ \Delta^n_{2j} \wt L^{\mathrm{v}} (\wt u^n_{j})}{\Lf(\call^n_j(\wt u^n_j))}-\frac{ \Delta^n_{2j-\dsone} \wt L^{\mathrm{v}} (\wt u^n_{j-\dsone/2})}{\Lf(\call^n_{j-\dsone/2}(\wt u^n_{j-\dsone/2}))}\biggr\}
			\end{split}
		\end{equation}
		and $\call^n_j(u)=\exp(-\frac12u^2c_{((2j-2)p_n-m-1)\Den}\Den^{(1-2\varpi)_+})\Psi(u\Den^{(\varpi-1/2)_+}\rho_{((2j-2)p_n-m-1)\Den})$.
	\end{lemma}

	The asymptotic distribution of $\wt V^n$ can now be established by an application of the martingale central limit theorem. Recall that by assumption, $\theta_0=0$ if $\varpi\leq \frac12$, while for $\varpi>\frac12$, both $\theta_0=0$ and $\theta_0>0$ are allowed.
	
	\begin{lemma}\label{lem:CLT}
		Under Assumptions~\ref{ass:H0-2} and \ref{ass:U-2}, we have
		\begin{equation}\label{eq:lim} 
			\wt	V^n\limst	Q^{1/2} Z, 
		\end{equation}
		where $Z$ is a standard normal random variable that is defined on a very good extension of the original probability space and independent of $\calf_\infty$. The asymptotic $\calf_\infty$-conditional variance $Q$ is given by
		\begin{equation}\label{eq:Q-formula} 
			Q=\frac{1}{T} \int_{I_T} (q_t+q_{t-1})^2 dt ,
		\end{equation}
		where $I_T=\bigcup_{k=1}^{T/2} [(2k-1),2k]$,
		\begin{equation}\label{eq:Q-limit} 
			q_t=\begin{dcases} 
				\frac{16\eta_t^2}{\theta_0^4c_t^2}\sinh^2(\tfrac12 \theta_0^2c_t/\eta_t)
				&\text{if } \theta_0>0,\\ 
				\frac{4q_1(c_t\bone_{\{\varpi\geq\frac12\}},\rho_t\bone_{\{\varpi\leq \frac12\}})}{q_2(c_t\bone_{\{\varpi\geq\frac12\}},\rho_t\bone_{\{\varpi\leq \frac12\}})}
				&\text{if } \theta_0=0\end{dcases} 
		\end{equation}
		and 
		\begin{align*}
			q_1(c,\rho)&= c^2+2c\rho^2\E[(\Delta\chi_1)^2]+\rho^4\biggl(\frac12\Var((\Delta\chi_1)^2)+\sum_{r=1}^{m+1}\Cov((\Delta\chi_{r+1})^2,(\Delta\chi_1)^2)\biggr),\\
			q_2(c,\rho)&= (c+\rho^2\E[(\Delta\chi_1)^2])^2.
		\end{align*}
	\end{lemma}

	Correspondingly, the denominator $W^n$ in \eqref{eq:T3-2} is a consistent estimator of $Q$.
	
	\begin{lemma}\label{lem:denom}
		Under Assumptions~\ref{ass:H0-2} and \ref{ass:U-2}, we have $W^n\stackrel{\P}{\longrightarrow}Q$ as $n\to\infty$.
	\end{lemma}
	
	\begin{proof}[Proof of Theorem~\ref{thm:H0} under \ref{ass:H0}]
		By localization, we can suppose that Assumptions~\ref{ass:H0-2} and \ref{ass:U-2} are in force. Since $Q$ is $\calf_\infty$-measurable,   \eqref{eq:T-H0} immediately follows from Lemmas~\ref{lem:linearize}--\ref{lem:denom}, the relation  $\wh T^n=V^n/\sqrt{W^n}$ and   property (2.2.5) in \citet{JP12}.
	\end{proof}	
	
	For the proof of Theorem~\ref{thm:H0} under the alternative hypothesis, we can suppose that Assumptions~\ref{ass:H1-2} and \ref{ass:U-2} hold true. A key difference between a semimartingale  and a rough volatility process is that the latter has asymptotically correlated increments. So, instead of $V^n$ and $W^n$, the correctly normalized quantities are now 
	\begin{equation}\label{eq:V2} 
		V^{n,\rm{alt}}=\frac{1}{\pi_n^2\lvert \calj_n\rvert}\sum_{j\in\calj_n}\un\Delta^n_{2j}\wt{\cf} (u^n_{j},u^n_{j-\dsone/2}) \un\Delta^n_{2j-2}\wt{\cf} (u^n_{j-1},u^n_{j-1-\dsone/2})
	\end{equation}
	and 
	\begin{equation}\label{eq:W2} 
		W^{n,\rm{alt}}=\frac{1}{\pi_n^4\lvert \calj_n\rvert}\sum_{j\in\calj_n}\bigl(\un\Delta^n_{2j}\wt{\cf} (u^n_{j},u^n_{j-\dsone/2}) \un\Delta^n_{2j-2}\wt{\cf} (u^n_{j-1},u^n_{j-1-\dsone/2})\bigr)^2,
	\end{equation}
	where    $\pi_n=(p_n\Den)^H\Den^{(1-2\varpi)_+} + (p_n\Den)^{H_\rho}\Den^{(2\varpi-1)_+}$ and $H_\rho=0$ by convention if we have $\varpi>\frac12+\frac12 H$.
	\begin{lemma}\label{lem:V-alt}
		Recall from \eqref{eq:rates} that $\kappa$ is the asymptotic ratio of $p_n/k_n$. Furthermore, let $\La=\lim_{n\to\infty} (p_n\Den)^H\Den^{(1-2\varpi)_+}/\pi_n$. (By our choice of $p_n$ in \eqref{eq:rates}, we either have $\La=0$, $\La=1$ or $\La=\frac12$. The last case happens precisely when $H=H_\rho$ and $\varpi=\frac12$.)
		Under Assumptions~\ref{ass:H1-2} and \ref{ass:U-2}, there are finite constants $C_{\kappa,H}$ and $\ov C_{\kappa,H}$ that only depend on $\kappa$ and $H$ such that 
		\begin{equation}\label{eq:VW}\begin{split}
				V^{n,\rm{alt}}	 &\stackrel{\P}{\longrightarrow}\frac{1}{T} (C_{\kappa,H}\bone_{\{\La\in\{\frac12,1\}\}}+C_{\kappa,H_\rho}\bone_{\{\La=0\}})\int_{I_T} (A(t)+A(t-1)) dt, \\
				W^{n,\rm{alt}}	& \stackrel{\P}{\longrightarrow} \frac{2}{T}[1+2(C_{\kappa,H}\bone_{\{\La\in\{\frac12,1\}\}}+C_{\kappa,H_\rho}\bone_{\{\La=0\}})^2] \int_{I_T}A(t)A(t-1)dt\\
				&\qquad+\frac{1}{T} [\ov C_{\kappa,H}\bone_{\{\La\in\{\frac12,1\}\}}+\ov C_{\kappa,H_\rho}\bone_{\{\La=0\}}]\int_{I_T} ( B(t)+B(t-1)) dt,
			\end{split}
		\end{equation}
		where
		\begin{align*}
			A(t)	&=
			\varsigma_t^{-4} \Bigl[(  f'(v_t)\si^v_t\bone_{\{\La\in\{\frac12,1\}\}}+2\E[(\Delta\chi_1)^2]\rho_t F'(w_t)\si^w_t\bone_{\{\La\in\{0,\frac12\}\}})^2\\
			&\qquad\qquad    +(  f'(v_t)\ov\si^v_t\bone_{\{\La\in\{\frac12,1\}\}}+2\E[(\Delta\chi_1)^2]\rho_t F'(w_t)\ov\si^w_t\bone_{\{\La\in\{0,\frac12\}\}})^2\\
			&\qquad\qquad+(2\E[(\Delta\chi_1)^2]\rho_tF'(w_t)\wh\si^w_t\bone_{\{\La\in\{0,\frac12\}\}})^2\Bigr]
		\end{align*}
		and
		\begin{align*}  
			B(t)	&=
			\varsigma_t^{-4}\bigl[(  f'(v_t)\si^v_t\bone_{\{\La\in\{\frac12,1\}\}}+2\E[(\Delta\chi_1)^2]\rho_t F'(w_t)\si^w_t\bone_{\{\La\in\{0,\frac12\}\}})^2\\
			&	\qquad\qquad   +(  f'(v_t)\ov\si^v_t\bone_{\{\La\in\{\frac12,1\}\}}+2\E[(\Delta\chi_1)^2]\rho_t F'(w_t)\ov\si^w_t\bone_{\{\La\in\{0,\frac12\}\}})^2\\
			&\qquad\qquad+(2\E[(\Delta\chi_1)^2]\rho_tF'(w_t)\wh\si^w_t\bone_{\{\La\in\{0,\frac12\}\}})^2\bigr]^2 
		\end{align*}
		and $\varsigma_t^2=c_t\bone_{\{\varpi\geq\frac12\}}+\rho^2_t\E[(\Delta\chi_1)^2]\bone_{\{\varpi\leq \frac12\}}$. Moreover,  $C_{\kappa,H}<0$ for all $\kappa\in[1,\infty)$ and $H\in(0,\frac12)$.
	\end{lemma}

	\begin{proof}[Proof of Theorem~\ref{thm:H0} under \ref{ass:H1}]
		By localization, we may assume the stronger Assumptions~\ref{ass:H1-2} and \ref{ass:U-2}.
		Since $\wh T^n=\sqrt{\lvert\calj_n\rvert} {V^{n,\rm{alt}}}/{\sqrt{W^{n,\rm{alt}}}}$, we obtain \eqref{eq:T-alt} from Lemma~\ref{lem:V-alt} and our assumptions (in particular, $A(t)$ is strictly negative and $B(t)$ is strictly positive).
	\end{proof}
	
	\subsection{Proofs for Examples~\ref{ex:1}, \ref{ex:2} and \ref{ex:3}}
	\begin{lemma}\label{lem:ex1}
		In the setting of Example~\ref{ex:1}, $\vp(u)_t$ from \eqref{eq:vp-ex} satisfies \eqref{eq:prop2}, \eqref{eq:prop3} and \eqref{eq:prop4-2}. 
	\end{lemma}
	\begin{proof}
		We only prove \eqref{eq:prop3}; the other two properties can be shown similarly. 
		We can   assume that $K$ and $\la$ are driven by $W^\vp$ and $\ov W^\vp$ and
		by symmetry, it suffices to analyze $\si^\vp$. Using It\^o's formula, we have that
		\[
		\si^\vp(u)_t	=\int_\R \Bigl(\la_t (e^{iuK_t z}-1)iuz\si^K_t+ (e^{iu K_{t}z}-1-iuK_{t}z)\si^\la_t\Bigr)F(dz),
		\]
		where $\si^K$ and $\si^\la$ are the coefficients of $K$ and $\la$ with respect to $W$. Therefore
		\begin{equation}\label{eq:sup}\begin{split}
				\sup_{t\in[0,T]}\sup_{u\in\calu} \lvert\delta\si^\vp(u/\sqrt{\delta})_t\rvert	&\leq\int_\R \delta \sup_{t\in[0,T]}\sup_{u\in\calu}\Bigl\lvert\la_t (e^{iuK_t z/\sqrt{\delta}}-1)iuz\si^K_t/\sqrt{\delta}\\
				& \quad+ (e^{iu K_{t}z/\sqrt{\delta}}-1-iuK_{t}z/\sqrt{\delta})\si^\la_t\Bigr\rvert F(dz).
		\end{split}\end{equation}
		As $\delta\to0$, the integrand in the last line converges to $0$ pointwise in $z$ for all $\om$. Moreover, it is bounded by
		\[  \biggl(\sup_{u\in\calu} u^2\biggr)\biggl( \sup_{t\in[0,T]} \lvert\la_tK_t\si_t^K\rvert +\frac12 \sup_{t\in[0,T]}\lvert K_t^2\si^\la_t\rvert\biggr) z^2,\]
		which does not depend on $\delta$ and is integrable with respect to $F$ almost surely. Therefore, the dominated convergence theorem shows that \eqref{eq:sup} tends to $0$ as $\delta\to0$.
	\end{proof}
	
	\begin{lemma}\label{lem:ex2}
		In the setup of Example~\ref{ex:2}, the process $\vp(u)_t$ from \eqref{eq:vp-ex} satisfies \eqref{eq:prop3-alt} and \eqref{eq:prop4-alt}.
	\end{lemma}
	\begin{proof}
		Property  \eqref{eq:prop3-alt} follows immediately from \eqref{eq:vp-ex} and  dominated convergence. For \eqref{eq:prop4-alt}, we can assume $\tau_1=\infty$ by localization and that $\lvert \la_t\rvert, \lvert K_t\rvert\leq C$. Then, because
		\begin{align*}
			\delta\vp(u/\sqrt{\delta})_t-\delta\vp(u/\sqrt{\delta})_s&=\int_\R \delta\Bigl(e^{iuK_sz/\sqrt{\delta}}(e^{iu(K_t-K_s)z/\sqrt{\delta}}- 1- iu(K_t-K_s)z/\sqrt{\delta})\\
			&\quad\qquad\qquad+(e^{iuK_sz/\sqrt{\delta}}-1) iu(K_t-K_s)z/\sqrt{\delta}\Bigr)\la_t F(dz)\\
			&\quad+\int_\R \delta\Bigl( e^{iuK_sz/\sqrt{\delta}}- 1 - iuK_sz/\sqrt{\delta}\Bigr)(\la_t-\la_s) F(dz),
		\end{align*}
		we have
		\begin{align*}
			&\delta^{-H}\sup_{s,t\geq0,\lvert t-s\rvert\leq \delta}\sup_{u\in\calu} \E\Bigl[\bigl\lvert\delta\vp(u/\sqrt{\delta})_t-\delta\vp(u/\sqrt{\delta})_s\bigr\rvert^2\Bigr]\\
			&~\leq \int_\R \delta^{1-H}\sup_{s,t\geq0,\lvert t-s\rvert\leq \delta}\sup_{u\in\calu}\E\Bigl[\Bigl\lvert\Bigl( e^{iuK_sz/\sqrt{\delta}}(e^{iu(K_t-K_s)z/\sqrt{\delta}}- 1- iu(K_t-K_s)z/\sqrt{\delta})\\
			&~\qquad\qquad\qquad\qquad\qquad\qquad\qquad\qquad \qquad +(e^{iuK_sz/\sqrt{\delta}}-1) iu(K_t-K_s)z/\sqrt{\delta}\Bigr)\la_t\\
			&~\quad+ ( e^{iuK_sz/\sqrt{\delta}}- 1 - iuK_sz/\sqrt{\delta} )(\la_t-\la_s)\Bigr\rvert^2\Bigr]^{1/2} F(dz).
		\end{align*}
		Clearly,  the integrand tends to $0$ pointwise in $z$ as $\delta\to0$. Since $\lvert e^{iux}-1\rvert \leq \lvert ux\rvert$, $\lvert e^{iux}-1-iux\rvert\leq \frac12 u^2 x^2$ 
		and $\lvert K_t-K_s\rvert^2\leq 2C\lvert K_t-K_s\rvert$, we can further bound it by
		\begin{align*}
			&\delta^{-H}\sup_{s,t\geq0,\lvert t-s\rvert\leq \delta}\sup_{u\in\calu}\E\biggl[\biggl(\frac12u^2\lvert K_t-K_s\rvert^2 \lvert \la_t\rvert z^2 + u^2\lvert K_s\la_t\rvert \lvert K_t-K_s\rvert z^2\\
			&\qquad\qquad\qquad\qquad\qquad\qquad\qquad\qquad \qquad \qquad\qquad+\frac12u^2K_s^2z^2\lvert \la_t-\la_s\rvert\biggr)^2\biggr]^{1/2}	\\
			&\quad\leq \delta^{-H}\sup_{s,t\geq0,\lvert t-s\rvert\leq \delta}\biggl(2 C^2\E[\lvert K_t-K_s\rvert^2]^{1/2} +\frac12 C^2\E[\lvert \la_t-\la_s\rvert^2]^{1/2} \biggr)\biggl(\sup_{u\in\calu} u^2\biggr)z^2\\
			&\quad\leq \frac52 C_1C^2\biggl(\sup_{u\in\calu} u^2\biggr)z^2,
		\end{align*}
		which is $F$-integrable. The claim now follows by dominated convergence.
	\end{proof}
	
	\begin{lemma}\label{lem:ex3}
		If $\eta^n_j$ is chosen according to \eqref{eq:ups} (with $\wh c^n_j$ from \eqref{eq:bp}   computed using the observed prices $y$ in the noisy case), then Assumption~\ref{ass:U} is satisfied under both \ref{ass:H0} and \ref{ass:H1}, while Assumption~\ref{ass:U-noise} is  satisfied under both \ref{ass:H0-noise} and \ref{ass:H1-noise}.
	\end{lemma}
	\begin{proof} It suffices to show that Assumption~\ref{ass:U-noise} is satisfied, as Assumption~\ref{ass:U} follows by considering the special case $\varpi=1$ and $\rho\equiv0$. To simplify notation, we further restrict ourselves to the  case $\ell_1=\ell_2=2$ (i.e., $\eta^n_j=\wh c^n_{2j-2}$). The general case is similar but more cumbersome to write. With this choice, $\eta^n_j$ is clearly $\calf_{(2j-2)p_n\Den}$-measurable. In order to show \eqref{eq:U-1} and \eqref{eq:U-2}, we can work under the strengthened conditions of Assumption~\ref{ass:H0-2} and \ref{ass:H1-2}  by localization. In particular, we can assume $\tau_m=\infty$ for all $m\in\N$. 
		Introducing the notation $\ov c^n_j = (\pi/2k_n\Den)\sum_{i=(j-1)p_n+1}^{(j-1)p_n+k_n} \lvert \Delta^n_i x^c+\Delta^n_i \eps\rvert \lvert \Delta^n_{i-1} x^c+\Delta^n_{i-1} \eps\rvert$, where $x^c=\int_0^t\si_sdW_s$, we have
		\begin{align*}
			\wh c^n_j-\ov c^{n}_j	&=\frac{\pi}{2k_n\Den}\sum_{i=(j-1)p_n+1}^{(j-1)p_n+k_n}\Bigl[( \lvert\Delta^n_i y\rvert-\lvert\Delta^n_i x^c+\Delta^n_i \eps\rvert)\lvert\Delta^n_{i-1} y\rvert\\
			&\qquad\qquad\qquad\qquad\qquad + \lvert\Delta^n_i x^c+\Delta^n_i \eps\rvert(\lvert\Delta^n_{i-1} y\rvert -\lvert \Delta^n_{i-1}x^c+\Delta^n_{i-1} \eps\rvert)\Bigr].
		\end{align*}
		Therefore,
		\begin{equation}\label{eq:aux13}\begin{split} 
				&\E[\lvert \wh c^n_j-\ov c^{n}_j\rvert/\Den^{(2\varpi-1)\wedge0}]\\	
				&\quad\leq \frac{\pi}{2k_n\Den^{2\varpi\wedge1}}\sum_{i=(j-1)p_n+1}^{(j-1)p_n+k_n}\E\Bigl[\lvert\Delta^n_{i-1} y\rvert\E^n_{i-1}\bigl[\bigl\lvert \lvert\Delta^n_i y\rvert-\lvert\Delta^n_i x^c+\Delta^n_i \eps\rvert\bigr\rvert\bigr]  \\
				&\qquad+\E^n_{i-1}\bigl[ \lvert\Delta^n_i x^c+\Delta^n_i \eps\rvert\bigr]\bigl\lvert\lvert\Delta^n_{i-1} y\rvert -\lvert \Delta^n_{i-1}x^c+\Delta^n_{i-1} \eps\rvert\bigr\rvert\Bigr]. 
		\end{split}	\end{equation}
		As $\E^n_{i-1}\bigl[\bigl\lvert \lvert\Delta^n_i y\rvert-\lvert\Delta^n_i x^c+\Delta^n_i \eps\rvert\bigr\rvert\bigr]\leq \E[\lvert \Delta^n_i y-\Delta^n_i x^c-\Delta^n_i \eps\rvert]\leq (K+\int_E J(z)\la(dz))\Den+\sqrt{\Den}h_3(\Den)$ by Lemma~\ref{lem:o}, we   obtain   $\E[\lvert \wh c^n_j-\ov c^{n}_j\rvert/\Den^{(2\varpi-1)\wedge0}]\leq Ch_3(\Den)/\Den^{(2\varpi-1)\wedge0}$, which goes to $0$ uniformly in $j$. With  jumps and drift removed, it is now easy to show that we can    replace $\Delta^n_{i-\ell} x^c$ by $\si_{(i-2)\Den}\Delta^n_{i-\ell} W$ and $\Delta^n_{i-\ell}\eps$ by $\Den^{\varpi}\rho_{((i-2)\Den}\Delta^n_{i-\ell} \chi$ ($\ell=0,1$) in $\ov c^{n}_j/\Den^{(2\varpi-1)\wedge0}$, incurring only an $O_p(\sqrt{\Den})$- or $O_p(\Den^H)$-error (depending on whether we have \ref{ass:H0-noise} or \ref{ass:H1-noise}) that is uniform in $j$. Using a martingale argument and writing $K(\chi)=\frac\pi2 \E[\lvert\Delta \chi_2\Delta\chi_1\rvert]$, one can   prove that the resulting expression approaches, at a rate of $\sqrt{k_n}$, the term $k_n^{-1}\sum_{i=(j-1)p_n+1}^{(j-1)p_n+k_n} (\si^2_{(i-2)\Den}\Den^{(1-2\varpi)\vee 0}+K(\chi)\rho^2_{(i-2)\Den}\Den^{(1-2\varpi)\wedge0})$, which approximates $\si^2_{(j-1)p_n\Den+}\bone_{\{\varpi\geq\frac12\}}+K(\chi)\rho^2_{(j-1)p_n\Den}\bone_{\{\varpi\leq \frac12\}}$ with a uniform $O_p(\sqrt{p_n\Den})$- or $O_p((p_n\Den)^H )$-error. This shows \eqref{eq:U-1} and its extension mentioned in Assumption~\ref{ass:U-noise}.
		
		Next, we consider  \eqref{eq:U-2} and observe that $\wh c^n_j =  \wt c^{n,1}_j+\wt c^{n,2}_j$, where
		\begin{align*}
			\wt c^{n,1}_j &= \frac{\pi}{2k_n\Den} \sum_{i=(j-1)p_n+1}^{(j-1)p_n+k_n}  \E^n_{i-2}\bigl[\lvert \Delta^n_i y  \Delta^n_{i-1} y\rvert \bigr], \\
			\wt c^{n,2}_j &=\frac{\pi}{2k_n\Den} \sum_{i=(j-1)p_n+1}^{(j-1)p_n+k_n} \Bigl\{\lvert \Delta^n_i y  \Delta^n_{i-1} y\rvert - \E^n_{i-2}\bigl[\lvert \Delta^n_i y \Delta^n_{i-1} y\rvert \bigr]\Bigr\}.
		\end{align*}
		In $\wt c^{n,2}_j$, the $i$th term is $\calf_{i\Den}$-measurable and has a zero expectation conditionally on $\calf_{(i-2)\Den}$. So if we split the sum over $i$ into two, one taking only even and one only taken odd values of $i$, both will be martingale sums. Taking $p$th moments for $p\geq2$ and using the BDG inequality and (\ref{eq:L2}) or (\ref{eq:L2-2}), we obtain
		\begin{align*}
			\E[\lvert\wt c^{n,2}_j/\Den^{(2\varpi-1)\wedge0}\rvert^p]	&\leq \frac{C}{k_n^p\Den^{(2\varpi \wedge 1)p}}\E\Biggl[\Biggl(\sum_{i=(j-1)p_n+1}^{(j-1)p_n+k_n} \E^n_{i-2}[\lvert \Delta^n_i y  \Delta^n_{i-1} y\rvert^2]\Biggr)^{p/2}\Biggr]\\
			&=\frac{C}{k_n^p\Den^{(2\varpi \wedge 1)p}}\E\Biggl[\Biggl(\sum_{i=(j-1)p_n+1}^{(j-1)p_n+k_n} \E^n_{i-2}\Bigl[\E^n_{i-1}[\lvert \Delta^n_iy\rvert^2]\lvert  \Delta^n_{i-1} y\rvert^2\Bigr]\Biggr)^{p/2}\Biggr] \\
			&\leq \frac{C }{k_n^p\Den^{(\varpi \wedge {1/2})p}} \E\Biggl[\Biggl(\sum_{i=(j-1)p_n+1}^{(j-1)p_n+k_n}\E^n_{i-2} [\lvert  \Delta^n_{i-1} y\rvert^2]\Biggr)^{p/2}\Biggr] \leq \frac{C}{k_n^{p/2}}.
		\end{align*}
		Markov's inequality with $p=4$ implies that for any sequence $a_n\to0$,
		\begin{equation}\label{eq:aux12}
			\sum_{j=1}^{\lfloor T/(2p_n\Den)\rfloor} \P(\lvert \eta^n_j-\wt c^{n,1}_{(2j-2)p_n\Den} \rvert/\Den^{(2\varpi-1)\wedge0} \geq a_n)\leq \frac{Cn}{a_n^4k_n^2p_n}.
		\end{equation}
		By \eqref{eq:rates}, the last line tends to $0$ if $a_n\to0$ slowly enough. In this case, it follows that
		\[ \lim_{n\to\infty} \P(\lvert  \eta^n_j-\wt c^{n,1}_{(2j-2)p_n\Den}  \rvert /\Den^{(2\varpi-1)\wedge0}\leq a_n  \text{ for all } j=1,\dots,\lfloor T/(2p_n\Den)\rfloor ) =1. \]
		This implies \eqref{eq:U-2} (with $\eta^n_j/\Den^{(2\varpi-1)\wedge0}$ instead of $\eta^n_j$) provided that   $\wt c^{n,1}_{(2j-2)p_n\Den}/\Den^{(2\varpi-1)\wedge0}$ is  bounded from above and below, uniformly in $\om$ and $j$. Thanks to (\ref{eq:L2}) or (\ref{eq:L2-2}), we have $\E^n_{i-1}[\lvert \Delta^n_i y\rvert^2]\leq C\Den^{1\wedge 2\varpi}$, which readily gives the upper bound. For the lower bound, observe that for a general random variable $X$, we have $\E[\lvert X\rvert^{3/2}]=\E[\lvert X\rvert^{1/2} \lvert X\rvert]\leq \E[\lvert X\rvert]^{1/2}\E[\lvert X\rvert^2]^{1/2}$ by the Cauchy--Schwarz inequality, which implies 
		\begin{equation}\label{eq:PZ} 
			\E[\lvert X\rvert]\geq\frac{\E[\lvert X\rvert^{3/2}]^2}{\E[\lvert X\rvert^2]}.
		\end{equation}
		We want to apply this to $\E^n_{i-1}[\lvert \Delta^n_i y\rvert]$ in $\E^n_{i-2}\bigl[\lvert \Delta^n_i y  \Delta^n_{i-1} y\rvert \bigr]=\E^n_{i-2}\bigl[\lvert   \Delta^n_{i-1} y\rvert \E^n_{i-1}[\lvert \Delta^n_i y\rvert]\bigr]$. The denominator satisfies the bound  $\E^n_{i-1}[\lvert \Delta^n_i y\rvert^2]\leq C\Den^{1\wedge 2\varpi}$. For the numerator, we distinguish whether $\varpi\geq\frac12$ or $\varpi<\frac12$. In the first case,  Jensen's inequality applied to the function $x\mapsto \lvert x\rvert^{3/2}$ for the conditional expectation $\E[\cdot\mid\calf_\infty]$ yields $\E^n_{i-1}[\lvert \Delta^n_i y\rvert^{3/2}]=\E^n_{i-1}[\E[\lvert  \Delta^n_i x+ \Delta^n_i\eps\rvert^{3/2}\mid\calf_\infty]] \geq \E^n_{i-1}[\lvert \Delta^n_i x\rvert^{3/2}]$. Recalling
		(\ref{eq:x'}), we further have
		$\E^n_{i-1}[\lvert \Delta^n_i x\rvert^{3/2}]^{2/3}\geq \E^n_{i-1}[\lvert \Delta^n_i x'\rvert^{3/2}]^{2/3}-\E^n_{i-1}[\lvert \Delta^n_i x''\rvert^{3/2}]^{2/3}\geq \E^n_{i-1}[\lvert\Delta^n_i x'\rvert^{3/2}]^{2/3}-(K+\int_E J(z)\la(dz))\Den$ by the reverse triangle inequality. Moreover,  combining Doob's martingale inequality with the BDG inequality, we get
		\begin{align*}
			\E^n_{i-1}[\lvert \Delta^n_i x'\rvert^{3/2}]	&\geq C\E^n_{i-1}\biggl[\sup_{s\in[(i-1)\Den,i\Den]}\lvert x'_s-x'_{(i-1)\Den}\rvert^{3/2}\biggr]\\
			&\geq C\E\Biggl[\Biggl(\int_{(i-1)\Den}^{i\Den} \si^2_s ds + \iint_{(i-1)\Den}^{i\Den} (\ga(s,z)+\Ga(s,z))^2\mu(ds,dz)\Biggr)^{3/4}\Biggr] \\
			&\geq CK^{-3/4}\Den^{3/4}.
		\end{align*}
		Therefore, if $n$ is sufficiently large, $\E^n_{i-1}[\lvert \Delta^n_i x\rvert^{3/2}]\geq CK^{-3/4}\Den^{3/4}$ and consequently, by \eqref{eq:PZ}, $\E^n_{i-1}[\lvert \Delta^n_i y\rvert]\geq C\Den^{1/2}$, where $C$ does not depend on $i$ or $\om$. If $\varpi<\frac12$, we instead bound $\E^n_{i-1}[\lvert \Delta^n_i y\rvert]\geq \E^n_{i-1}[\lvert \Delta^n_i \eps\rvert]-\E^n_{i-1}[\lvert \Delta^n_i x\rvert]\geq \Den^{\varpi}\rho_{(i-1)\Den}\E[\lvert \Delta \chi_i\rvert] - C\Den^{\varpi+1/2}-C\Den^{1/2}\geq C\Den^{\varpi}$. So in both cases, it follows that  $\wt c^{n,1}_j /\Den^{(2\varpi-1)\wedge0}\geq C$ where $C$ is independent of $\om$ and $j$, which concludes the proof of \eqref{eq:U-2}.
	\end{proof}

\section{Proof of Auxiliary Results}\label{sec:proof-app}

Throughout this section, $C$ denotes a deterministic constant in $(0,\infty)$, whose value does not depend on any important parameters and may change from line to line.   Given random variables $X_n$ and $Y_n$ and a deterministic sequence $a_n$, we write $X_n =Y_n +O(a_n)$ if $\lvert X_n-Y_n\rvert\leq Ca_n$, where $C$ neither depends on $n$ nor   $\om$. If $X_n$ and $Y_n$ further depend on indices such as $i$ or $j$, then $C$ does not depend on $i$ or $j$, either. The notations $X_n=Y_n+o(a_n)$ and $X_n=Y_n+o'(a_n)$ are used analogously. We also use the abbreviations $\int_a^b\int_E=\iint_a^b$ and $\E^n_i = \E[\cdot \mid \calf_{i\Den}]$.

\subsection{Preliminary Estimates}

\begin{lemma}\label{lem:prelim} Let   $\phi(u)_t=\int_E e^{iu\ga(s,z)}(e^{iu\Ga(s,z)}-1)\la(dz)$ and  $\calu\subseteq\R$ be a compact set.
	\begin{enumerate}
		\item[(i)] Under Assumption~\ref{ass:H0-2}, there is  $C\in(0,\infty)$ such that for all $i$,
		\begin{equation}\label{eq:L2} \begin{split}
				&	\E^n_{i-1}[(\Delta^n_i x)^2] + \E^n_{i-1}[(\Delta^n_i c)^2]+ \E^n_{i-1}[(\Delta^n_i \rho)^2\bone_{\{\varpi\leq \frac34\}}]\leq C\Den,  \\ 	&\sup_{u\in\calu}\Den^2\E^n_{i-1}[(\Delta^n_i \vp(u/\sqrt{\Den}))^2]\leq C\Den h_1(\Den),\quad\E^n_{i-1}[(\Delta^n_i \eps)^2] \leq C\Den^{ 2\varpi}.
		\end{split}\end{equation}
		In the last line, $h_1$ is a function that satisfies $h_1(\Den)=o'(1)$.
		\item[(ii)]
		Under Assumption~\ref{ass:H1-2}, there is $C\in(0,\infty)$ such that for all $i$,
		\begin{equation}\label{eq:L2-2} \begin{aligned}
				\E^n_{i-1}[(\Delta^n_i x)^2]& \leq C\Den,& \E^n_{i-1}[(\Delta^n_i c)^2]&\leq C\Den^{2H},\\
				\E^n_{i-1}[(\Delta^n_i \rho)^2]&\leq C\Den^{2H_\rho},& \E^n_{i-1}[(\Delta^n_i \eps)^2]&\leq   C\Den^{ 2\varpi}.
			\end{aligned}
		\end{equation}
		\item[(iii)] Under both Assumptions~\ref{ass:H0-2} and \ref{ass:H1-2}, 
		\begin{equation}\label{eq:vp-conv} \begin{split}
				\vp(n)&=	\sup_{\om\in\Om}\sup_{t\in[0,\infty)}\sup_{u\in\calu} \Den\lvert\vp(u/\sqrt{\Den})_t\rvert =o'(1),\\
				\phi(n)&=	\sup_{\om\in\Om}\sup_{t\in[0,\infty)}\sup_{u\in\calu} \Den^{1/2}\lvert\phi(u/\sqrt{\Den})_t\rvert =o'(1).
			\end{split}
		\end{equation}
		\item[(iv)] Let $\Psi$ be the characteristic function of $\Delta\chi_1=\chi_1-\chi_0$. Then, for all $u,u'\in\R$,
		\begin{equation}\label{eq:psi} \begin{split}
				\lvert\Psi(u)-1\rvert&\leq \tfrac12 u^2 \E[( \Delta\chi_1)^2],\\
				\lvert \Psi(u)-\Psi(u')\rvert&\leq (\lvert u(u-u')\rvert+\tfrac12(u-u')^2)\E[(\Delta\chi_1)^2].
			\end{split}
		\end{equation}
	\end{enumerate}
\end{lemma}
\begin{proof} Writing
	\begin{equation}\label{eq:x'} 
		x''_t = \int_0^t \al_s ds+\iint_0^t \Ga(s,z)\la(dz)ds,\qquad x'_t=x_t-x''_t,
	\end{equation}
	we have $\E^n_{i-1}[(\Delta^n_i x)^2]^{1/2}\leq \E^n_{i-1}[(\Delta^n_i x')^2]^{1/2}+\E^n_{i-1}[(\Delta^n_i x'')^2]^{1/2}$. Since $\lvert \al_t\rvert\leq K$ and $\lvert \Ga(t,z)\rvert\leq J(z)$, we have $\E^n_{i-1}[(\Delta^n_i x'')^2]^{1/2}\leq (K+\int_E J(z)\la(dz))\Den\leq 2K\Den$. Similarly, because $\si^2_t\leq C$ and $\lvert \ga(t,z)\rvert^2\leq J(z)$, It\^o's isometry gives $\E^n_{i-1}[(\Delta^n_i x')^2]\leq  (C+\int_E J(z)\la(dz))\Den\leq C\Den$, showing the bound on $\Delta^n_i x$ in both \eqref{eq:L2} and \eqref{eq:L2-2}. And because $\rho_t\leq K$, we have $\E^n_{i-1}[(\Delta^n_i \eps)^2] \leq 4K^2\Den^{2\varpi}$, which is the last inequality in both \eqref{eq:L2} and \eqref{eq:L2-2}. Under Assumption~\ref{ass:H0-2}, the processes $c$ and $\rho$ (if $\varpi\leq\frac34$) are   It\^o semimartingales with bounded coefficients, so the bounds on $\Delta^n_i c$ and $\Delta^n_i\rho$ in \eqref{eq:L2} follow by applying the bound for $\Delta^n_i x$ to $x\in\{c,\rho\}$. 
	A similar argument, combined with (\ref{eq:prop2-2}) and (\ref{eq:prop3-2})--(\ref{eq:prop7-2}), yields the bound on $\Delta^n_i \vp(u/\sqrt{\Den})$. Under Assumption~\ref{ass:H1-2}, the derivative $f'$ is bounded by some constant $C$, so the mean-value theorem implies that  $\E^n_{i-1}[(\Delta^n_i c)^2] \leq C^2\E^n_{i-1}[(\Delta^n_i v)^2]\leq 2C^2\E^n_{i-1}[(\Delta^n_i \wt v)^2 + (\Delta^n_i  \ov v)^2]$, where $\ov v_t=v_t-v_0-\wt v_t$. By (\ref{eq:prop2-alt-2}), we have $\E^n_{i-1}[(\Delta^n_i \wt v)^2]\leq \Den^{2H}(h(\Den))^2\leq C\Den^{2H}$. Since $\si^v$ and $\ov\si^v$ are uniformly bounded by $K$ and $g_0\equiv0$, It\^o's isometry shows that
	\begin{equation}\label{eq:rough-mom}\begin{split}
			\E^n_{i-1}[(\Delta^n_i \ov v)^2]	& = \int_0^{i\Den} [g(i\Den-s)-g((i-1)\Den-s)]^2[(\si^v_s)^2+(\ov\si^v_s)^2] ds\\
			&\leq \frac{2K^2}{K_H^2}\int_0^{i\Den} [(i\Den-s)_+^{H-1/2}-((i-1)\Den-s)_+^{H-1/2}]^2 ds\\
			&= \frac{2K^2\Den^{2H}}{K_H^2}\int_0^{i} [s^{H-1/2}-(s-1)_+^{H-1/2}]^2 ds\\
			&\leq\frac{2K^2\Den^{2H}}{K_H^2}\int_0^{\infty} [s^{H-1/2}-(s-1)_+^{H-1/2}]^2 ds.
	\end{split}\end{equation}
	The last integral is finite (in fact, equal to $K_H^2$ by Theorem 1.3.1 of \citet{Mishura08}), which yields the bound on $\Delta^n_i c$ in \eqref{eq:L2-2}. 
	The bound on $\Delta^n_i \rho$ in \eqref{eq:L2-2} follows analogously. 
	
	For the third part of the lemma,   dominated convergence   shows $\vp(n)\to0$ because  
	\begin{equation*}
		\vp(n)	\leq \int_E \sup_{\om\in\Om}\sup_{t\in[0,\infty)}\sup_{u\in\calu} \Den\lvert e^{iu\ga(t,z)/\sqrt{\Den}}-1-iu\ga(t,z)/\sqrt{\Den}\rvert \la(dz)
	\end{equation*}
	and the integrand tends to $0$ pointwise in $z$ and is bounded by $\frac12(\sup_{u\in\calu} u^2) J(z)$, which is integrable with respect to $\la$ and independent of $n$. If $\theta_0=0$, we can use the estimate $\lvert e^{ix}-1-ix\rvert \leq C\lvert x\rvert^{2-\vareps}$ and (\ref{eq:prop1-2-noise}) to upgrade the previous bound to
	$\vp(n)\leq C\Delta_n^{\varepsilon/2}\sup_{u\in\calu}\lvert u\rvert^{2-\vareps} \int_E J(z) \la(dz) = o'(1)$. 
	A similar argument shows $\phi(n)=o'(1)$.
	
	For the last part of the lemma, we   deduce the first inequality from the bound $\lvert e^{iux}-1-iux\rvert\leq \frac12 (ux)^2$ and the assumption that $\E[\Delta\chi_1]=0$. To get the second inequality, we bound
	\begin{align*}
		\lvert \Psi(u')-\Psi(u)\rvert 	&= \lvert \E[e^{iu\Delta\chi_1}(e^{i(u'-u)\Delta\chi_1}-1)]\rvert\\
		&\leq \lvert \E[ e^{iu\Delta\chi_1} (u'-u)\Delta\chi_1] \rvert + \tfrac12 (u'-u)^2\E[(\Delta\chi_1)^2].
	\end{align*}
	Since $\lvert e^{iu\Delta\chi_1}-1\rvert\leq \lvert u\Delta\chi_1\rvert$ and $\Delta\chi_1$ is centered, we obtain the desired inequality.
	%
\end{proof}

Next, we recall (\ref{eq:split}) and   introduce the notations $\calg_i=\si( \chi_j:j\leq i)$, $\calh^n_i = \calf_{i\Den}\vee\calg_i$ and   $\ov\E^n_i[X]=\E[X\mid \calh^n_i]$. The next lemma gathers estimates for the increments  $\Delta^n_i \wt L^{\mathrm{v}}(u)=\wt L^{n,\rm v}_j(u)-\wt L^{n,\rm v}_{j-1}(u)$, $\Delta^n_i \wt L^{\mathrm{b}}(u)=\wt L^{n,\rm b}_j(u)-\wt L^{n,\rm b}_{j-1}(u)$ and $\Delta^n_i \wt L^{\mathrm{s}}(u)=\wt L^{n,\rm s}_j(u)-\wt L^{n,\rm s}_{j-1}(u)$. 

\begin{lemma}\label{lem:bounded} Grant Assumption~\ref{ass:H0-2} and let $\calu\subseteq\R$ be a compact set. There are constants $C\in(0,\infty)$ and, for any $p\geq2$,  $C_p\in(0,\infty)$ such that the following holds for any $n\in\N$, $j=2,\dots,\lfloor T/(2p_n\Den)\rfloor$, $\ell\in\{0,1\}$ and   $\calh^n_{((2j-2)p_n-m-1)}$-measurable random  variable $U$ with values in $\calu$:
	\begin{enumerate}
		\item[(i)] Let $[j]^m_n=jp_n-m-1$. Then  
		\begin{equation}\label{eq:L'} 
			\ov	\E^n_{[2j-2]^m_n}\bigl[ \lvert \wt{L}_{2j-\ell}^{n,\mathrm{v}}(U)\rvert^p\bigr]\leq C_p/k_n^{p/2},\quad\ov \E^n_{[2j-2]^m_n}\bigl[ \lvert \Delta^n_{2j}\wt{L}^{\mathrm{v}}(U)\rvert^p \bigr]\leq C_p/k_n^{p/2}.
		\end{equation}
		\item[(ii)] For some function $h_2(t)$ satisfying $h_2(\Den)=o'(1)$, we have
		\begin{equation}\label{eq:L''} 
			\ov\E^n_{[2j-2]^m_n}\bigl[ \lvert \wt{L}_{2j-\ell}^{n,\mathrm{b}}(U)\rvert^p\bigr]\leq C_p \Den^{p/2},\quad \ov\E^n_{[2j-2]^m_n}\bigl[ \lvert \Delta^n_{2j} \wt{L}^{\mathrm{b}}(U)\rvert^p\bigr]\leq C_p (\Den h_2(\Den))^{p/2}.
		\end{equation}
		\item[(iii)] Let 
		$
		\Theta(u)_t=iu\int_0^t \al_r dr- \frac12 u^2\int_0^t\si^2_r dr + \int_0^t\vp(u)_r dr+\int_0^t \phi(u)_rdr$ and
		\begin{align*}
			\Delta^n_i \Theta(u)&=\Theta(u)_{i\Den}-\Theta(u)_{(i-1)\Den}, \\
			\Delta\Theta(u)^n_{i}&=iu\al_{i\Den}\Den -\tfrac12 u^2 \si^2_{i\Den}\Den+\Den\vp(u)_{i\Den} +\Den\phi(u)_{i\Den}.
		\end{align*}
		Then 
		\begin{equation}\label{eq:L'''-2} 
			\wt{L}_{2j-\ell}^{n,\mathrm{s}}(U) =\frac{1}{k_n}\sum_{i=(2j-\ell-1)p_n+1}^{(2j-\ell-1)p_n+k_n}e^{\Delta\Theta(U_n)^n_{i-1}}  \Psi(U_n\Den^{\varpi}\rho_{(i-1)\Den}),
		\end{equation}
		and
		\begin{equation}\label{eq:L'''-3} 
			\ov\E^n_{[2j-2]^m_n}\bigl[ \lvert \Delta^n_{2j}\wt{L}^{\mathrm{s}}(U)\rvert^2 \bigr]=\E^n_{[2j-2]^m_n}\bigl[ \lvert \Delta^n_{2j}\wt{L}^{\mathrm{s}}(U)\rvert^2 \bigr]\leq Cp_n\Den.
		\end{equation} 
	\end{enumerate}
\end{lemma}
\begin{proof} The identity \eqref{eq:L'''-2} is a simple consequence of the Lévy--Khintchine formula and the independence of $(\chi_i)_{i\in\Z}$ from $\calf_\infty$.  
	For all other statements, there is no loss of generality to assume that $\calu=[0,u_\ast]$ for some $u_\ast\in(0,\infty)$. Let $Y^n(u)_t=e^{X^n(u)_t}$, $Y^{\prime n}(u)_t= e^{X^{\prime n}(u)_t}$, $\ov Y^n(u)_t=Y^{ n}(u)_tY^{\prime n}(u)_t$ and
	\begin{equation}\label{eq:XY}\begin{split}
			X^n(u)_t&=-\tfrac12u^2\Den^{(1-2\varpi)_+}\si^2_t+\Den\vp_t(u/\Den^{\varpi\wedge 1/2})+\wt X^n(u)_t\bone_{\{\varpi\leq \frac34\}}, \\
			X^{\prime n}(u)_t&=iu\Den^{(1-\varpi)\vee 1/2}\al_t +\Den\phi(u/\Den^{\varpi\wedge 1/2})_t+\wt X^n(u)_t\bone_{\{\varpi> \frac34\}},\\
			\wt X^n(u)_t&=\Log \Psi(u\Den^{(\varpi-1/2)_+}\rho_t).
	\end{split}\end{equation}
	By \eqref{eq:L'''-2},  we can decompose
	\begin{align} 
		\Delta^n_{2j}\wt L^{\mathrm{s}}(U)&= \frac{1}{k_n}\sum_{i=(2j-1)p_n+1}^{(2j-1)p_n+k_n} (\ov Y^n(U)_{(i-1)\Den}-\ov Y^n(U)_{(i-p_n-1)\Den})\nonumber  \\
		&=\frac{1}{k_n}\sum_{i=(2j-1)p_n+1}^{(2j-1)p_n+k_n}   \Bigl\{Y^{n}(U)_{(i-1)\Den}(  Y^{\prime n}(U)_{(i-1)\Den}-  Y^{\prime n}(U)_{(i-p_n-1)\Den})\label{eq:Delta-Ls}\\ &\quad+   (  Y^{  n}(U)_{(i-1)\Den}-  Y^{  n}(U)_{(i-p_n-1)\Den})Y^{\prime n}(U)_{(i-1)\Den} \Bigr\}.\nonumber
	\end{align}
	Since $\rho$ takes the form (\ref{eq:pi}) if $\varpi\leq\frac34$,  the mean-value theorem combined with \eqref{eq:L2}  
	yields
	\begin{equation}\label{eq:aux22} 
		\E^n_{[2j-2]^m_n}[\lvert Y^n(U)_{(i-1)\Den}-Y^n(U)_{(i-p_n-1)\Den}  \rvert^2]\leq C
		p_n\Den.
	\end{equation}
	By assumption, $\al$ is continuous (if $\theta_0=0$, $\vareps$-H\"older continuous) in $L^2$, so by \eqref{eq:vp-conv} and \eqref{eq:psi},  
	\begin{equation}\label{eq:aux23} 
		\E^n_{[2j-2]^m_n}[\lvert Y^{\prime n}(U)_{(i-1)\Den}-Y^{\prime n}(U)_{(i-p_n-1)\Den}  \rvert^2]=o'(\Den)=o'(p_n\Den). 
	\end{equation}
	Because $Y^n$ and $Y^{\prime n}$ are bounded by $1$, we obtain \eqref{eq:L'''-3}.
	
	For \eqref{eq:L'}, 
	note that $\xi^n_i = e^{iU_n\Delta^n_i y}-\E^n_{i-1}[e^{iU_n\Delta^n_i y}]$ 
	is $\calh^n_i$-measurable and because $(\Delta\chi_i)_{i\in\Z}$ is $(m+1)$-dependent and independent of $\calf_\infty$, we have $\ov\E^n_{i-m-2}[e^{iU_n\Delta^n_i y}]=\E^n_{i-m-2}[e^{iU_n\Delta^n_i y}]$. This shows that $\ov\E^n_{i-m-2}[\xi^n_i]=0$. Upon writing
	\[ \wt{L}_{2j-\ell}^{n,\mathrm{v}}(U) =\sum_{i=1}^{m+2} \wt{L}_{2j-\ell,i}^{n,\mathrm{v}}(U),\quad\wt{L}_{2j-\ell,i}^{n,\mathrm{v}}(U)= \frac1{k_n}\sum_{k\geq0} \xi^n_{(2j-\ell-1)p_n+i+k(m+2)}\bone_{\{i+k(m+2)\leq k_n\}}, \]
	we realize that $\wt{L}_{2j-\ell,i}^{n,\mathrm{v}}(U)$ is a martingale sum for each $i$, relative to the discrete-time filtration $(\calh^n_{(2j-\ell-1)p_n+i+k(m+2)}: k\geq-1)$. Also, $\wt{L}_{2j-\ell,i}^{n,\mathrm{v}}(U)$ sums over at most $k_n/(m+2)$ nonzero terms and $\lvert \xi^n_i\rvert\leq2$. Therefore, combining the Minkowski and the Burkholder--Davis--Gundy (BDG) inequalities, we obtain the first inequality in \eqref{eq:L'} from the estimates
	\begin{equation}\label{eq:aux10}\begin{split}
			&\E^n_{[2j-2]^m_n}[\lvert   \wt{L}_{2j-\ell}^{n,\mathrm{v}}(U)\rvert^p]^{1/p}=\ov\E^n_{[2j-2]^m_n}[\lvert   \wt{L}_{2j-\ell}^{n,\mathrm{v}}(U)\rvert^p]^{1/p}\leq \sum_{i=1}^{m+2}\ov\E^n_{[2j-2]^m_n}[\lvert   \wt{L}_{2j-\ell,i}^{n,\mathrm{v}}(U)\rvert^p]^{1/p}\\
			&\quad\leq C_pk_n^{-1}\sum_{i=1}^{m+2}\ov\E^n_{[2j-2]^m_n}\Biggl[\biggl(\sum_{k}\ov\E^n_{[2j-\ell-1]^m_n+i+k(m+2)-1}[\lvert\xi^n_{(2j-\ell-1)p_n+i+k(m+2)}\rvert^2 ] \biggr)^{p/2}\Biggr]^{1/p} \\
			&\quad\leq 2 C_p\sqrt{(m+2)/k_n},\end{split}
	\end{equation}
	where $\sum_k$ sums over all $k$ such that $i+k(m+2)\leq k_n$
	The second inequality in \eqref{eq:L'}   follows by the triangle inequality.
	
	Next, we decompose
	\begin{align} 
		\E^n_{i-1}[e^{iU_n\Delta^n_i y}-e^{iU_n\Delta^n_{i,i-1} y} ] &=\E^n_{i-1}[e^{iU_n\Delta^n_{i,i-1}x}(e^{iU_n\Delta^n_i \eps}-e^{iU_n\Delta^n_{i,i-1} \eps})] \nonumber\\
		&\quad+\E^n_{i-1}[e^{iU_n\Delta^n_{i,i-1}\eps}(e^{iU_n\Delta^n_i x}-e^{iU_n\Delta^n_{i,i-1} x})]\label{eq:aux15}  \\
		&\quad +\E^n_{i-1}[(e^{iU_n\Delta^n_i x}-e^{iU_n\Delta^n_{i,i-1} x})(e^{iU_n\Delta^n_i \eps}-e^{iU_n\Delta^n_{i,i-1} \eps})].\nonumber
	\end{align}
	By Lemma~\ref{lem:prelim} and the Cauchy--Schwarz inequality, the last term is $o'(\sqrt{\Den})$ and can therefore be neglected in the proof of \eqref{eq:L''}. Similarly,
	\begin{align*}
		&\E^n_{i-1}[e^{iU_n\Delta^n_{i,i-1}x}(e^{iU_n\Delta^n_i \eps}-e^{iU_n\Delta^n_{i,i-1} \eps})]\\
		&\quad=iU_n\E^n_{i-1}[e^{iU_n\Delta^n_{i,i-1}x}e^{iU_n\Delta^n_{i,i-1} \eps}(\Delta^n_i \eps-\Delta^n_{i,i-1} \eps)] + O(\Den\bone_{\{\varpi\leq \frac34\}}+\Den^{ (2\varpi-1)}\bone_{\{\varpi>\frac34\}}) \\
		&\quad=iU \Den^{(\varpi-1/2)_+}\E^n_{i-1}[e^{iU_n\Delta^n_{i,i-1}x}e^{iU_n\Delta^n_{i,i-1} \eps}(\rho_{i\Den}-\rho_{(i-1)\Den})\chi_{i}]+o'(\sqrt{\Den}).
	\end{align*}
	Since $(\chi_i)_{i\in\Z}$ is independent of $\calf_\infty$, we obtain (using the notation $\wh\Psi(u)=\E[e^{iu\Delta\chi_1}\chi_1]$)
	\begin{equation}\label{eq:aux18}\begin{split}
			&\E^n_{i-1}[e^{iU_n\Delta^n_{i,i-1}x}(e^{iU_n\Delta^n_i \eps}-e^{iU_n\Delta^n_{i,i-1} \eps})]\\
			&\quad=iU \Den^{(\varpi-1/2)_+}\wh\Psi(U\Den^{(\varpi-1/2)_+}\rho_{(i-1)\Den})\E^n_{i-1}[e^{iU_n\Delta^n_{i,i-1}x}\Delta^n_i\rho]+o'(\sqrt{\Den}),
	\end{split}\end{equation}
	where  the $o'(\sqrt{\Den})$-term is independent of $\om$. If $\varpi> \frac34$,   the mean-value theorem implies the bound $\lvert \Re\wh\Psi(U\Den^{(\varpi-1/2)_+}\rho_{(i-1)\Den})-\Re\wh\Psi(0)\rvert \leq \lvert \frac d{du} \Re \wh\Psi(\wh \rho_n) \rvert U\Den^{\varpi-1/2}\lvert \rho_{(i-1)\Den}\rvert\leq u_\ast K \lvert \E[\chi_1\Delta\chi_1]\rvert \Den^{\varpi-1/2}$ for some intermediate value $\wh\rho_n$ and an analogous bound for the imaginary part. Since $\wh \Psi(0) = \E[\chi_1]=0$ and we have another factor $\Den^{\varpi-1/2}$ in front of $\wh\Psi$ in \eqref{eq:aux18}, it follows that \eqref{eq:aux18} is $o'(\sqrt{\Den})$  if $\varpi>\frac34$. If $\varpi \leq\frac34$, the last line in \eqref{eq:aux18} is $O(\sqrt{\Den})$ because of   \eqref{eq:L2}.  Therefore, any additional modification that yields an extra $o'(1)$-term can be absorbed into the $o'(\sqrt{\Den})$-bin in \eqref{eq:aux18}. There are two cases: If $\frac12<\varpi\leq \frac34$, then the $\Den^{\varpi-1/2}$-factor in front of $\wh\Psi$ renders \eqref{eq:aux18} an $o'(\sqrt{\Den})$-term. If $\varpi\leq \frac12$, we note that  $\lvert \al^\rho\rvert\leq K$ and $\Ga^\rho \equiv0$, so by Lemma~\ref{lem:o},  
	$
	\lvert\E^n_{i-1}[e^{iU_n\Delta^n_{i,i-1}x}\Delta^n_i \rho]-\E^n_{i-1} [e^{iU_n\Delta^n_{i,i-1}x}\int_{(i-1)\Den}^{i\Den} (\si_s^\rho dW_s +\wt \si_s^\rho d\wt W_s) ] \rvert	  \leq K\Den + \sqrt{\Den}h_3(\Den)
	$
	almost surely, for some $h_3$ that satisfies $h_3(\Den)=o'(1)$. In addition, because $\wt W$ is independent of $W$ and $w(\Den),\vp(n),\phi(n)=o'(1)$, we have
	\begin{align*}
		&\E^n_{i-1}\biggl[e^{iU_n\Delta^n_{i,i-1}x}\int_{(i-1)\Den}^{i\Den} (\si_s^\rho dW_s +\wt \si_s^\rho d\wt W_s)\biggr]	=\E^n_{i-1}\biggl[e^{iU_n\Delta^n_{i,i-1}x}\int_{(i-1)\Den}^{i\Den} \si_s^\rho dW_s\biggr] \\
		&\quad=\si^\rho_{(i-1)\Den}\E^n_{i-1} [e^{iU_n\Delta^n_{i,i-1}x}\Delta^n_i W]+O(\sqrt{\Den}w(\Den))\\
		&\quad=\si^\rho_{(i-1)\Den}e^{iU_n\al_{(i-1)\Den}\Den+\Den\vp(U_n)+\Den\phi(U_n)}\E^n_{i-1} [e^{iU_n\si_{(i-1)\Den}\Delta^n_i W}\Delta^n_i W] + o'(\sqrt{\Den})\\
		&\quad=-i^{-1}e^{-\frac12 U^2\Den^{1-2\varpi}\si^2_{(i-1)\Den}}U\si_{(i-1)\Den}\si^\rho_{(i-1)\Den}\Den^{1-\varpi}+ o'(\sqrt{\Den}),
	\end{align*}
	where the last step follows from the fact that $U_n=U/\Den^\varpi$ when $\varpi\leq\frac12$ and $\E[e^{iuX}X]=i^{-1}\frac{d}{du} \E[e^{iuX}]$. Since the last line  above is $o'(\sqrt{\Den})$ except when $\varpi=\frac12$,  we have shown that 
	\begin{equation}\label{eq:aux14}\begin{split}
			&\E^n_{i-1}[e^{iU_n\Delta^n_{i,i-1}x}(e^{iU_n\Delta^n_i \eps}-e^{iU_n\Delta^n_{i,i-1} \eps})]\\
			&\quad=-U^2e^{-\frac12 U^2\si^2_{(i-1)\Den}} \wh\Psi(U\rho_{(i-1)\Den})\si_{(i-1)\Den}\si^\rho_{(i-1)\Den}\sqrt{\Den}\bone_{\{\varpi=\frac12\}}+o'(\sqrt{\Den}).\!\!\!
	\end{split}\end{equation}
	
	We pause here and move to the second term on the right-hand side of \eqref{eq:aux15}, that is,
	\[\E^n_{i-1}[e^{iU_n\Delta^n_{i,i-1}\eps}(e^{iU_n\Delta^n_i x}-e^{iU_n\Delta^n_{i,i-1} x})]=\Psi(U_n\Den^{\varpi}\rho_{(i-1)\Den})\E^n_{i-1}[e^{iU_n\Delta^n_i x}-e^{iU_n\Delta^n_{i,i-1} x}].\]
	Writing
	$\La(u)_t=iu\int_0^t \si_r dW_r+\int_0^t\int_E (e^{iu(\ga(r,z)+\Ga(r,z))}-1)(\mu-\nu)(dr,dz)$ and
	using It\^o's formula (see   \citet[Theorem I.4.57]{JS03}), one can verify that for any fixed $u$ and $s$, the process $Z(u,s)_{t}=e^{iu(x_t-x_s)-(\Theta(u)_t-\Theta(u)_s)}$
	satisfies the stochastic differential equation $dZ(u,s)_t = Z(u,s)_{t-} d\La(u)_t$ with $Z(u,s)_s=1$.
	In particular, since $\La(u)$ is a martingale, so is $Z(u,s)$.
	Combining this with the Lévy--Khintchine formula and using the notation $Z^n_i(u)_t = Z(u,(i-1)\Den)_t$, we obtain  
	\begin{align}\nonumber
		\E^n_{i-1}[e^{iU_n\Delta^n_i x}-e^{iU_n\Delta^n_{i,i-1} x} ] &=\E^n_{i-1}[e^{\Delta^n_i \Theta(U_n)}Z^n_i(U_n)_{i\Den}] -e^{\Delta\Theta(U_n)^n_{i-1}} \\
		& =\E^n_{i-1}[(e^{\Delta^n_i \Theta(U_n)}-e^{\Delta\Theta(U_n)^n_{i-1}})Z^n_i(U_n)_{i\Den}] \label{eq:aux11}\\
		&=e^{\Delta\Theta(U_n)^n_{i-1}}\E^n_{i-1}[(e^{\Delta^n_i \Theta(U_n)-\Delta\Theta(U_n)^n_{i-1}}-1)Z^n_i(U_n)_{i\Den}]. 
		\nonumber
	\end{align}
	As a consequence of our assumptions on $\al$, $\si^2$, $\ga$ and $\Ga$ and the elementary inequalities $\lvert U_n\rvert\leq u_\ast/\sqrt{\Den}$, $\lvert e^{ix}-1\rvert \leq \lvert x\rvert$ and $\lvert e^{ix}-1-ix\rvert\leq \frac12x^2$, we have
	\begin{align*}
		&	\lvert \Delta^n_i \Theta(U_n)\rvert\vee\lvert\Delta \Theta(U_n)^n_{i-1}\rvert\\	&\quad\leq u_\ast \sqrt{\Den} K+\frac12 u_\ast^2K + \frac12u_\ast^2\int_E J(z)\la(dz)+\lvert U\rvert\sqrt{\Den} \int_E J(z)\la(dz) \\
		&	\quad	\leq 2K (u_\ast +\tfrac12 u_\ast^2 ).	
	\end{align*}
	Denote the last term by $M$. Then   $\lvert e^{\Delta \Theta(U_n)^n_{i-1}}\rvert\leq e^M$ and $\lvert Z^n_i(U_n)_{i\Den}\rvert =e^{-\Re \Delta^n_i\Theta(U_n)} \leq e^M$, so  
	\begin{equation}\label{eq:aux17}\begin{split}
			&\Bigl\lvert	\E^n_{i-1}[e^{iU_n\Delta^n_i x}-e^{iU_n\Delta^n_{i,i-1} x} ]-e^{\Delta\Theta(U_n)^n_{i-1}}\E^n_{i-1}[(\Delta^n_i \Theta(U_n)-\Delta\Theta(U_n)^n_{i-1}) Z^n_i(U_n)_{i\Den}]\Bigr\rvert \\
			&\quad\leq \tfrac12e^{2M}\E^n_{i-1}[\lvert \Delta^n_i \Theta(U_n)-\Delta\Theta(U_n)^n_{i-1}\rvert^2].
	\end{split}\end{equation}
	
	Defining 
	\begin{equation}\label{eq:xi}\begin{split}
			\xi^{n,i}_1&=-\frac{U^2_n}2\int_{(i-1)\Den}^{i\Den}  \int_{(i-1)\Den}^{s} \si^c_rdW_rds,\\ \xi^{n,i}_2&=-\frac{U^2_n}2\int_{(i-1)\Den}^{i\Den}\int_{(i-1)\Den}^s \ov\si^c_rd\ov W_rds,\end{split}
	\end{equation}
	and recalling $\phi(n)$ from \eqref{eq:vp-conv}, we have
	\begin{equation}\label{eq:4parts}\begin{split}
			& \E^n_{i-1}[\lvert \Delta^n_i \Theta(U_n)-\Delta\Theta(U_n)^n_{i-1}-\xi^{n,i}_1-\xi^{n,i}_2\rvert] \\
			&\quad\leq \frac{u_\ast}{\sqrt{\Den}}\int_{(i-1)\Den}^{i\Den} \E^n_{i-1}[ \lvert \al_s-\al_{(i-1)\Den}\rvert]ds +2\sqrt{\Den}\phi(n)\\
			&  \quad  +\Den \sup_{s\in[(i-1)\Den,i\Den]}  \E^n_{i-1}[\lvert \vp(U_n)_s-\vp(U_n)_{(i-1)\Den}\rvert]+\frac12u_\ast^2  \int_{(i-1)\Den}^{i\Den}  \E^n_{i-1}[\lvert \al^c_r\rvert] dr\\
			&   \quad +\frac12u_\ast^2  \E\biggl[\sup_{s\in[(i-1)\Den,i\Den]}\biggl\lvert \iint_{(i-1)\Den}^s \ga^c(r,z)(\mu-\nu)(dr,dz)\biggr\lvert\biggr].
	\end{split}\end{equation}
	The first term on the right-hand side is bounded by $u_\ast \sqrt{\Delta_n}w(\Delta_n)$, where  $w(\Delta_n)=o'(1)$. For the second term, note that $\phi(n)=o'(1)$ by \eqref{eq:vp-conv}. 
	The third  term  is bounded by
	\begin{equation}\label{eq:aux3}\begin{split}
			&  \sup_{s\in[0,\infty)} \Biggl(\Den^2\E^n_{i-1}[\lvert\al^\vp(U_n)_s\rvert] + \Den^{3/2}\E^n_{i-1}[\si^\vp(U_n)_s^2+\ov\si^\vp(U_n)_s^2]^{1/2}\\ 
			&\qquad  +\Den^{3/2}\sup_{z\in E} \bigl(\E^n_{i-1}[\ga^\vp(U_n;s,z)^2]/J(z)\bigr)^{1/2}\biggl(\int_E J(z)\la(dz)\biggr)^{1/2} \\
			&\qquad + \Den^2\sup_{z\in E} \lvert \Ga^\vp(U_n;s,z)/J(z)\rvert \int_E J(z)\la(dz)\Biggr)\\
			&\quad\leq K\Den +\sqrt{C_1(\Delta_n)\Den}+\sqrt{KC_2(\Den)\Den} + KC_3(\Den)\Den
	\end{split}\end{equation}
	due to (\ref{eq:prop2-2}) and  (\ref{eq:prop3-2})--(\ref{eq:prop7-2}). Thanks to the bound $\lvert \al^c_r\rvert \leq K$ and Lemma~\ref{lem:o}, the last two terms in \eqref{eq:4parts} are bounded by $C\Den$ and $C\sqrt{\Den}h_3(\Den)$,  respectively. 
	Altogether, we have shown that the right-hand side of \eqref{eq:4parts} is $o'(\sqrt{\Den})$. In a similar fashion, one can further show that the right-hand side of \eqref{eq:aux17} is $O({\Den})$, which implies
	\begin{align*}
		&\E^n_{i-1}[e^{iU_n\Delta^n_{i,i-1}\eps}(e^{iU_n\Delta^n_i x}-e^{iU_n\Delta^n_{i,i-1} x})]\\
		&\quad=\Psi(U_n\Den^{\varpi}\rho_{(i-1)\Den})e^{\Delta\Theta(U_n)^n_{i-1}}\E^n_{i-1}[(\xi^{n,i}_1+\xi^{n,i}_2)Z^n_i(U_n )_{i\Den}] + o'(\sqrt{\Den})\\
		&\quad=\Psi(U_n\Den^{\varpi}\rho_{(i-1)\Den})e^{\Delta\Theta(U_n)^n_{i-1}}\E^n_{i-1}[\xi^{n,i}_1 Z^n_i(U_n )_{i\Den}] + o'(\sqrt{\Den}),
	\end{align*}
	almost surely. The last step holds because $Z^n_i(U_n )_t$ is a martingale driven by $W$ and $\mu-\nu$ (without any component driven by $\ov W$), so $Z^n_i(U_n )_t$ and $\xi^{n,i}_2$ are uncorrelated conditionally on $\calf_{(i-1)\Den}$. Since $\si^c$ satisfies (\ref{eq:prop0-2}), we further have
	\begin{align*}
		&\E^n_{i-1}[\xi^{n,i}_1 Z^n_i(U_n )_{i\Den}]\\	
		&~=-\frac{U_n^2}2\si_{(i-1)\Den}^c\E^n_{i-1}\biggl[ \int_{(i-1)\Den}^{i\Den} \int_{(i-1)\Den}^s dW_sds  Z^n_i(U_n )_{i\Den}\biggr] +o'(\sqrt{\Den})\\
		&	~ =-\frac{U_n^2}{2}\si^c_{(i-1)\Den}  \E^n_{i-1}\biggl[\int_{(i-1)\Den}^{i\Den} (i\Den-r) dW_r\int_{(i-1)\Den}^{i\Den} Z^n_i(U_n )_{r-}d\La(U_n)_r  \biggr] +o'(\sqrt{\Den}) \\
		& ~=-\frac{iU_n^3}2\si^c_{(i-1)\Den}  \E^n_{i-1}\biggl[\int_{(i-1)\Den}^{i\Den} (i\Den-r) dW_r\int_{(i-1)\Den}^{i\Den} Z^n_i(U_n )_r\si_rdW_r  \biggr]+o'(\sqrt{\Den})\\
		& ~=-\frac{iU_n^3}2\si_{(i-1)\Den}\si^c_{(i-1)\Den}  \int_{(i-1)\Den}^{i\Den} (i\Den-r)   \E^n_{i-1}[Z^n_i(U_n )_r]dr   +o'(\sqrt{\Den})\\
		&~=-\frac14iU^3\Den^{(3/2-3\varpi)_+}\si_{(i-1)\Den}\si^c_{(i-1)\Den}\sqrt{\Den}+o'(\sqrt{\Den}).
	\end{align*}
	Recalling \eqref{eq:aux15} and \eqref{eq:aux14}, we arrive at 
	\begin{equation}\label{eq:aux16}\begin{split} 
			&\E^n_{i-1}[e^{iU_n\Delta^n_i y}-e^{iU_n\Delta^n_{i,i-1}y}]	\\
			& \quad=-U^2e^{-\frac12 U^2\si_{(i-1)\Den}^2} \wh\Psi(U\rho_{(i-1)\Den})\si_{(i-1)\Den}\si^\rho_{(i-1)\Den}\sqrt{\Den}\bone_{\{\varpi=\frac12\}} \\
			&\qquad-\frac{iU^3}4e^{-\frac12U^2\si^2_{(i-1)\Den}} \Psi(U\Den^{\varpi-1/2}\rho_{(i-1)\Den})\si_{(i-1)\Den}\si^c_{(i-1)\Den}\sqrt{\Den}\bone_{\{\varpi\geq\frac12\}}\\
			&\qquad+o'(\sqrt{\Den}),
	\end{split}\end{equation}
	where the $o'$-term does not depend on $i$ or $\om$.
	This   yields the first estimate in \eqref{eq:L''}. 
	
	To prove the second estimate, denote the right-hand side of \eqref{eq:aux16} (without the $o'$-term) by $A^n_i$ and note that $\Delta^n_{2j}\wt L^{\rm b}(U)=
	\frac1{k_n}\sum_{i=(2j-1)p_n+1}^{(2j-1)p_n+k_n} (A^n_i-A^n_{i-p_n})+o'(\sqrt{\Den})$ by what we have shown so far. If $\varpi>\frac34$, \eqref{eq:psi} implies   $\lvert\Psi(U\Den^{\varpi-1/2}\rho_{(i-1)\Den})-\Psi (U\Den^{\varpi-1/2}\rho_{(i-p_n-1)\Den})\rvert\leq C \Den^{2\varpi-1}=o'(\sqrt{\Den})$.
	Because $\si$, $\si^c$ and, if $\varpi\leq \frac34$, $\rho$ and $\si^\rho$   are continuous (if $\theta_0=0$, at least $\vareps$-H\"older continuous) in $L^p$ and both $\Psi$ and $\wh\Psi$ are differentiable (as $\chi$ has moments of all orders), it follows that $\E^n_{[2j-2]^m_n}[\lvert A^n_i-A^n_{i-p_n}\rvert^p]^{1/p}=o'(\sqrt{\Den})$, proving the second estimate in \eqref{eq:L''}.
\end{proof}

\begin{lemma}\label{lem:o}
	Let $X_t=\iint_0^t \ga^X(s,z)(\mu-\nu)(ds,dz)$, where $\ga^X$ is a predictable function and satisfies $\lvert \ga^X(s,z)\rvert^{2-\vareps\bone_{\{\theta_0=0\}}} \leq J(z)$ for all $s\geq0$, $z\in E$ and some measurable nonnegative function $J(z)$ with $\int_E J(z)\la(dz)<\infty$. Then, for any $p\in[1,2)$,
	\[ \E\biggl[\sup_{t\in[s,s+\Den]} \lvert X_t-X_s\rvert^p\biggr]^{1/p} \leq \sqrt{\Den}h_3(\Den), \]
	where $h_3$ does not depend on  $\ga^X$ or $s$ and satisfies $h_3(\Den)=o'(1)$.
\end{lemma}
\begin{proof} Let $q=1$ if $\theta_0>0$ and $q=\frac{2}{2-\varepsilon}$ if $\theta_0=0$. There is no loss of generality to assume that $qp/2 < 1$.
	By the BDG inequality, 
	\begin{align*}
		&	\E\biggl[\sup_{t\in[s,s+\Den]} \lvert X_t-X_s\rvert^p\biggr] \leq	C\E\biggl[\biggl(\iint_{(i-1)\Den}^{i\Den} (\ga^X(s,z))^2 \mu(ds,dz)\biggr)^{p/2}\biggr]\\
		&\quad\leq C\Biggl\{\E\biggl[\biggl(\iint_{(i-1)\Den}^{i\Den} J(z)^q\bone_{\{J(z)\leq \Den\}} \mu(ds,dz)\biggr)^{p/2}\biggr]\\
		&\qquad\qquad\qquad\qquad\qquad+\E\biggl[\biggl(\iint_{(i-1)\Den}^{i\Den} J(z)^q\bone_{\{J(z)> \Den\}} \mu(ds,dz)\biggr)^{p/2}\biggr]\Biggr\}\\
		&\quad\leq C\Biggl\{\Den^{p/2}\biggl(\int_E J(z)^q\bone_{\{J(z)\leq \Den\}}\la(dz)\biggr)^{p/2}+\Den\int_E J(z)^{qp/2}\bone_{\{J(z)>\Den\}}\la(dz)\Biggr\}.
	\end{align*}
	For the last step, we applied Jensen's inequality to the first expectation and the bound $(a+b)^{p/2}\leq a^{p/2}+b^{p/2}$ to the second. Concerning the first term, note that $J(z)^q\leq \Den^{q-1}J(z)$ if $J(z)\leq \Den$. Moreover, since $J$ is integrable, $h_{31}(t)=\int_E J(z)\bone_{\{J(z)\leq t\}}\la(dz)$ satisfies $h_{31}(t)\to0$ as $t\to0$. Concerning the second term, we use H\"older's inequality to bound
	\begin{equation*}
		\Den^{1-qp/2}\int_E J(z)^{qp/2}\bone_{\{J(z)>\Den\}}\la(dz)	\leq\biggl(\int_E J(z)\la(dz)\biggr)^{qp/2} (\Den\la(J(z)>\Den))^{1-qp/2}.
	\end{equation*}
	Since $J$ is integrable and $t\bone_{\{J(z)>t\}}\leq t(J(z)/t)=J(z)$, the dominated convergence theorem implies  that $h_{32}(t)=t\la(J(z)>t)$ satisfies $h_{32}(t)\to0$ as $t\to0$. The lemma is proved by choosing $h_3(t)=C^{1/p}t^{(q-1)/2}[\sqrt{h_{31}(t)}+C^{1/p}(\int_E J(z)\la(dz))^{q/2}(h_{32}(t))^{1/p-q/2}]$. 
\end{proof}

The moment estimates derived in the Lemma~\ref{lem:bounded} translate into   pathwise bounds for the variance and bias terms $\wt L^{n,\mathrm{v}}_j(u)$ and $\wt L^{n,\mathrm{b}}_j(u)$ that hold  with high probability. 

\begin{lemma}\label{lem:Om} Under Assumptions~\ref{ass:H0-2} and \ref{ass:U-2}, if we are 
	given a compact set $\calu$, we have $\P(\Om_n)\to1$, where $\Om_n=\Om_n^{(1)}\cap\Om_n^{(2)}$ and
	\begin{align*}
		\Om_n^{(1)}&= \Bigl\{ \lvert  \wt L^{n,\mathrm{v}}_{2j-\ell}(\wt u^n_j)\rvert\vee \lvert  \wt L^{n,\mathrm{b}}_{2j}(\wt u^n_j)\rvert  \leq \Den^{1/9}   \text{ for all }  \ell\in\{0,1\} \text{ and } j=1,\dots,\lfloor T/(2p_n\Den)\rfloor\Bigr\},\\
		\Om_n^{(2)}&=\Bigl\{ \theta_n/\sqrt{\eta_0^+} \leq \wt  u^n_j  \leq   \theta_n/\sqrt{\eta_0^-} \text{ for all }   j=1,\dots,\lfloor T/(2p_n\Den) \rfloor\Bigr\}.
	\end{align*}
	In particular, there are deterministic constants $C_1,C_2\in(0,\infty)$ such that for sufficiently large $n$, we have the following bounds on $\Om_n$ for all $j=1,\dots,\lfloor T/(2p_n\Den)\rfloor$:
	\begin{equation}\label{eq:as-bounds} \begin{split}
			C_1&\leq \lvert\wt L^n_{2j}(\wt u^n_j)\rvert,\lvert\wt L^n_{2j-1}(\wt u^n_j)\rvert,\lvert\wt L^{n,\mathrm{s}}_{2j}(\wt u^n_j)\rvert,\lvert\wt L^{n,\mathrm{s}}_{2j-1}(\wt u^n_j)\rvert\leq C_2,\\
			C_1\theta_n^2&\leq   \log\lvert\wt L^{n,\mathrm{s}}_{2j}(\wt u^n_j)\rvert^{-1},   \log\lvert\wt L^{n,\mathrm{s}}_{2j-1}(\wt u^n_j)\rvert^{-1} \leq C_2\theta_n^2.
	\end{split}	\end{equation}
\end{lemma}
\begin{proof}
	Note that $\P(\Om_n^{(2)})\to1$ by (\ref{eq:U-3}). On $\Om_n^{(2)}$, we have $ \wt u^n_j 	\leq \ov \theta/\sqrt{\eta_0^-}$, where $\ov \theta = \sup_{n\in\N} \theta_n$. 
	Therefore, by Lemma~\ref{lem:bounded} (note that $\wt u^n_j$ is $\calh^n_{[2j-2]^m_n}$-measurable) and Markov's inequality, we have
	$
	\sum_{j=1}^{\lfloor T/(2p_n\Den)\rfloor}\P ( \lvert  \wt L^{n,\mathrm{v}}_{2j-\ell}(\wt u^n_j)\rvert \vee \lvert  \wt L^{n,\mathrm{b}}_{2j-\ell}(\wt u^n_j)\rvert >\Den^{1/9}\text{ for } \ell\in\{0,1\} ) \leq C\Den^{-4/9}\lfloor T/(2p_n\Den)\rfloor k_n^{-2}$,
	which implies $\P((\Om_n^{(1)})^c\cap\Om_n^{(2)})\to0$	 by (\ref{eq:rates}). This yields the first statement of the lemma. 
	For the first set of bounds in \eqref{eq:as-bounds}, observe from \eqref{eq:vp-conv} that 
	\begin{align*}
		\Re (\Delta\Theta(\wt u^n_j/\Den^{\varpi\wedge 1/2})^n_{i-1})  &	\leq  \tfrac12 (\wt u^n_j)^2 K + \vp(n) +\phi(n)\leq   C_2^2 K,\\
		\lvert \Im(\Delta\Theta(\wt u^n_j/\Den^{\varpi\wedge 1/2})^n_{i-1})\rvert 	&\leq Ku^n_j\sqrt{\Den}+\vp(n) + \phi(n)=o(1)
	\end{align*}
	for large $n$. Hence,   $  \Re e^{\Delta\Theta(\wt u^n_j/\Den^{\varpi\wedge 1/2})^n_{i-1}}  = e^{\Re(\Delta\Theta(\wt u^n_j/\Den^{\varpi\wedge 1/2})^n_{i-1})} \cos\Im(\Delta\Theta(\wt u^n_j/\Den^{\varpi\wedge 1/2}) )$ and\linebreak $\lvert \Im e^{\Delta\Theta(\wt u^n_j/\Den^{\varpi\wedge 1/2})^n_{i-1}}\rvert = e^{\Re(\Delta\Theta(\wt u^n_j/\Den^{\varpi\wedge 1/2})^n_{i-1})}\lvert\sin \Im(\Delta\Theta(\wt u^n_j/\Den^{\varpi\wedge 1/2})^n_{i-1})\rvert$ satisfy
	\begin{equation*}
		\tfrac12e^{-C_2^2K}\leq	  \Re e^{\Delta\Theta(\wt u^n_j/\Den^{\varpi\wedge 1/2})^n_{i-1}} \leq 1,\quad
		\lvert \Im e^{\Delta\Theta(\wt u^n_j/\Den^{\varpi\wedge 1/2})^n_{i-1}}\rvert\leq KC_2\sqrt{\Den}+\vp(n)+\phi(n)
	\end{equation*}
	if $n$ is   large. Moreover, $\lvert \wt u^n_j \Den^{(\varpi-1/2)_+}\rho_{(i-1)\Den}\rvert \leq K \Den^{(\varpi-1/2)_+}\theta_n/\sqrt{\eta_0^-} \to 0$ on $\Om_n^{(2)}$, which implies $\Psi( \wt u^n_j \Den^{(\varpi-1/2)_+}\rho_{(i-1)\Den})\to1$. By \eqref{eq:L'''-2}, this gives $\frac14e^{-C_2^2K}\leq\lvert\wt L^{n,\mathrm{s}}_{2j}(\wt u^n_j)\rvert\leq 1$, which in turn shows $\frac18e^{-C_2^2K}\leq\lvert\wt L^{n}_{2j}(\wt u^n_j)\rvert\leq 1$ on $\Om_n$.
	The bounds for $\lvert\wt L^n_{2j-1}(\wt u^n_j)\rvert$ and $\lvert\wt L^{n,\rm s}_{2j-1}(\wt u^n_j)\rvert$ can be derived in the same way. For the second set of inequalities in \eqref{eq:as-bounds}, note that 
	\begin{align*}
		\wt L^{n,\mathrm{s}}_{2j}(\wt u^n_j)&=k_n^{-1}\sum_{i=(2j-\ell-1)p_n+1}^{(2j-\ell-1)p_n+k_n}\ov Y^n(\wt u^n_j)_{(i-1)\Den}\\
		&=k_n^{-1}\sum_{i=(2j-\ell-1)p_n+1}^{(2j-\ell-1)p_n+k_n}  Y^n(\wt u^n_j)_{(i-1)\Den} + O(\sqrt{\Den})
	\end{align*} by  \eqref{eq:XY}. As $   X^n(\wt u^n_j)_{(i-1)\Den} = -\frac12 (\wt u^n_j)^2  c_{(i-1)\Den}\bone_{\{\varpi\geq\frac12\}}+\Log \Psi(\wt u^n_j \rho_{(i-1)\Den} )\bone_{\{\varpi\leq \frac12\}}+o(\theta_n^2)$, we can use \eqref{eq:psi} to obtain    $ \frac14(\theta_n^2/\eta_0^+)[  K^{-1}\bone_{\{\varpi\geq\frac12\}}+ K^{-2}\E[(\Delta\chi_1)^2]\bone_{\{\varpi\leq \frac12\}}]\leq \lvert X^n(\wt u^n_j)_{(i-1)\Den}\rvert\leq (\theta_n^2/\eta_0^-)[  K^{-1}\bone_{\{\varpi\geq\frac12\}}+ K^{-2}\E[(\Delta\chi_1)^2]\bone_{\{\varpi\leq \frac12\}}]$, which completes the proof of \eqref{eq:as-bounds}.
\end{proof}

\subsection{Proof of Technical Results from Appendix~\ref{sec:proof}}

\begin{proof}[Proof of Lemma~\ref{lem:linearize}] Expanding the product, we can decompose $V^n-\wt V^n$ into the sum of four differences, all of which can be treated similarly. Therefore, we only detail the proof that $V^{\prime n}-\wt V^{\prime n}\stackrel{\P}{\longrightarrow}0$, where 
	\begin{equation}\label{eq:Vn-prime}\begin{split}
			V^{\prime n}&=\frac{ k_n }{\sqrt{\lvert\calj_n\rvert }}\sum_{j\in\calj_n} \Delta^n_{2j} \wt \cf( \wt{u}^n_j)  \Delta^n_{2j-2} \wt \cf(\wt u^n_{j-1}),\\	\wt V^{\prime n}&=\frac{k_n}{\sqrt{\lvert \calj_n\rvert}}\sum_{j\in\calj_n}\Re\biggl\{\frac{ \Delta^n_{2j-2} \wt L^{\mathrm{v}} (\wt u^n_{j-1})}{\Lf(\call^n_{j-1}(\wt u^n_{j-1}))}\biggr\}\Re\biggl\{\frac{ \Delta^n_{2j} \wt L^{\mathrm{v}} (\wt u^n_{j})}{\Lf(\call^n_j(\wt u^n_j))}\biggr\}.
	\end{split}\end{equation}
	To this end, 
	we further introduce
	\begin{equation}\label{eq:V-2} 
		\ov V^{\prime n}=\frac{k_n}{\sqrt{\lvert \calj_n\rvert}}\sum_{j\in\calj_n}\Delta^n_{2j} \ov \cf( \wt{u}^n_j) \Delta^n_{2j-2} \ov \cf(\wt u^n_{j-1}),
	\end{equation}
	where  $\Delta^n_j \ov \cf(u)=\ov \cf^n_j(u)-\ov \cf^n_{j-1}(u)$ and
	\begin{equation}\label{eq:cbar}\begin{split}
			\ov \cf^n_j(u)&=\log\log \lvert \wt L^{n,\mathrm{s}}_j(u)\rvert^{-1} + \Re \biggl\{ \frac{\wt L^{n,\mathrm{v}}_j(u)+\wt L^{n,\mathrm{b}}_j(u)}{  \Lf(\wt L^{n,\mathrm{s}}_j(u)) }\biggr\}\\
			&\quad-\frac12\Re\biggl\{\frac{(\wt L^{n,\mathrm{v}}_j(u)+\wt L^{n,\mathrm{b}}_j(u))^2}{\wt L^{n,\mathrm{s}}_j(u)\Lf(\wt L^{n,\mathrm{s}}_j(u))} \biggr\}-\frac12\biggl(\Re\biggl \{ \frac{\wt L^{n,\mathrm{v}}_j(u)+\wt L^{n,\mathrm{b}}_j(u)}{\Lf(\wt L^{n,\mathrm{s}}_j(u)) }\biggr\}\biggr)^2.
	\end{split}\end{equation}
	
	Recalling the set $\Om_n$ from Lemma~\ref{lem:Om},
	we decompose
	$  V^{\prime n}	=\ov V^{\prime n}+B^n_1+B^n_2 + B^n_3$,
	where 
	\begin{align*}
		B^n_1&= \frac{k_n}{\sqrt{\lvert \calj_n\rvert}}\sum_{j\in\calj_n}(\Delta^n_{2j} \wt \cf( \wt{u}^n_j)-\Delta^n_{2j} \ov \cf( \wt{u}^n_j))  \Delta^n_{2j-2} \wt \cf(\wt u^n_{j-1})\bone_{\Om_n}, \\
		B^n_2	&=\frac{k_n}{\sqrt{\lvert \calj_n\rvert}}\sum_{j\in\calj_n}\Delta^n_{2j} \ov \cf( \wt{u}^n_j) ( \Delta^n_{2j-2} \wt \cf(\wt u^n_{j-1})- \Delta^n_{2j-2} \ov \cf(\wt u^n_{j-1}))\bone_{\Om_n}
	\end{align*}
	and $B^n_3=(  V^{\prime n}-\ov V^{\prime n})\bone_{\Om_n^c}$.  By Lemma~\ref{lem:Om}, we have $  B^n_3 \to 0$ in probability. For the other two terms,  we may assume that we are on the set $\Om_n$. In fact, we will tacitly do so for the remainder of this section and write  $\E[X\bone_{\Om_n}]$ for $\E[X]$. Also, we will often use the fact that $o'(1)/\theta_n^4\to0$ by (\ref{eq:rates}). On $\Om_n$, we have $\lvert\Lf( \wt L^{n,\rm s}_{2j-\ell}(\wt u^n_j))\rvert\geq C\theta_n^2$ and $\lvert  \wt  L^{n,\mathrm{v}}_{2j-\ell}(\wt u^n_j)\rvert\vee\lvert  \wt  L^{n,\mathrm{b}}_{2j-\ell}(\wt u^n_j)\rvert\leq \Den^{1/9}$ for all  $\ell\in\{0,1\}$ and $j=1,\dots, \lfloor T/(2p_n\Den)\rfloor$. So if $n$ is large, we have $\lvert (\wt  L^{n,\mathrm{v}}_{2j-\ell}(\wt u^n_j)+\wt  L^{n,\mathrm{b}}_{2j-\ell}(\wt u^n_j))/\Lf(\wt  L^{n,\mathrm{s}}_{2j-\ell}(\wt u^n_j))\rvert \leq \frac12$. Upon realizing that $\ov\cf^n_j(u)$ is a second-order expansion of $\wt \cf^n_j(u)=\log\log \lvert \wt L^n_j(u)\rvert^{-1}=\log \Re \Log \wt L^n_j(u)^{-1}$ around $\wt L^{n,\rm s}_j(u)$, we can use Lemma~\ref{lem:bounded} to derive
	\begin{equation}\label{eq:c-expand}\begin{split}
			\ov	\E^n_{[2j-2]^m_n}	\bigl[\lvert\wt \cf^n_{2j-\ell}(\wt u^n_j)-\ov \cf^n_{2j-\ell}(\wt u^n_j)\rvert^p\bigr]&\leq C\ov\E^n_{[2j-2]^m_n}\Biggl[\biggl\lvert\frac{\wt  L^{n,\mathrm{v}}_{2j-\ell}(\wt u^n_j)+\wt  L^{n,\mathrm{b}}_{2j-\ell}(\wt u^n_j)}{\Lf(\wt  L^{n,\mathrm{s}}_{2j-\ell}(\wt u^n_j))}\biggr\rvert^{3p}\Biggr] \\
			&\leq C(k_n^{-3p/2}+\Den^{3p/2})/\theta_n^{6p}\leq Ck_n^{-3p/2}/\theta_n^{6p}.
		\end{split}
	\end{equation}
	Similarly,   for large enough $n$ and because $p_n\Den\leq 1/k_n$ by (\ref{eq:rates}), we have
	\[ \ov\E^n_{[2j-2]^m_n}\bigl[ \lvert\Delta^n_{2j} \wt \cf(\wt u^n_j)\rvert^2\bigr] \leq \ov\E^n_{[2j-2]^m_n}\biggl[ \biggl(\frac{\lvert\Delta^n_{2j} \wt  L(\wt u^n_j)\rvert}{\textstyle\bigwedge_{\ell=0,1} \lvert \Lf( \wt  L^n_{2j-\ell}(\wt u^n_j))\rvert }\biggr)^2\biggr]\leq Ck_n^{-1}/\theta_n^4. \]
	Analogously, $\ov\E^n_{[2j-2]^m_n}\bigl[ \lvert\Delta^n_{2j} \ov \cf(\wt u^n_j)\rvert^2\bigr] \leq Ck_n^{-1}/\theta_n^4$.  
	Hence, by the Cauchy--Schwarz inequality, we have
	$\E [ \lvert B^n_1\rvert ] \vee \E [\lvert B^n_2 \rvert ]\leq C\times k_n\sqrt{ p_n\Den} \times    1/ (p_n\Den)   \times k_n^{-1/2}/\theta_n^2\times k_n^{-3/2}/\theta_n^6\leq C/\sqrt{k_n^3\Den\theta_n^{16}}$, which  converges to $0$ by (\ref{eq:rates}).
	
	Next, 	we want to replace $ \Delta^n_{2j} \ov \cf(\wt u^n_j)$   by a simpler expression $ \Delta^n_{2j}\ov \cf(\wt u^n_j)_{\rm{s}}$ ($\rm{s}$ for ``simplified''), to be determined in several steps in the following, and accordingly replace $\ov V^{\prime n}$ by 
	\begin{equation}\label{eq:Vbar-s} 
		\ov V^{\prime n}_{\rm s}= \frac{k_n}{\sqrt{\lvert \calj_n\rvert}}\sum_{j\in\calj_n} \Delta^n_{2j}\ov{\cf}( \wt{u}^n_j )_{\rm s} \Delta^n_{2j-2}\ov{\cf}( \wt{u}^n_{j-1} )_{\rm s},
	\end{equation} 
	Clearly,
	\begin{align*}
		\ov V^{\prime n}-\ov V^{\prime n}_{\rm s}&=\frac{k_n}{\sqrt{\lvert \calj_n\rvert}}\sum_{j\in\calj_n}\Bigl[( \Delta^n_{2j}\ov{\cf}( \wt{u}^n_j )- \Delta^n_{2j}\ov{\cf}( \wt{u}^n_j )_{\rm s}) \Delta^n_{2j-2}\ov{\cf}( \wt{u}^n_{j-1} )\\
		&\quad+  \Delta^n_{2j}\ov{\cf}( \wt{u}^n_j )_{\rm s} ( \Delta^n_{2j-2}\ov{\cf}( \wt{u}^n_{j-1} )- \Delta^n_{2j-2}\ov{\cf}( \wt{u}^n_{j-1} )_{\rm s})\Bigr],
	\end{align*}
	where $ \Delta^n_{2j-2}\ov{\cf}( \wt{u}^n_{j-1} )$ is and $ \Delta^n_{2j-2}\ov{\cf}( \wt{u}^n_{j-1} )_{\rm{s}}$ will be chosen as $\calh^n_{(2j-3)p_n+k_n}$-measurable. 
	Since $\ov\E^n_{[2j-2]^m_n}[ \lvert \Delta^n_{2j}\ov{\cf}( \wt{u}^n_j )\rvert]=O(k_n^{-1/2}/\theta_n^2)$  and $(2j-3)p_n+k_n\leq [2j-2]^m_n$ by (\ref{eq:rates}), 
	we only have to make sure   $\Delta^n_{2j}\ov \cf(\wt u^n_j)_{\rm{s}}$  is chosen such that $\ov\E^n_{[2j-2]^m_n}[ \lvert \Delta^n_{2j}\ov{\cf}( \wt{u}^n_j )_{\rm s}\rvert]\leq Ck_n^{-1/2}/\theta_n^2$ on the one hand  and $\ov\E^n_{[2j-2]^m_n}[\lvert \Delta^n_{2j}\ov{\cf}( \wt{u}^n_j )- \Delta^n_{2j}\ov{\cf}( \wt{u}^n_j )_{\rm{s}}\rvert]=o(\sqrt{\Den}\theta_n^2)$ on the other hand. Then the last display will  converge  to $0$ in $L^1$.
	
	To get the first version of $\Delta^n_{2j}\ov \cf(\wt u^n_j)_{\rm{s}}$, we expand the increment $ \Delta^n_{2j}\ov \cf(\wt u^n_j)$ using \eqref{eq:cbar} and omit all terms that are $o(\sqrt{\Den}\theta_n^2)$. Simplifying terms, we arrive at $	\Delta^n_{2j}\ov \cf(\wt u^n_j)_{\rm{s}}=  \Delta^n_{2j}\ov \cf(\wt u^n_j)^{\rm I}_{\rm s}+\Delta^n_{2j}\ov \cf(\wt u^n_j)^{\rm II}_{\rm s}-\Delta^n_{2j}\ov \cf(\wt u^n_j)^{\rm III}_{\rm s}-\Delta^n_{2j}\ov \cf(\wt u^n_j)^{\rm IV}_{\rm s}-\Delta^n_{2j}\ov \cf(\wt u^n_j)^{\rm V}_{\rm s}$, where
	\begin{align*}
		\Delta^n_{j}\ov \cf(u)_{\rm{s}}^{\rm I}&=\Re\Biggl\{\frac{\Delta^n_{j}\wt L^{\mathrm{s}}(u)}{\Lf(\wt L^{n,\mathrm{s}}_{j-1}(u))}\Biggr\}, \qquad\Delta^n_{j}\ov \cf(u)_{\rm{s}}^{\rm II}=  \Re\Biggl\{\frac{\Delta^n_{j} \wt L^{\mathrm{v}}(u)}{\Lf(\wt L^{n,\mathrm{s}}_{j-1}(u))}\Biggr\},\\
		\Delta^n_{j}\ov \cf(u)_{\rm{s}}^{\rm III}&= \Re\Biggl\{ \frac{(\wt L^{n,\rm v}_{j}(u))^2-(\wt L^{n,\rm v}_{j-1}(u))^2}{2\wt L^{n,\mathrm{s}}_{{j}-1}(u)\Lf(\wt L^{n,\mathrm{s}}_{j-1}(u))}\Biggr\}, \qquad\Delta^n_{j}\ov \cf(u)_{\rm{s}}^{\rm IV}=\Re\Biggl\{ \frac{\Delta^n_{j}\wt L^{\mathrm{s}}(u)\wt L^{n,\rm v}_{j}(u)}{(\Lf(\wt L^{n,\mathrm{s}}_{j-1}(u)))^2}\Biggr\}, \\
		\Delta^n_{j}\ov \cf(u)_{\rm{s}}^{\rm V}&= \frac12\biggl[\biggl( \Re\biggl\{\frac{\wt L^{n,\mathrm{v}}_{j}(u)}{\Lf(\wt L^{n,\mathrm{s}}_{j-1}(u))}\biggr\}\biggr)^2-\biggl(\Re\biggl\{\frac{\wt L^{n,\mathrm{v}}_{j-1}(u)}{\Lf(\wt L^{n,\mathrm{s}}_{j-1}(u))}\biggr\}\biggr)^2\biggr] .
	\end{align*}

	
	We start by investigating the contribution of   $\Delta^n_{2j}\ov \cf(\wt u^n_j)_{\rm{s}}^{\rm I}$.  Our goal is to show that 
	$
	D^n_1=\frac{k_n}{\sqrt{\lvert \calj_n\rvert}}\sum_{j\in\calj_n}  \Delta^n_{2j-2}\ov\cf (\wt u^n_{j-1} )_{\rm s} \Delta^n_{2j}\ov \cf(\wt u^n_j)_{\rm{s}}^{\rm I} 
	$
	is negligible. To this end, it suffices to show that
	\begin{equation*}
		\wt D^n_1=\frac{k_n}{\sqrt{\lvert \calj_n\rvert}}\sum_{j\in\calj_n}  \Delta^n_{2j-2}\ov\cf (\wt u^n_{j-1} )_{\rm s} \ov \E^n_{(2j-2)p_n}[\Delta^n_{2j}\ov \cf(\wt u^n_j)_{\rm{s}}^{\rm I}]
	\end{equation*}
	vanishes asymptotically. Indeed, since $  \Delta^n_{2j-2}\ov\cf (\wt u^n_{j-1} )_{\rm s}$ is $\calh^n_{(2j-2)p_n}$-measurable, the $j$th term that appears in the difference $D^n_1-\wt D^n_1$ will be $\calh^n_{2jp_n}$-measurable with a zero $\calh^n_{(2j-2)p_n}$-conditional expectation. Therefore, $D^n_1-\wt D^n_1$ is a martingale sum and  consequently, we have $\E[\lvert D^n_1-\wt D^n_1\rvert^2]= O(k_n^2p_n\Den\times 1/(p_n\Den)\times k_n^{-1}/\theta_n^4\times p_n\Den/\theta_n^4)=O(k_n^2\Den/\theta_n^8)=o(1)$ by (\ref{eq:rates}). Next, consider the term $ \ov \E^n_{(2j-2)p_n}[\Delta^n_{2j}\ov \cf(\wt u^n_j)_{\rm{s}}^{\rm I}]=   \E^n_{(2j-2)p_n}[\Delta^n_{2j}\ov \cf(\wt u^n_j)_{\rm{s}}^{\rm I}]$. Recalling \eqref{eq:XY}, we can use \eqref{eq:vp-conv} and \eqref{eq:psi} to show that 
	\begin{equation}\label{eq:aux25} 
		1-Y^{\prime n}(\wt u^n_j)_{(i-1)\Den}=O(\sqrt{\Den}).
	\end{equation}
	Moreover, $Y^n$ is $\frac12$-H\"older continuous in squared mean, so we have $\ov \E^n_{(2j-2)p_n}[\lvert\wt L^{n,\rm s}_{2j-1}(\wt u^n_j) - Y^n(\wt u^n_j)_{(2j-2)p_n\Den}\rvert^2]\leq Cp_n\Den$. Therefore, instead of $\wt D^n_1$, 
	it is enough to consider
	\begin{align*}
		\wh D^n_1&=\frac{k_n}{\sqrt{\lvert \calj_n\rvert}}\sum_{j\in\calj_n}  \Delta^n_{2j-2}\ov\cf (\wt u^n_{j-1} )_{\rm s} \Re\biggl\{\frac{  \E^n_{(2j-2)p_n} [\Delta^n_{2j}\wt L^{\mathrm{s}}(\wt u^n_j)]}{\Lf(Y^n(\wt u^n_j)_{(2j-2)p_n\Den})}\biggr\}.
	\end{align*}
	Note that $X^n(u)$ from \eqref{eq:XY} is an It\^o semimartingale with bounded coefficients (uniformly in sufficiently small $u$ such that $\Psi(uK)\geq\frac12$). Thus,
	by \eqref{eq:Delta-Ls}, \eqref{eq:aux23} and a first-order expansion, 
	\begin{align*}
		\E^n_{(2j-2)p_n} [\Delta^n_{2j} \wt L^{\mathrm{s}}(\wt u^n_j)]&	= \frac{Y^n(\wt u^n_j)_{(2j-2)p_n\Den}}{k_n}\sum_{i=(2j-1)p_n+1}^{(2j-1)p_n+k_n} \E^n_{(2j-2)p_n} [(X^n(\wt u^n_j)_{(i-1)\Den}\\
		&\quad -X^n(\wt u^n_j)_{(i-p_n-1)\Den})] +O_p(p_n\Den)  =  O_p(p_n\Den),
	\end{align*}
	where $O_p(p_n\Den)$ signifies a term whose $\calf_{(2j-2)p_n\Den}$-conditional $L^1$-norm is $O(p_n\Den)=o'(\sqrt{\Den})$, uniformly in $\om$. This shows that $\wh D^n_1$ is negligible.

	We proceed to $\Delta^n_{2j-2} \ov \cf(\wt   u^n_{j-1})_{\rm{s}}^{\rm I}$ and consider
	\begin{equation}\label{eq:aux2} 
		\frac{k_n}{\sqrt{\lvert \calj_n\rvert}}\sum_{j\in\calj_n} \Re\biggl\{ \frac{\Delta^n_{2j-2} \wt L^{\mathrm{s}}(\wt   u^n_{j-1})}{\Lf(\wt L^{n,\mathrm{s}}_{2j-3}(\wt   u^n_{j-1}) )}\biggr\}\Delta^n_{2j}\ov{\cf} (\wt u^n_j)_{\rm{s}}.
	\end{equation}
	Revisiting the proof of \eqref{eq:L'''-3}, one can actually show   $\E^n_{[2j-2]^m_n}[\lvert\Delta^n_{2j} \wt L^{\mathrm{s}}(\wt   u^n_{j})\rvert^2]\leq Cp_n\Den\theta_n^2$.\footnote{This is because we can improve the right-hand side of \eqref{eq:aux22} to $Cp_n\Den\theta_n^2$ if $U=\wt u^n_j$.}
	Since $\E[\lvert\Delta^n_{2j}\ov{\cf} (\wt u^n_j)_{\rm{s}}^{\rm IV}\rvert^2]^{1/2}\leq C\sqrt{p_n\Den/(k_n\theta_n^8)}\leq C\sqrt{\Den}/\theta_n^4$, its contribution to \eqref{eq:aux2} is
	a term whose $L^1$-norm is   $O( k_n\sqrt{p_n\Den}\times  1/(p_n\Den)\times\sqrt{p_n\Den}\times \sqrt{ \Den }/\theta_n^4)=O(\sqrt{k_n^2\Den/\theta_n^8})$, which is negligible by (\ref{eq:rates}). Since  $\Delta^n_{2j}\ov \cf(\wt u^n_j)_{\rm{s}}^{\rm I}$ can be omitted  as we have shown, we can replace $\Delta^n_{2j}\ov \cf(\wt u^n_j)_{\rm{s}}$ in \eqref{eq:aux2} by $\Delta^n_{2j}\ov \cf(\wt u^n_j)_{\rm{s}}^{\rm II}+\Delta^n_{2j}\ov \cf(\wt u^n_j)_{\rm{s}}^{\rm III}$. As above, because $Y^n(u)$ is an It\^o semimartingale with bounded coefficients (for $u$ sufficiently small) and $1-Y^{\prime n}(\wt u^n_j)_{(i-1)\Den}=O(\sqrt{\Den})$, we have $\E^n_{[2k-2]^m_n}[ \lvert\wt L^{n,\mathrm{s}}_{2k-\ell}(\wt u^n_k) - Y^n(\wt u^n_k)_{[2k-2]^m_n\Den}\rvert^2]^{1/2}\leq C\sqrt{p_n\Den}$ for $\ell\in\{0,1\}$ and $k\in\{j-1,j\}$. As a consequence, instead of \eqref{eq:aux2}, it is sufficient to further analyze
	\begin{equation}\label{eq:aux4} \begin{split} 
			&\frac{k_n}{\sqrt{\lvert \calj_n\rvert}}\sum_{j\in\calj_n}\Re\biggl\{\frac{\Delta^n_{2j-2} \wt L^{\mathrm{s}}(\wt u^n_{j-1})}{\Lf(Y^n(\wt u^n_{j-1})_{[2j-4]^m_n\Den})}\biggr\}\\
			&\quad \times \Biggl(\Re\biggl\{\frac{\Delta^n_{2j} \wt L^{\mathrm{v}}(\wt u^n_j)}{\call(Y^n(\wt u^n_j)_{[2j-2]^m_n\Den})}  +\frac{(\wt L^{n,\mathrm{v}}_{2j-1}(\wt u^n_j))^2-(\wt L^{n,\mathrm{v}}_{2j}(\wt u^n_j))^2}{2Y^n(\wt u^n_{j})_{[2j-2]^m_n\Den}\Lf(Y^n(\wt u^n_{j})_{[2j-2]^m_n\Den})}\biggr\}\\
			&\quad     +\frac12\biggl[\biggl(\Re\biggl\{\frac{\wt L^{n,\rm v}_{2j}(\wt u^n_j)}{\Lf(Y^n(\wt u^n_{j})_{[2j-2]^m_n\Den})}\biggr\}\biggr)^2-\biggl(\Re\biggl\{\frac{\wt L^{n,\rm v}_{2j-1}(\wt u^n_j)}{\Lf(Y^n(\wt u^n_{j})_{[2j-2]^m_n\Den})}\biggr\}\biggr)^2\biggr]\Biggr). \end{split}
	\end{equation}
	Since  $\ov\E^n_{[2j-2]^m_n\Den}[ \Delta^n_{2j}\wt L^{\mathrm{v}}(\wt u^n_j)]=0$, 
	a martingale argument (along the lines of the analysis of $D^n_1$) reduces \eqref{eq:aux4} to 
	\begin{equation}\begin{split}
			&	\frac{k_n}{\sqrt{\lvert \calj_n\rvert}}\sum_{j\in\calj_n}\Re\biggl\{\frac{\Delta^n_{2j-2} \wt L^{\mathrm{s}}(\wt u^n_{j-1})}{\Lf(Y^n(\wt u^n_{j-1})_{[2j-4]^m_n\Den})}\biggr\} \Biggl(\Re\biggl\{ \frac{\ov\E^n_{[2j-2]^m_n}[ (\wt L^{n,\mathrm{v}}_{2j-1}(\wt u^n_j))^2-(\wt L^{n,\mathrm{v}}_{2j}(\wt u^n_j))^2]}{2Y^n(\wt u^n_{j})_{[2j-2]^m_n\Den}\Lf(Y^n(\wt u^n_{j})_{[2j-2]^m_n\Den})}\biggr\}\\
			&\quad +\frac12\ov\E^n_{[2j-2]^m_n}\biggl[\biggl(\Re\biggl\{\frac{\wt L^{n,\rm v}_{2j}(\wt u^n_j)}{\Lf(Y^n(\wt u^n_{j})_{[2j-2]^m_n\Den})}\biggr\}\biggr)^2-\biggl(\Re\biggl\{\frac{\wt L^{n,\rm v}_{2j-1}(\wt u^n_j)}{\Lf(Y^n(\wt u^n_{j})_{[2j-2]^m_n\Den})}\biggr\}\biggr)^2\biggr]\Biggr).\end{split}\label{eq:aux5}
	\end{equation}
	Now, recall the notation $u_n=u/\Den^{\varpi\wedge 1/2}$ and observe that for $\ell\in\{0,1\}$,
	\begin{equation}\label{eq:aux9}\begin{split}
			&\ov	\E^n_{[2j-2]^m_n}[(\wt L^{n,\mathrm{v}}_{2j-\ell}(\wt u^n_j))^2]	=\frac1{k_n^2} \sum_{i=(2j-\ell-1)p_n+1}^{(2j-\ell-1)p_n+k_n} \ov\E^n_{[2j-2]^m_n}\Bigl[ \bigl( e^{i(\wt u^n_j)_n\Delta^n_i y } - \E^n_{i-1}[e^{i(\wt u^n_j)_n\Delta^n_i y}]\bigr)^2 \Bigr]\\
			&\quad+\frac2{k_n^2} \sum_{i'<i} \ov\E^n_{[2j-2]^m_n}\Bigl[ \bigl( e^{i(\wt u^n_j)_n\Delta^n_i y } - \E^n_{i-1}[e^{i(\wt u^n_j)_n\Delta^n_i y}]\bigr)  \bigl( e^{i(\wt u^n_j)_n\Delta^n_{i'} y } - \E^n_{i-1}[e^{i(\wt u^n_j)_n\Delta^n_{i'} y}]\bigr)\Bigr],
		\end{split}
	\end{equation}
	where the second sum is taken over all $(2j-\ell-1)p_n+1\leq i'<i\leq (2j-\ell-1)p_n+k_n$. The covariance term in \eqref{eq:aux9} is zero whenever $\lvert i-i'\rvert\geq m+2$.
	We want to replace $e^{i(\wt u^n_j)_n\Delta^n_i y} = e^{\Delta^n_i \Theta((\wt u^n_j)_n)}Z^n_i((\wt u^n_j)_n)_{i\Den}e^{i(\wt u^n_j)_n\Delta^n_i \eps}$ everywhere in \eqref{eq:aux9} by   $\wh \xi^{n,i}_j=e^{\Delta \Theta((\wt u^n_j)_n)^n_{[2j-2]^m_n}}Z^n_i((\wt u^n_j)_n)_{i\Den} \times e^{i(\wt u^n_j)_n\Delta^n_{i,[2j-2]^m_n} \eps}$, and similarly for the terms indexed by $i'$. As 
	\begin{equation}\label{eq:aux27}\begin{split}
			\ov	\E^n_{[2j-2]^m_n}\Bigl[\bigl\lvert ( e^{\Delta^n_i \Theta((\wt u^n_j)_n)}-e^{\Delta \Theta((\wt u^n_j)_n)^n_{[2j-2]^m_n}})Z^n_i((\wt u^n_j)_n)_{i\Den}e^{i(\wt u^n_j)_n\Delta^n_i \eps}\bigr\rvert^2\Bigr] &\leq C {p_n\Den}, \\
			\ov	\E^n_{[2j-2]^m_n}\Bigl[\bigl\lvert e^{\Delta \Theta((\wt u^n_j)_n)^n_{[2j-2]^m_n}}Z^n_i((\wt u^n_j)_n)_{i\Den}(e^{i(\wt u^n_j)_n\Delta^n_i \eps}-e^{i(\wt u^n_j)_n\Delta^n_{i,[2j-2]^m_n} \eps})\bigr\rvert^2\Bigr] &\leq C {p_n\Den}
	\end{split}\end{equation} with some $C$ that is uniform in $\om$ and $j$,\footnote{If $\varpi\leq \frac34$, the second bound follows from the assumption that  $\rho$ is an It\^o semimartingale with bounded coefficients. If $\varpi>\frac34$, we use the fact that $u_n\Den^{\varpi} = \Den^{\varpi-1/2}=o(\sqrt{p_n\Den})$ by (\ref{eq:rates}).}  we have
	\begin{equation}\label{eq:aux6-2}\begin{split}
			&\ov	\E^n_{[2j-2]^m_n}[(\wt L^{n,\mathrm{v}}_{2j-\ell}(\wt u^n_j))^2] =	\frac1{k_n^2} \sum_{i=(2j-\ell-1)p_n+1}^{(2j-\ell-1)p_n+k_n} \ov\E^n_{[2j-2]^m_n}\Bigl[ \bigl( \wh \xi^{n,i}_j-\E^n_{i-1}[\wh \xi^{n,i}_j]\bigr)^2\Bigr]\\
			&\quad+\frac{2}{k_n^2} \sum_{i'<i}\ov  \E^n_{[2j-2]^m_n}\Bigl[ \bigl( \wh \xi^{n,i}_j-\E^n_{i-1}[\wh \xi^{n,i}_j]\bigr)\bigl( \wh \xi^{n,i'}_j-\E^n_{i-1}[\wh \xi^{n,i'}_j]\bigr)\Bigr]+O(\sqrt{\Den/k_n}).
	\end{split}\end{equation}
	Since $\E^n_{i-1}[\wh\xi^{n,i}_j] = e^{\Delta \Theta((\wt u^n_j)_n)^n_{[2j-2]^m_n}}\Psi(\wt u^n_j \Den^{(\varpi-1/2)_+} \rho_{[2j-2]^m_n\Den})$, we have
	\begin{equation}\label{eq:aux6}\begin{split}
			&\ov\E^n_{[2j-2]^m_n} [  ( \wh \xi^{n,i}_j-\E^n_{i-1}[\wh \xi^{n,i}_j] )^2 ] \\
			& \quad =e^{2\Delta \Theta((\wt u^n_j)_n)^n_{[2j-2]^m_n}}\ov\E^n_{[2j-2]^m_n} [ e^{2i(\wt u^n_j)_n\Delta^n_{i,[2j-2]^m_n} \eps}( Z^n_i((\wt u^n_j)_n)_{i\Den}-1)^2]\\
			& \qquad+e^{2\Delta \Theta((\wt u^n_j)_n)^n_{[2j-2]^m_n}}\ov\E^n_{[2j-2]^m_n}[(e^{i(\wt u^n_j)_n\Delta^n_{i,[2j-2]^m_n} \eps}-\Psi(\wt u^n_j \Den^{(\varpi-1/2)_+} \rho_{[2j-2]^m_n\Den}))^2]\\
			& \qquad+e^{2\Delta \Theta((\wt u^n_j)_n)^n_{[2j-2]^m_n}}\ov\E^n_{[2j-2]^m_n}[e^{i(\wt u^n_j)_n\Delta^n_{i,[2j-2]^m_n} \eps}( Z^n_i((\wt u^n_j)_n)_{i\Den}-1)\\
			&\qquad \qquad\qquad\qquad\qquad\times(e^{i(\wt u^n_j)_n\Delta^n_{i,[2j-2]^m_n} \eps}-\Psi(\wt u^n_j \Den^{(\varpi-1/2)_+} \rho_{[2j-2]^m_n\Den}))].
	\end{split} \end{equation}
	Recall that $Z(u,s)_t=e^{iu(x_t-x_s)-(\Theta(u)_t-\Theta(u)_s)}$ is a martingale. By    taking conditional expectation with respect to $\calf_{[2j-2]^m_n\Den}\vee \calg_\infty$ and using the fact that $\ov\E^n_{[2j-2]^m_n}[Z^n_i((\wt u^n_j)_n)_{i\Den}]= \E^n_{[2j-2]^m_n}[Z^n_i((\wt u^n_j)_n)_{i\Den}]=1$, one can show that the third term is   zero, while the first one   on the right-hand side of \eqref{eq:aux6} is equal to $e^{2\Delta \Theta((\wt u^n_j)_n)^n_{[2j-2]^m_n}}\Psi(2\wt u^n_j \Den^{(\varpi-1/2)_+} \rho_{[2j-2]^m_n\Den})$ times
	\begin{equation}\label{eq:aux26}\begin{split}
			&  \E^n_{[2j-2]^m_n} [  ( Z^n_i((\wt u^n_j)_n)_{i\Den}-1)^2]=	\E^n_{[2j-2]^m_n} [( Z^n_i((\wt u^n_j)_n)_{i\Den})^2 +1 -2   Z^n_i((\wt u^n_j)_n)_{i\Den}]\\
			&  \quad= 	\E^n_{[2j-2]^m_n} [e^{2i(\wt u^n_j)_n\Delta^n_i x -2\Delta^n_i\Theta((\wt u^n_j)_n)}]+1 -2  \E^n_{[2j-2]^m_n} [ Z^n_i((\wt u^n_j)_n)_{i\Den}]\\
			& \quad=e^{-2\Delta\Theta((\wt u^n_j)_n)^n_{[2j-2]^m_n}} \E^n_{[2j-2]^m_n} [e^{2i(\wt u^n_j)_n\Delta^n_i x}]-1 + O(\sqrt{p_n\Den})\\
			&  \quad=e^{-2\Delta\Theta((\wt u^n_j)_n)^n_{[2j-2]^m_n}} \E^n_{[2j-2]^m_n} [e^{\Delta^n_i \Theta(2(\wt u^n_j)_n)}Z^n_i(2(\wt u^n_j)_n)_{i\Den}]-1 + O(\sqrt{p_n\Den})\\
			& \quad=e^{\Delta\Theta(2(\wt u^n_j)_n)^n_{[2j-2]^m_n}-2\Delta\Theta((\wt u^n_j)_n)^n_{[2j-2]^m_n}}-1+ O(\sqrt{p_n\Den}),
	\end{split} \end{equation}
	where the $O$-term is uniform in $\om$ and $j$. Note that the last line (apart from the $O(\sqrt{p_n\Den})$-term) is independent of $i$ (and hence, of $\ell$), and so is the second term on the right-hand side of  \eqref{eq:aux6} since $\chi$ is a stationary sequence. A similar argument shows that the same is true for the nondiagonal terms in \eqref{eq:aux9}, which means that
	%
	%
	\begin{equation}\label{eq:diff} 
		\sup_{\om\in\Om}\sup_{j=2,\dots,\lfloor T/(2p_n\Den)\rfloor}\Bigl\lvert\ov\E^n_{[2j-2]^m_n}\Bigl[(\wt L^{n,\mathrm{v}}_{2j-1}(\wt u^n_j))^2-(\wt L^{n,\mathrm{v}}_{2j}(\wt u^n_j))^2\Bigr]\Bigr\rvert= O(\sqrt{\Den/k_n}).
	\end{equation}
	A similar argument proves that
	\begin{equation}\label{eq:diff-2}\begin{split}
			&\sup_{\om\in\Om}\sup_{j=2,\dots,\lfloor T/(2p_n\Den)\rfloor}\biggl\lvert\ov\E^n_{[2j-2]^m_n}\biggl[\biggl(\Re\biggl\{\frac{\wt L^{n,\rm v}_{2j}(\wt u^n_j)}{\Lf(Y^n(\wt u^n_{j})_{[2j-2]^m_n\Den})}\biggr\}\biggr)^2\\
			&	\qquad\qquad\qquad\qquad\qquad-\biggl(\Re\biggl\{\frac{\wt L^{n,\rm v}_{2j-1}(\wt u^n_j)}{\Lf(Y^n(\wt u^n_{j})_{[2j-2]^m_n\Den})}\biggr\}\biggr)^2\biggr]\biggr\rvert = O(\sqrt{\Den/k_n}/\theta_n^4). \end{split}
	\end{equation}
	Therefore, the $L^1$-norm of \eqref{eq:aux5} is $O(\sqrt{k_n\Den/\theta_n^8})$, which  
	shows that $\Delta^n_{2j-2} \ov \cf(\wt u^n_{j-1})_s^{\rm I}$ has no asymptotic contribution. In other words,  we have $\E[  \lvert \ov V^{\prime n} -\ov V^{\prime n}_{\rm s}\rvert]\to0$ where $\ov V_{\rm s}^n$ is given by \eqref{eq:Vbar-s} and $	\Delta^n_{2j} \ov \cf(\wt u^n_j)_{\rm{s}}=	\Delta^n_{2j} \ov \cf(\wt u^n_j)_{\rm{s}}^{\rm II}-	\Delta^n_{2j} \ov \cf(\wt u^n_j)_{\rm{s}}^{\rm III}-	\Delta^n_{2j} \ov \cf(\wt u^n_j)_{\rm{s}}^{\rm IV}-\Delta^n_{2j} \ov \cf(\wt u^n_j)_{\rm{s}}^{\rm V}$.
	
	Our next goal is to show that $\Delta^n_{2j}\ov\cf(\wt u^n_j)^{\rm IV}_{\rm s}$ and $\Delta^n_{2j-2}\ov\cf(\wt u^n_{j-1})^{\rm IV}_{\rm s}$ can also be omitted, that is, 
	\begin{align*}
		&\frac{k_n}{\sqrt{\lvert \calj_n\rvert}}\sum_{j\in\calj_n} \Re\Biggl\{
		\frac{\wt L^{n,\mathrm{v}}_{2j-2}(\wt u^n_j)\Delta^n_{2j-2} \wt L^{\mathrm{s}}(\wt u^n_{j-1})}{(\Lf(\wt L^{n,\mathrm{s}}_{2j-3}(\wt u^n_j)))^2} \Biggr\}(\Delta^n_{2j} \ov\cf(\wt u^n_j)^{\rm II}_{\rm s} +\Delta^n_{2j} \ov\cf(\wt u^n_j)^{\rm III}_{\rm s} +\Delta^n_{2j} \ov\cf(\wt u^n_j)^{\rm V}_{\rm s} )\\
		&\quad+\frac{k_n}{\sqrt{\lvert \calj_n\rvert}}\sum_{j\in\calj_n}
		\Delta^n_{2j-2}\ov \cf(\wt u^n_{j-1})_{\rm s}\Re\Biggl\{
		\frac{\wt L^{n,\mathrm{v}}_{2j}(\wt u^n_j)\Delta^n_{2j} \wt L^{\mathrm{s}}(\wt u^n_j)}{(\Lf(\wt L^{n,\mathrm{s}}_{2j-1}(\wt u^n_j)))^2} \Biggr\}
	\end{align*}
	is asymptotically negligible. By Lemma~\ref{lem:bounded}, one can readily see that $\Delta^n_{2j} \ov\cf(\wt u^n_j)^{\rm III}_{\rm s} +\Delta^n_{2j} \ov\cf(\wt u^n_j)^{\rm V}_{\rm s}$ do not contribute in the limit as $n\to\infty$. Furthermore, similarly to how we obtained \eqref{eq:aux4},
	we can replace $\wt L^{n,\mathrm{s}}_{2k-\ell}(\wt u^n_k)$ (for $\ell\in\{0,1\}$ and $k\in\{j-1,j\}$) by $Y^n(\wt u^n_k)_{[2k-2]^m_n\Den}$. A martingale argument further allows us to take $\ov\E^n_{[2j-2]^m_n}[\cdot]$ inside the sum over $j$, which also eliminates $\Delta^n_{2j} \ov\cf(\wt u^n_j)^{\rm II}_{\rm s}$. Hence, only
	\begin{align*}
		\frac{k_n}{\sqrt{\lvert \calj_n\rvert}}\sum_{j\in\calj_n}\Delta^n_{2j-2}\ov{\cf} (\wt u^n_{j-1})_{\rm{s}}\Re \Biggl\{
		\frac{\ov\E^n_{[2j-2]^m_n}[\wt L^{n,\mathrm{v}}_{2j}(\wt u^n_j)\Delta^n_{2j} \wt L^{\mathrm{s}}(\wt u^n_j)]}{(\Lf(Y^n(\wt u^n_j)_{[2j-2]^m_n\Den})^2} \Biggr\}  
	\end{align*}
	has to be studied further. Here, we only need to keep the dominating part of $\Delta^n_{2j-2}\ov \cf(\wt u^n_{j-1})_{\rm s}$ (all higher-order terms only have   negligible contributions). 
	Thus,  we need to analyze  
	\begin{equation}\label{eq:aux7} 
		\frac{k_n}{\sqrt{\lvert \calj_n\rvert}}\sum_{j\in\calj_n} \Re\Biggl\{ \frac{\Delta^n_{2j-2}\wt L^{\mathrm{v}}(\wt u^n_{j-1})}{\Lf(Y^n(\wt u^n_{j-1})_{[2j-4]^m_n\Den})}\Biggr\}\Re \Biggl\{
		\frac{\ov\E^n_{[2j-2]^m_n}[\wt L^{n,\mathrm{v}}_{2j}(\wt u^n_j)\Delta^n_{2j} \wt L^{\mathrm{s}}(\wt u^n_j)]}{(\Lf(Y^n(\wt u^n_j)_{[2j-2]^m_n\Den})^2} \Biggr\}  .
	\end{equation}
	We decompose $	\Delta^n_j \wt L^{\mathrm{s}}(u)=	\Delta^n_j \wt L^{\mathrm{s}}_1(u)+	\Delta^n_j \wt L^{\mathrm{s}}_2(u)+	\Delta^n_j \wt L^{\mathrm{s}}_3(u)$, where
	\begin{align*}
		\Delta^n_j \wt L^{\mathrm{s}}_1(u)&=\frac{1}{k_n}\sum_{i=(j-1)p_n+1}^{(j-1)p_n+k_n}e^{\Den\vp(u_n)_{(i-1)\Den}}\Bigl\{e^{-\frac12 u^2_nc_{(i-1)\Den}\Den+(\Log \Psi(u_n\Den^\varpi\rho_{(i-1)\Den}))\bone_{\{\varpi\leq \frac34\}}}\\
		&\qquad\qquad  \qquad\qquad -e^{-\frac12 u^2_n c_{(i-p_n-1)\Den}\Den+(\Log \Psi(u_n\Den^{\varpi} \rho_{(i-p_n-1)\Den}))\bone_{\{\varpi\leq \frac34\}}}\Bigr\}, \\
		\Delta^n_j \wt	L^{\mathrm{s}}_2(u)&=\frac{1}{k_n}\sum_{i=(j-1)p_n+1}^{(j-1)p_n+k_n}e^{-\frac12 u^2_n c_{(i-p_n-1)\Den}\Den+(\Log \Psi(u_n\Den^\varpi \rho_{(i-p_n-1)\Den}))\bone_{\{\varpi\leq \frac34\}}}\\
		&\qquad\qquad\qquad  \qquad\qquad\qquad\qquad\times(e^{\Den\vp(u_n)_{(i-1)\Den}}-e^{\Den\vp(u_n)_{(i-p_n-1)\Den}}),\\
		\Delta^n_j \wt	L^{\mathrm{s}}_3(u)&=\Delta^n_j \wt	L^{\mathrm{s}}(u)-\Delta^n_j \wt	L^{\mathrm{s}}_1(u)-\Delta^n_j \wt	L^{\mathrm{s}}_2(u).
	\end{align*}
	Since $\Delta^n_j \wt L^{\rm s}_3(u)$ contains the contributions of drift, finite variation jumps and noise (if $\varpi>\frac34$), we have $\E^n_{(2j-2)p_n}[\lvert\Delta^n_{2j}\wt L^{\rm s}_3(\wt u^n_j)\rvert^2]^{1/2}=O(\sqrt{\Den})$ by \eqref{eq:psi}. By \eqref{eq:L2}, we further have the bound $\E^n_{(2j-2)p_n}[\lvert\Delta^n_{2j}\wt L^{\rm s}_2(\wt u^n_j)\rvert^2]^{1/2}=o'(\sqrt{p_n\Den})$. Therefore, in \eqref{eq:aux7}, it suffices to keep  $\Delta^n_{2j}\wt L^{\rm s}_1(\wt u^n_j)$ instead of $\Delta^n_{2j}\wt L^{\rm s}(\wt u^n_j)$.
	Writing $\call^n(u;c,\rho)=e^{-\frac12 u^2_n \Den c+(\Log \Psi(u_n\Den^\varpi\rho))\bone_{\{\varpi\leq 3/4\}}}$, we have that $\Delta^n_{2j}\wt L^{\mathrm{s}}_1(\wt u^n_j)$---up to an error whose $\E^n_{(2j-2)p_n}[(\cdot)^2]^{1/2}$-norm is $o'(\sqrt{p_n\Den})$---is equal to 
	\[
	\frac1{k_n} \sum_{i=(2j-1)p_n+1}^{(2j-1)p_n+k_n} \Bigl(\call^n(\wt u^n_j; c_{(i-1)\Den},\rho_{(i-1)\Den})-\call^n(\wt u^n_j; c_{(i-p_n-1)\Den},\rho_{(i-p_n-1)\Den})\Bigr)\]
	or
	\begin{align*}
		&\frac{1}{k_n} \sum_{i=(2j-1)p_n+1}^{(2j-1)p_n+k_n} \biggl(   \int_{(i-p_n-1)\Den}^{(i-1)\Den} \tfrac{\partial}{\partial c} \call^n(\wt u^n_j;c_s,\rho_s) (\si_s^cdW_s + \ov \si^c_sd\ov W_s ) \\
		&\quad+ \int_{(i-p_n-1)\Den}^{(i-1)\Den} \tfrac{\partial}{\partial \rho} \call^n(\wt u^n_j;c_s,\rho_s) (\si_s^\rho dW_s + \wt \si^\rho_sd\wt W_s )\\
		&\quad +  \iint_{(i-p_n-1)\Den}^{(i-1)\Den} \Bigl(\call^n(\wt u^n_j;c_{s-}+\ga^c(s,z),\rho_{s-}+\ga^\rho(s,z))\\
		&\qquad\qquad\qquad\qquad\qquad\qquad-\call^n(\wt u^n_j;c_{s-} ,\rho_{s-} )  \Bigr)(\mu-\nu)(ds,dz) \biggr),
	\end{align*}
	where the second step follows from It\^o's formula and by ignoring all drift terms of order $o'_p(\sqrt{p_n\Den})$. At the same time, because $\Den^{\varpi-1/2}=o'(\sqrt{p_n\Den})$ by (\ref{eq:rates}) for all $\varpi>\frac34$ and $c$, $\Den\vp(u_n)$ and $\rho$ (if $\varpi\leq \frac34$) are It\^o semimartingales with bounded coefficients, we have 
	$
	e^{i(\wt u^n_j)_n \Delta^n_i y} = e^{\Delta^n_i\Theta((\wt u^n_j)_n)+i(\wt u^n_j)_n\Delta^n_i \eps}Z^n_i((\wt u^n_j)_n)_{i\Den} 
	=e^{\Delta\ov\Theta((\wt u^n_j)_n)^n_{(2j-2)p_n}+i(\wt u^n_j)_n \Delta^n_{i,(2j-2)p_n}\eps}\times Z^n_i((\wt u^n_j)_n)_{i\Den} +o'_p(\sqrt{p_n\Den})$,
	where $\Delta\ov \Theta(u)^n_i = -\frac12 u^2c_{i\Den}\Den + \Den\vp(u)_{i\Den}$. Hence,
	\begin{align*}
		\wt L^{n,\mathrm{v}}_{2j}(\wt u^n_j)&= \frac1{k_n} \sum_{i=(2j-1)p_n+1}^{(2j-1)p_n+k_n} \Bigl\{ e^{\Delta\Theta((\wt u^n_j)_n)^n_{[2j-2]^m_n}+i(\wt u^n_j)_n \Delta^n_{i,[2j-2]^m_n}\eps} (Z^n_i((\wt u^n_j)_n)_{i\Den}-1)\\
		&\quad+e^{\Delta\Theta((\wt u^n_j)_n)^n_{[2j-2]^m_n}}(e^{i(\wt u^n_j)_n \Delta^n_{i,[2j-2]^m_n}\eps} -\Psi((\wt u^n_j)_n\Den^\varpi \rho_{[2j-2]^m_n\Den}))\Bigr\} \\
		&\quad+o'_p(\sqrt{p_n\Den}).
	\end{align*}
	By conditioning on $\calf_\infty\vee\calg_{[2j-2]^m_n}$, it is easy to see that the second part does not contribute to $\ov\E^n_{[2j-2]^m_n}[\wt L^{n,\mathrm{v}}_{2j}(\wt u^n_j)\Delta^n_{2j} \wt L^{\mathrm{s}}_1(\wt u^n_j)]$. By definition,  we have $Z^n_i((\wt u^n_j)_n)_{i\Den}-1=\int_{(i-1)\Den}^{i\Den} Z^n_i((\wt u^n_j)_n)_{s}d\La((\wt u^n_j)_n)_s$, where   $$\La(u)_t =iu\int_0^t \si_r dW_r + \iint_0^t (e^{iu(\ga(r,z)+\Ga(r,z))}-1) (\mu-\nu)(dr,dz),$$  so a straightforward computation shows that  
	\begin{equation}\label{eq:aux19}\begin{split}
			&	\ov\E^n_{[2j-2]^m_n}[\wt L^{n,\mathrm{v}}_{2j}(\wt u^n_j)\Delta^n_{2j} \wt L^{\mathrm{s}}_1(\wt u^n_j)]\\
			&\quad=\frac{1}{k_n^2} \sum_{i=(2j-1)p_n+1}^{(2j-1)p_n+k_n} q^n_ie^{\Delta\Theta((\wt u^n_j)_n)^n_{[2j-2]^m_n}}\Psi((\wt u^n_j)_n \Den^\varpi\rho_{[2j-2]^m_n\Den})\\
			& \qquad \times\biggl( \iint_{(i-1)\Den}^{i\Den} \E^n_{[2j-2]^m_n}\Bigl[ (\cdots)Z^n_i((\wt u^n_j)_n)_s (e^{i(\wt u^n_j)_n(\ga(s,z)+\Ga(s,z))}-1) \Bigr]\la(dz)ds\\
			&\qquad+ \int_{(i-1)\Den}^{i\Den} i(\wt u^n_j)_n\E^n_{[2j-2]^m_n}\Bigl[  Z^n_i((\wt u^n_j)_n)_s\si_s\bigl(\tfrac{\partial}{\partial c} \call^n(\wt u^n_j;c_s,\rho_s)  \si_s^c\\
			&\qquad\qquad\qquad\qquad\qquad\qquad\qquad+ \tfrac{\partial}{\partial \rho} \call^n(\wt u^n_j;c_s,\rho_s)  \si_s^\rho\bigr)\Bigr] ds\biggr)+o'_p(\sqrt{p_n\Den}),
	\end{split} \end{equation}
	where $q^n_i$ are integers satisfying $0\leq q^n_i \leq p_n$ and $(\cdots)$ stands for the integrand of the $(\mu-\nu)$-integral in the expansion of $\Delta^n_{2j}\wt L^{\rm s}_1(\wt u^n_j)$ above. After a moment of thought, one realizes that $\lvert (\cdots)\rvert \leq CJ(z)$, where $C$ does not depend on $n$, $\om$, $i$ or $j$. Since $\lvert Z^n_i((\wt u^n_j)_n)_s\rvert\leq C$ (as we have shown right below \eqref{eq:aux11}), it follows by dominated convergence that the $\la(dz)ds$-integral in the display above is $o'(\sqrt{\Den})$. Since the $ds$-integral in \eqref{eq:aux19}  is of order $\Den^{1-(\varpi\wedge1/2)}$, we only need to further consider the case where $\varpi\geq\frac12$. Then the $ds$-integral is of size $O(\sqrt{\Den})$, which implies that we are free to make any additional modification of order $o'(1)$. Because of the $L^2$-continuity properties of $\si$ and $\si^c$, and of $\rho$ and $\si^\rho$  if $\varpi\leq \frac34$, and because $ \E^n_{[2j-2]^m_n}[Z^n_i((\wt u^n_j)_n)_s] = 1$, it suffices, instead of \eqref{eq:aux7}, to consider 
	\begin{equation}\label{eq:aux24} 
		\frac{k_n}{\sqrt{\lvert \calj_n\rvert}}\sum_{j\in\calj_n} \Re\Biggl\{ \frac{\Delta^n_{2j-2}\wt L^{\mathrm{v}}(\wt u^n_{j-1})}{\Lf(Y^n(\wt u^n_{j-1})_{[2j-4]^m_n\Den})}\Biggr\}\Re\Biggl\{\frac{ \beta^n_j}{(\Lf(Y^n(\wt u^n_j)_{[2j-2]^m_n\Den}))^2}\Biggr\}
	\end{equation}
	where
	\begin{align*}
		\beta^n_j&=i\wt u^n_j \sqrt{\Den}\si_{[2j-4]^m_n\Den}(\tfrac{\partial}{\partial c} \call^n(\wt u^n_j;c_{[2j-4]^m_n\Den},\rho_{[2j-4]^m_n\Den})  \si_{[2j-4]^m_n\Den}^c\\
		&\quad+  \tfrac{\partial}{\partial \rho} \call^n(\wt u^n_j;c_{[2j-4]^m_n\Den},\rho_{[2j-4]^m_n\Den})  \si_{[2j-4]^m_n\Den}^\rho ).
	\end{align*}
	In fact, we   can also replace   $Y^n(\wt u^n_j)_{[2j-2]^m_n\Den}$ in \eqref{eq:aux24} first by $$\exp(-\frac12(\wt u^n_j)^2c_{[2j-2]^m_n\Den}\Den^{(1-2\varpi)_+})\Psi(\wt u^n_j \Den^{(\varpi-1/2)_+}\rho_{[2j-2]^m_n\Den})$$ and then, in a second step, by  $$\exp(-\frac12(\wt u^n_j)^2c_{[2j-4]^m_n\Den}\Den^{(1-2\varpi)_+})\Psi(\wt u^n_j \Den^{(\varpi-1/2)_+}\rho_{[2j-4]^m_n\Den}).$$ Thanks to (\ref{eq:U-1'}) and (\ref{eq:eta-cont}),  we can also substitute $\wt u^n_{j-1}$ for $\wt u^n_j$ in the last expression and in the definition of $\beta^n_j$. After all these modifications, the $j$th term in \eqref{eq:aux24} will have a zero $\calh^n_{[2j-4]^m_n}$-conditional expectation, so a martingale argument shows that \eqref{eq:aux24} is asymptotically negligible. Hence, in \eqref{eq:Vbar-s}, we can choose $
	\Delta^n_{2j} \ov \cf(\wt u^n_j)_{\rm{s}}	= \Delta^n_{2j} \ov \cf(\wt u^n_j)_{\rm{s}}^{\rm II}-\Delta^n_{2j} \ov \cf(\wt u^n_j)_{\rm{s}}^{\rm III}-\Delta^n_{2j} \ov \cf(\wt u^n_j)_{\rm{s}}^{\rm V}$. 
	
	%
	
	The next reduction is to remove $\Delta^n_{2j-2\ell} \ov \cf(\wt u^n_{j-\ell})_{\rm{s}}^{\rm III}+\Delta^n_{2j-2\ell} \ov \cf(\wt u^n_{j-\ell})_{\rm{s}}^{\rm V}$ ($\ell\in\{0,1\}$), for which we have to show that 
	\begin{align*}
		&\frac{k_n}{\sqrt{\lvert \calj_n\rvert}}\sum_{j\in\calj_n} (\Delta^n_{2j-2 } \ov \cf(\wt u^n_{j-1})_{\rm{s}}^{\rm III}+\Delta^n_{2j-2 } \ov \cf(\wt u^n_{j-1})_{\rm{s}}^{\rm V})\Re\Biggl\{  \frac{\Delta^n_{2j} \wt L^{\mathrm{v}}(\wt u^n_j)}{\Lf(\wt L^{n,\mathrm{s}}_{2j-1}(\wt u^n_j))}\Biggr\}\\
		&\quad+\frac{k_n}{\sqrt{\lvert \calj_n\rvert}}\sum_{j\in\calj_n} 
		\Delta^n_{2j-2}\ov \cf(\wt u^n_{j-1})_{\rm s}(\Delta^n_{2j} \ov \cf(\wt u^n_{j})_{\rm{s}}^{\rm III}+\Delta^n_{2j} \ov \cf(\wt u^n_{j})_{\rm{s}}^{\rm V})
	\end{align*}
	is asymptotically negligible. As before, we can replace     $\wt L^{n,\mathrm{s}}_{2k-\ell}(\wt u^n_k)$ by $Y^n(\wt u^n_k)_{[2k-2]^m_n\Den}$ ($k\in\{j-1,j\}$, $\ell\in\{0,1\}$).
	By a martingale argument, we can further take conditional expectation with respect to $\calh^n_{[2j-2]^m_n}$, which eliminates the first part of the sum above, while it renders the second part negligible by \eqref{eq:diff} and \eqref{eq:diff-2}. This shows that we can choose 
	$
	\Delta^n_{2j} \ov c(\wt u^n_j)_{\rm{s}}	= 	\Delta^n_{2j} \ov c(\wt u^n_j)_{\rm{s}}^{\rm II}
	$.
	
	Lastly, we want to show that we can replace $\ov V^{\prime n}_{\rm s}$ as defined in \eqref{eq:Vbar-s} (with $	\Delta^n_{2j} \ov c(\wt u^n_j)_{\rm{s}}	= 	\Delta^n_{2j} \ov c(\wt u^n_j)_{\rm{s}}^{\rm II}$) by $\wt V^{\prime n}$ from \eqref{eq:Vn-prime}, which will conclude the proof.
	We proceed in three steps by decomposing $\ov V^{\prime n}_{\rm s}-\wt V^{\prime n}=D^n_2 + D^n_3+D^n_4$, where (recall the definition of $\call^n_{j}(u)$ from the statement of the lemma)
	\begin{align*}
		D^n_2 	&=\ov V^{\prime n}_{\rm s}- \frac{k_n}{\sqrt{\lvert \calj_n\rvert}}\sum_{j\in\calj_n} \Re\Biggl\{\frac{ \Delta^n_{2j-2} \wt L^{\mathrm{v}} (\wt u^n_{j-1})}{\Lf(\ov Y^n(\wt u^n_{j-1})_{[2j-4]^m_n\Den})}\Biggr\} \Re\Biggl\{\frac{\Delta^n_{2j} \wt L^{\mathrm{v}} (\wt u^n_j)}{\Lf(\ov Y^n(\wt u^n_{j})_{[2j-2]^m_n\Den})}\Biggr\},\\
		D^n_3&=\frac{k_n}{\sqrt{\lvert \calj_n\rvert}}\sum_{j\in\calj_n} \Re\Biggl\{\frac{ \Delta^n_{2j-2} \wt L^{\mathrm{v}} (\wt u^n_{j-1})}{\Lf(\ov Y^n(\wt u^n_{j-1})_{[2j-4]^m_n\Den})}\Biggr\}\Re\Biggl\{\frac{\Delta^n_{2j} \wt L^{\mathrm{v}} (\wt u^n_{j})}{\Lf(\ov Y^n(\wt u^n_{j})_{[2j-2]^m_n\Den})} -\frac{\Delta^n_{2j} \wt L^{\mathrm{v}} (\wt u^n_{j})}{\Lf(\call^n_{j}(\wt u^n_{j}))}\Biggr\} ,\\
		D^n_4 	&= \frac{k_n}{\sqrt{\lvert \calj_n\rvert}}\sum_{j\in\calj_n}\Re\Biggl\{\frac{ \Delta^n_{2j} \wt L^{\mathrm{v}} (\wt u^n_j)}{\Lf(\call^n_{j}(\wt u^n_{j}))}\Biggr\}\Re\Biggl\{\frac{\Delta^n_{2j-2} \wt L^{\mathrm{v}} (\wt u^n_{j-1})}{\Lf(\ov Y^n(\wt u^n_{j-1})_{[2j-4]^m_n\Den})} -\frac{\Delta^n_{2j-2} \wt L^{\mathrm{v}} (\wt u^n_{j-1})}{\Lf(\call^n_{j-1}(\wt u^n_{j-1}))}\Biggr\} 
	\end{align*}
	and $\ov Y^n(u)_t$ was defined   before \eqref{eq:XY}.
	Because $D^n_3$ is  a martingale sum and we have $\lvert \ov Y^n(\wt u^n_{j})_{[2j-2]^m_n\Den}-\call^n_{j}(\wt u^n_{j})\rvert =o'(1)$ by \eqref{eq:vp-conv} and \eqref{eq:aux25},
	we readily obtain $\E[\lvert D^n_3\rvert^2]\to0$. The same argument shows that $\E[\lvert D^n_4\rvert^2]\to0$.
	Regarding $D^n_2$, we use the expansion $1/\Lf(z)=1/\Lf(z_0)-( (z-z_0)\log \lvert z_0\rvert+z_0\Re\{ (z-z_0)/z_0\})/(\Lf(z_0))^2 + O(\lvert z-z_0\rvert^2/(\lvert \Lf(z)\rvert\wedge \lvert\Lf(z_0)\rvert)^3)$ to bound
	\begin{align*} 
		&\ov\E^n_{[2j-2]^m_n}\Biggl[ \biggl\lvert\frac{1}{\Lf(\wt L^{n,\mathrm{s}}_{2j-1}(\wt u^n_j))}-\frac{1}{\Lf(\ov Y^n(\wt u^n_{j})_{[2j-2]^m_n\Den})}\\
		&\quad+\frac{\log \lvert\ov Y^n(\wt u^n_{j})_{[2j-2]^m_n\Den}\rvert }{(\Lf(\ov Y^n(\wt u^n_{j})_{[2j-2]^m_n\Den}))^2k_n} \sum_{i=(2j-2)p_n+1}^{(2j-2)p_n+k_n} \Bigl(\ov Y^n(\wt u^n_{j})_{(i-1)\Den}-\ov Y^n(\wt u^n_{j})_{[2j-2]^m_n\Den}\Bigr) 
		\\
		&\quad+\frac{  1 }{\Lf(\ov Y^n(\wt u^n_{j})_{[2j-2]^m_n\Den})k_n} \sum_{i=(2j-2)p_n+1}^{(2j-2)p_n+k_n} \Re\biggl\{\frac{\ov Y^n(\wt u^n_{j})_{(i-1)\Den}-\ov Y^n(\wt u^n_{j})_{[2j-2]^m_n\Den}}{\Lf(\ov Y^n(\wt u^n_{j})_{[2j-2]^m_n\Den})}\biggr\}\biggr\rvert\Biggr]  \end{align*}
	by ${Cp_n\Den}/{\theta_n^6}$.
	Since we have $  \ov Y^n(\wt u^n_{j})_{[2j-2]^m_n\Den}=\call^n_{j}(\wt u^n_{j})+o'(1)$ and   
	$\ov Y^n(\wt u^n_{j})_{(i-1)\Den}-\ov Y^n(\wt u^n_{j})_{[2j-2]^m_n\Den}=\psi^{n,j}_i+o'(\sqrt{p_n\Den})$, where  
	$
	\psi^{n,j}_i=\tfrac{\partial}{\partial c} \call^n(\wt u^n_j;c_{[2j-2]^m_n\Den},\rho_{[2j-2]^m_n\Den})  \linebreak (c_{(i-1)\Den}-c_{[2j-2]^m_n\Den})	 + \tfrac{\partial}{\partial \rho} \call^n(\wt u^n_j;c_{[2j-2]^m_n\Den},\rho_{[2j-2]^m_n\Den})  (\rho_{(i-1)\Den}-\rho_{[2j-2]^m_n\Den})$,
	the previous bound implies 
	that  instead of $D^n_2$, it suffices to study
	\begin{align*}
		D^n_{21}&=	\frac{k_n}{\sqrt{\lvert \calj_n\rvert}} \sum_{j\in\calj_n}\Re\Biggl\{\frac{ \Delta^n_{2j-2} \wt L^{\mathrm{v}} (\wt u^n_{j-1})}{\Lf(\wt L^{n,\mathrm{s}}_{2j-3}(\wt u^n_{j-1}))}\Biggr\}\Biggl[\Re\Biggl\{ \frac{\Delta^n_{2j} \wt L^{\mathrm{v}} (\wt u^n_j)\log \lvert \call^n_{j}(\wt u^n_{j})\rvert}{(\Lf( \call^n_{j}(\wt u^n_{j})))^2k_n} \sum_{i=(2j-2)p_n+1}^{(2j-2)p_n+k_n} \psi^{n,j}_i\Biggr\}\\
		&\quad +\Re\Biggl\{ \frac{\Delta^n_{2j} \wt L^{\mathrm{v}} (\wt u^n_j)   }{\Lf( \call^n_{j}(\wt u^n_{j})) } \Biggr\}\frac1{k_n}\sum_{i=(2j-2)p_n+1}^{(2j-2)p_n+k_n} \Re\biggl\{\frac{\psi^{n,j}_i}{\Lf(\call^n_{j}(\wt u^n_{j}))}\biggr\}\Biggr]
	\end{align*}
	and 
	\begin{align*}
		D^n_{22}&=	\frac{k_n}{\sqrt{\lvert \calj_n\rvert}} \sum_{j\in\calj_n}\Biggl[\Re\Biggl\{\frac{\Delta^n_{2j-2} \wt L^{\mathrm{v}} (\wt u^n_{j-1})\log \lvert \call^n_{j-1}(\wt u^n_{j-1})\rvert}{(\Lf( \call^n_{j-1}(\wt u^n_{j-1})))^2k_n}\sum_{i=(2j-4)p_n+1}^{(2j-4)p_n+k_n} \psi^{n,j-1}_i\Biggr\}\\
		&\quad +\Re\Biggl\{ \frac{\Delta^n_{2j-2} \wt L^{\mathrm{v}} (\wt u^n_{j-1})   }{\Lf( \call^n_{j-1}(\wt u^n_{j-1}))} \Biggr\}\frac1{k_n}\sum_{i=(2j-4)p_n+1}^{(2j-4)p_n+k_n} \Re\biggl\{\frac{\psi^{n,j-1}_i}{\Lf(\call^n_{j-1}(\wt u^n_{j-1}))}\biggr\}\Biggr]\\
		&\quad\times\Re\Biggl\{ \frac{\Delta^n_{2j} \wt L^{\mathrm{v}} (\wt u^n_j)}{\Lf(\ov Y^n(\wt u^n_{j})_{[2j-2]^m_n\Den})}\Biggr\}.
	\end{align*}
	Since the $j$th term in $D^n_{22}$ has a zero $\calh^n_{[2j-2]^m_n}$-conditional expectation, we can argue as in the analysis of $D^n_3$ to show that $D^n_{22}$ converges to $0$ in $L^2$. For $D^n_{21}$, by the same argument, we can take conditional expectation with respect to $\calh^n_{[2j-2]^m_n}$ for the $j$th term, which  can be computed similarly to \eqref{eq:aux7}. The final result after a tedious computation is that $D^n_{21}$ vanishes asymptotically.
\end{proof}

\begin{proof}[Proof of Lemma~\ref{lem:CLT}] 
	Let $\zeta^{n,\rm v}_{j} = \Re \{ { \Delta^n_{2j} \wt L^{\mathrm{v}} (\wt u^n_{j})}/{\Lf(\call^n_j(\wt u^n_j))} \}$. Then we can decompose $\wt V^n = \wt V^{n}_1+ \wt V^{n}_2 - \wt V^{n}_3-\wt V^{n}_4$, where
	\begin{align*}
		\wt V^{n}_1&=\frac{k_n}{\sqrt{\lvert \calj_n\rvert}}\sum_{j\in\calj_n}\zeta^{n,\rm v}_{j-1} \zeta^{n,\rm v}_{j}, & 
		\wt V^{n}_2&	=\frac{k_n}{\sqrt{\lvert \calj_n\rvert}}\sum_{j\in\calj_n}\zeta^{n,\rm v}_{j-1-\dsone/2} \zeta^{n,\rm v}_{j-\dsone/2}, \\
		\wt V^{n}_3&=	\frac{k_n}{\sqrt{\lvert \calj_n\rvert}}\sum_{j\in\calj_n}\zeta^{n,\rm v}_{j-1-\dsone/2} \zeta^{n,\rm v}_{j},&
		\wt V^{n}_4&=	\frac{k_n}{\sqrt{\lvert \calj_n\rvert}}\sum_{j\in\calj_n}\zeta^{n,\rm v}_{j-\dsone/2} \zeta^{n,\rm v}_{j-1}
	\end{align*}
	In each of the five sums,
	the $j$th term is $\calh^n_{(2j-1)p_n+k_n }$-measurable with a zero $\calh^n_{[2j-2]^m_n }$-conditional expectation. Therefore, $(\wt V^n_1,\dots,\wt V^n_4)$ is a four-dimensional martingale array and   we can use Theorem~2.2.15 in \citet{JP12} to prove its convergence, from which  (\ref{eq:lim}) can be easily deduced.  Verifying the assumptions of Theorem~2.2.15 in \citet{JP12} is straightforward, so we only derive the asymptotic covariance and leave the remaining conditions to the reader. A moment's thought reveals that the limits of $\wt V^n_1,\dots,\wt V^n_4$ are $\calf_\infty$-conditionally independent, so all we are left to show is that their $\calf_\infty$-conditional variances are  given, respectively, by $\frac{1}{T}\int_{I_T} q_t^2dt$, $\frac{1}{T}\int_{I_T} q_{t-1}^2dt$ and, for both $\wt V^n_3$ and $\wt V^n_4$, by $\frac{1}{T}\int_{I_T} q_tq_{t-1} dt$. As the proof is almost identical, we only determine the $\calf_\infty$-conditional variance of  $\wt V^n_1$, which we denote by 
	$Q^{(1)}$ and  is given by the limit as $n\to\infty$ of 
	\begin{equation}\label{eq:Q} 
		Q_n=	\frac{k_n^2}{ \lvert \calj_n\rvert}\sum_{j\in\calj_n} \ov\E^n_{[2j-2]^m_n}[(\zeta^{n,\rm v}_{j-1} \zeta^{n,\rm v}_{j})^2] =\frac{k_n^2}{ \lvert \calj_n\rvert}\sum_{j\in\calj_n}  (\zeta^{n,\rm v}_{j-1}  )^2\ov\E^n_{[2j-2]^m_n} [ ( \zeta^{n,\rm v}_{j})^2 ].\end{equation}
	By a martingale argument,\footnote{The difference between \eqref{eq:Q} and \eqref{eq:Q2} can be split into two martingale sums, one summing over even and one summing over odd values of $j$.} this is asymptotically equivalent to 
	\begin{equation}\label{eq:Q2} 
		\frac{k_n^2}{ \lvert \calj_n\rvert}\sum_{j\in\calj_n} \ov\E^n_{[2j-4]^m_n} [ (\zeta^{n,\rm v}_{j-1}  )^2\ov\E^n_{[2j-2]^m_n} [ ( \zeta^{n,\rm v}_{j})^2 ]].
	\end{equation}
	
	To determine the limit of \eqref{eq:Q2}, we start with   the following a priori estimate, which follows from the bounds in \eqref{eq:aux10}:\footnote{Note that one can replace $e^{iU_n\Delta^n_i y}$ in $\xi^n_i$ by $1-e^{iU_n\Delta^n_i y}$ without changing anything, where $U=\wt u^n_j$ in our case. Then the claim follows from the bound $\ov\E^n_{[2j-2]^m_n}[\lvert 1-e^{i(\wt u^n_j)_n\Delta^n_i y}\rvert^2]\leq (\wt u^n_j)^2_n\ov \E^n_{[2j-2]^m_n}[(\Delta^n_i y)^2]\leq C(\wt u^n_j)^2\leq C\theta_n^2$.}  
	\begin{equation}\label{eq:Lv-2} 
		\ov\E^n_{[2j-2]^m_n}[\lvert \Delta^n_{2j}\wt L^{\rm v}(\wt u^n_j)\rvert^2]= \E^n_{[2j-2]^m_n}[\lvert \Delta^n_{2j}\wt L^{\rm v}(\wt u^n_j)\rvert^2] \leq C \theta_n^2/k_n, 
	\end{equation}
	where $C$ does not depend on $\om$ or $j$. 
	Moreover, $\wt L^{n,\mathrm{v}}_{2j}(\wt u^n_j)$ and $\wt L^{n,\mathrm{v}}_{2j-1}(\wt u^n_j)$ are  uncorrelated and $\call^n_j(\wt u^n_j)$ is $\calh^n_{[2j-2]^m_n}$-measurable, so
	\begin{equation*}
		\ov\E^n_{[2j-2]^m_n} [ (  \zeta^{n,\rm v}_{j})^2 ]=\ov\E^n_{[2j-2]^m_n}\biggl[\biggl( \Re\biggl\{\frac{ \wt L^{n,\mathrm{v}}_{2j} (  \wt u^n_j)}{\Lf(\call^n_j(\wt u^n_j))}\biggr\}\biggr)^2+\biggl( \Re\biggl\{\frac{  \wt L^{n,\mathrm{v}}_{2j-1} (  \wt u^n_j)}{\Lf(\call^n_j(\wt u^n_j))}\biggr\}\biggr)^2\biggr]. 
	\end{equation*}
	For $\ell\in\{0,1\}$, we have the identity
	\begin{align}\nonumber
		\biggl( \Re\biggl\{\frac{ \wt L^{n,\mathrm{v}}_{2j-\ell} (  \wt u^n_j)}{\Lf(\call^n_j(\wt u^n_j))}\biggr\}\biggr)^2	&=\frac{(\Re\wt L^{n,\mathrm{v}}_{2j-\ell} (  \wt u^n_j) )^2(\Re \Lf(\call^n_j(\wt u^n_j)))^2}{\lvert \Lf(\call^n_j(\wt u^n_j))\rvert^4}+ \frac{(\Im\wt L^{n,\mathrm{v}}_{2j-\ell} (  \wt u^n_j) )^2(\Im \Lf(\call^n_j(\wt u^n_j)))^2}{\lvert\Lf( \call^n_j(\wt u^n_j))\rvert^4}\\
		& \quad+\frac{2(\Re\wt L^{n,\mathrm{v}}_{2j-\ell} (  \wt u^n_j) )(\Re \Lf(\call^n_j(\wt u^n_j)))(\Im\wt L^{n,\mathrm{v}}_{2j-\ell} (  \wt u^n_j) )(\Im \Lf(\call^n_j(\wt u^n_j)))}{\lvert \Lf(\call^n_j(\wt u^n_j))\rvert^4}.
		\label{eq:3terms}\end{align}
	Since $\lvert \cos x-1 \rvert \leq  \frac12 x^2$ and $\lvert \sin x  \rvert \leq  x$, one can argue as in \eqref{eq:aux10} to derive the bounds
	\begin{equation}\label{eq:bounds}
		\ov\E^n_{[2j-2]^m_n} [ \lvert\Re \wt L^{n,\mathrm{v}}_{2j-\ell} (  \wt u^n_j) \rvert^2 ]\leq C\theta_n^4/k_n,\qquad \ov\E^n_{[2j-2]^m_n} [ \lvert\Im \wt L^{n,\mathrm{v}}_{2j-\ell} (  \wt u^n_j) \rvert^2 ]\leq C\theta_n^2/k_n.
	\end{equation}
	For $\Lf(\call^n_j(\wt u^n_j))$, there are $C,C_1,C_2>0$ such that
	\begin{equation}\label{eq:bounds-2} 
		C_1\theta_n^2\leq \lvert \Re \Lf(\call^n_j(\wt u^n_j))\rvert, \lvert\Lf(\call^n_j(\wt u^n_j))\rvert\leq 
		C_2\theta_n^2,\quad \lvert\Im\Lf(\call^n_j(\wt u^n_j))\rvert \leq C \theta_n^5\Den^{3(\varpi-1/2)_+}, 
	\end{equation}  
	where the first set of bounds follows  as in Lemma~\ref{lem:Om} and  the last  bound holds because $\lvert e^{ix}-1-ix+\frac12 x^2\rvert\leq C\lvert x\rvert^3$ and $\E[\Delta\chi_1]=0$. In conjunction with \eqref{eq:Lv-2}, these bounds show that the second term on the right-hand side of \eqref{eq:3terms} is negligible for computing the limit of $Q_n$. Regarding the last term in \eqref{eq:3terms}, note that 	
	\eqref{eq:bounds} and the Cauchy--Schwarz inequality imply $\ov\E^n_{[2j-2]^m_n} [(\Re\wt L^{n,\mathrm{v}}_{2j-\ell} (  \wt u^n_j))(\Im\wt L^{n,\mathrm{v}}_{2j-\ell} (  \wt u^n_j)) ]\leq C\theta_n^3/k_n$. Therefore, $\ov\E^n_{[2j-2]^m_n}$ of the last term in \eqref{eq:3terms} is bounded by $C\theta_n^2/k_n$ and does not contribute to $Q^{(1)}$.
	Thus, $Q_n$ is asymptotically equivalent to 
	\begin{align*}
		&	\frac{k_n^2}{ \lvert \calj_n\rvert}\sum_{j\in\calj_n}\ov\E^n_{[2j-4]^m_n}\Biggl[\frac{[(\Re\wt L^{n,\mathrm{v}}_{2j-2} (  \wt u^n_{j-1}) )^2+(\Re\wt L^{n,\mathrm{v}}_{2j-3} (  \wt u^n_{j-1}) )^2](\Re \Lf(\call^n_{j-1}(\wt u^n_{j-1})))^2}{\lvert \Lf(\call^n_{j-1}(\wt u^n_{j-1}))\rvert^4}\\
		&\qquad\qquad\qquad  \qquad\times\E^n_{[2j-2]^m_n}\biggl[\frac{[(\Re\wt L^{n,\mathrm{v}}_{2j} (  \wt u^n_j) )^2+(\Re\wt L^{n,\mathrm{v}}_{2j-1} (  \wt u^n_j) )^2](\Re\Lf( \call^n_j(\wt u^n_j)))^2}{\lvert\Lf( \call^n_j(\wt u^n_j))\rvert^4}\biggr]\Biggr].
	\end{align*}
	In fact, by \eqref{eq:bounds} and \eqref{eq:bounds-2}, the previous display is $O_p(1)$, so we are allowed to make any further $o(1)$-modification. For example,  we can   ignore the imaginary part of $\call^n_{j-\ell}(\wt u^n_{j-\ell})$ ($\ell=0,1$) and replace $\call^n_{j}(\wt u^n_{j})$ by $\call^n_{j-1}(\wt u^n_{j-1})$. 
	This shows that $Q_n$ is asymptotically equivalent to
	\begin{equation}\label{eq:aux29}\begin{split}
			&	\frac{k_n^2}{ \lvert \calj_n\rvert}\sum_{j\in\calj_n}
			\frac{1}{(\Re\Lf(\call^n_{j-1}(\wt u^n_{j-1})))^4} \ov\E^n_{[2j-4]^m_n}\Bigl[ \bigl[(\Re\wt L^{n,\mathrm{v}}_{2j-2} (  \wt u^n_{j-1}) )^2+(\Re\wt L^{n,\mathrm{v}}_{2j-3} (  \wt u^n_{j-1}) )^2\bigr] \\
			&\qquad\qquad\qquad\qquad\qquad  \qquad\qquad \times\ov\E^n_{[2j-2]^m_n} [ (\Re\wt L^{n,\mathrm{v}}_{2j} (  \wt u^n_j) )^2+(\Re\wt L^{n,\mathrm{v}}_{2j-1} (  \wt u^n_j) )^2 ]\Bigr].
	\end{split}\end{equation}
	Regarding the denominator, we have
	\begin{equation}\label{eq:aux30} \begin{split}
			\Re\Lf(\call^n_{j-1}(\wt u^n_{j-1}))&=-\tfrac12(\wt u^n_j)^2(c_{[2j-2]^m_n\Den}\bone_{\{\varpi\geq\frac12\}}+\rho^2_{[2j-2]^m_n\Den}\E[(\Delta\chi_1)^2]\bone_{\{\varpi\leq\frac12\}})\\
			&\quad\times\exp(-\tfrac12(\theta_0^2/\wt\eta^n_j)^2c_{[2j-2]^m_n\Den})+o(\theta_n^2).\end{split}
	\end{equation}
	Moreover, since $(\chi_i)_{i\in\Z}$ is $m$-dependent and $\cos(x)\cos(y)=\frac12(\cos(x+y)+\cos(x-y))$,  
	\begin{equation}\label{eq:aux28}\begin{split}
			&\ov\E^n_{[2j-2]^m_n} [ (\Re\wt L^{n,\mathrm{v}}_{2j-\ell} (  \wt u^n_j) )^2]= \E^n_{[2j-2]^m_n} [ (\Re\wt L^{n,\mathrm{v}}_{2j-\ell} (  \wt u^n_j) )^2]\\
			&  \quad =\frac1{k_n^2}\sum_{i,i'} \E^n_{[2j-2]^m_n} \Bigl[ \cos((\wt u^n_j)_n\Delta^n_iy)\cos((\wt u^n_j)_n\Delta^n_{i'}y)\\
			&\qquad\qquad\qquad\qquad\qquad\qquad-\E^n_{i-1}[\cos((\wt u^n_j)_n\Delta^n_iy)]\E^n_{i'-1}[\cos((\wt u^n_j)_n\Delta^n_{i'}y)] \Bigr]\\
			&   \quad=\frac1{2k_n^2}\sum_{i,i'}  \E^n_{[2j-2]^m_n} \Bigl[ \E^n_{i\wedge i'-1}[\cos((\wt u^n_j)_n(\Delta^n_i y+\Delta^n_{i'}y))]\\
			&\qquad+ \E^n_{i\wedge i'-1}[\cos((\wt u^n_j)_n(\Delta^n_i y-\Delta^n_{i'}y))] -2\E^n_{i-1}[\cos((\wt u^n_j)_n\Delta^n_i y) ]\E^n_{i'-1}[\cos((\wt u^n_j)_n\Delta^n_{i'} y) ]\Bigr]
	\end{split}\end{equation}
	for $\ell\in\{0,1\}$,
	where the sum ranges over all $i,i'\in\{(2j-\ell-1)p_n+1,\dots,(2j-\ell-1)p_n+k_n\}$ such that $i-i'\in\{-m-1,\dots,m+1\}$.
	Using \eqref{eq:aux27} for the second step, \eqref{eq:vp-conv} for the fourth one and (\ref{eq:U-1'}) and (\ref{eq:eta-cont}) for the last one, we have
	\begin{align*}
		&\E^n_{i-1}[\cos((\wt u^n_j)_n\Delta^n_i y) ]	=\Re\E^n_{i-1}[e^{i(\wt u^n_j)_n\Delta^n_i y} ]\\
		&\quad=\Re\E^n_{i-1}[e^{\Delta \Theta((\wt u^n_j)_n)^n_{[2j-2]^m_n}}Z^n_i((\wt u^n_j)_n)_{i\Den}e^{i(\wt u^n_j)_n\Delta^n_{i,[2j-2]^m_n} \eps}]+O(p_n\Den) \\
		&\quad=\Re\E^n_{i-1}[e^{\Delta \Theta((\wt u^n_j)_n)^n_{[2j-2]^m_n}}Z^n_i((\wt u^n_j)_n)_{i\Den}\Psi((\wt u^n_j)_n\Den^\varpi\rho_{[2j-2]^m_n\Den})]+O(p_n\Den) \\
		&\quad=\exp(-\tfrac12(\wt u^n_j)^2c_{[2j-2]^m_n\Den}\Den^{(1-2\varpi)_+})\Re \Psi(\wt u^n_j \Den^{(\varpi-1/2)_+}\rho_{[2j-2]^m_n\Den})+o'(1)\\
		&\quad=\exp(-\tfrac12(\wt u^n_{j-1})^2c_{[2j-4]^m_n\Den}\Den^{(1-2\varpi)_+})\Re \Psi(\wt u^n_{j-1} \Den^{(\varpi-1/2)_+}\rho_{[2j-4]^m_n\Den})+o'(1).
	\end{align*}
	Similarly, with the notations $\Psi^\pm_r(u)=\E[e^{iu(\Delta\chi_{r+1}\pm\Delta\chi_1)}]$ and   $r=i-i'$, we can deduce
	\begin{align*}
		\E^n_{i\wedge i'-1}[\cos((\wt u^n_j)_n(\Delta^n_i y\pm\Delta^n_{i'}y))] &=\exp(- (1\pm\bone_{\{r=0\}})(\wt u^n_{j-1})^2c_{[2j-4]^m_n\Den}\Den^{(1-2\varpi)_+}) \\
		&\quad\times \Re \Psi^\pm_r(\wt u^n_{j-1} \Den^{(\varpi-1/2)_+}\rho_{[2j-4]^m_n\Den})+o'(1).
	\end{align*}
	Now, if $\varpi>\frac12$, then the contributions of $\Psi$ and $\Psi^\pm$ are $o'(1)$ and only the terms where $r=0$ remain, 
	which via \eqref{eq:aux28} shows that for $\ell\in\{0,1\}$,
	\begin{align*}
		&	\ov\E^n_{[2j-2\ell-2]^m_n} [ (\Re\wt L^{n,\mathrm{v}}_{2j-2\ell} (  \wt u^n_{j-\ell}) )^2+(\Re\wt L^{n,\mathrm{v}}_{2j-2\ell-1} (  \wt u^n_{j-\ell}) )^2 ]\\
		&\quad=\frac{(1-\exp(-(\wt u^n_{j-1})^2c_{[2j-4]^m_n\Den}))^2}{k_n}  + o'(1).
	\end{align*}
	
	If $\varpi<\frac12$, the contributions of $\Psi$ and $\Psi^\pm$ dominate asymptotically. Moreover, we have $\theta_n\to0$ by assumption, so   expanding $\cos(x)=1-\frac12 x^2+\frac1{24}x^4+O(x^6)$, we get
	\begin{align*}
		&\E^n_{i-1}[\cos((\wt u^n_j)_n\Delta^n_i y) ]=\Re \Psi(\wt u^n_{j-1} \rho_{[2j-4]^m_n\Den})+o'(1) \\
		&\quad=1-\tfrac12(\wt u^n_{j-1})^2\rho_{[2j-4]^m_n\Den}^2 \E[(\Delta\chi_1)^2] + \tfrac1{24}(\wt u^n_{j-1})^4\rho_{[2j-4]^m_n\Den}^4\E[(\Delta\chi_1)^4]+O(\theta_n^6)
	\end{align*}
	and 
	\begin{align*}
		\E^n_{i\wedge i'-1}[\cos((\wt u^n_j)_n(\Delta^n_i y\pm\Delta^n_{i'}y))]  &= \Re \Psi^\pm_r(\wt u^n_{j-1}  \rho_{[2j-4]^m_n\Den})+o'(1)\\
		& =1-\tfrac12(\wt u^n_{j-1})^2\rho_{[2j-4]^m_n\Den}^2 \E[(\Delta\chi_{r+1}\pm\Delta\chi_1)^2] \\
		&\quad+ \tfrac1{24}(\wt u^n_{j-1})^4\rho_{[2j-4]^m_n\Den}^4\E[(\Delta\chi_{r+1}\pm\Delta\chi_1)^4]+O(\theta_n^6),
	\end{align*}
	which inserted in \eqref{eq:aux28} and after simplifications yields  
	\begin{align*}
		&\ov\E^n_{[2j-2\ell-2]^m_n} [ (\Re\wt L^{n,\mathrm{v}}_{2j-2\ell} (  \wt u^n_{j-\ell}) )^2+(\Re\wt L^{n,\mathrm{v}}_{2j-2\ell-1} (  \wt u^n_{j-\ell}) )^2 ]\\
		&\quad=\frac{(\wt u^n_{j-1})^4}{2k_n}\rho^4_{[2j-4]^m_n\Den}\biggl(\Var((\Delta \chi_1)^2)+2\sum_{r=1}^{m+1}\Cov((\Delta\chi_{r+1})^2,(\Delta\chi_1)^2) \biggr)+ O(\theta_n^6).
	\end{align*}

	If $\varpi=\frac12$,   we also have $\theta_n\to0$, so combining the previous results, we obtain
	\begin{align*}
		\E^n_{i-1}[\cos((\wt u^n_j)_n\Delta^n_i y) ]&=1-\tfrac12(  \wt u^n_{j-1} )^2 (c_{[2j-4]^m_n\Den}+\rho^2_{[2j-4]^m_n\Den}\E[(\Delta\chi_1)^2])\\
		& \quad +\tfrac1{24}(  \wt u^n_{j-1} )^4(3c_{[2j-4]^m_n\Den}^2 +6c_{(2j-4)p_n\Den}\rho^2_{[2j-4]^m_n\Den}\E[(\Delta\chi_1)^2]\\
		&\quad+\rho^4_{[2j-4]^m_n\Den}\E[(\Delta\chi_1)^4]) +O(\theta_n^6)
	\end{align*}
	and
	\begin{align*}
		&\E^n_{i\wedge i'-1}[\cos((\wt u^n_j)_n(\Delta^n_i y\pm\Delta^n_{i'}y))]\\
		& =1-(  \wt u^n_{j-1} )^2\bigl((1\pm\bone_{\{r=0\}}) c_{[2j-4]^m_n\Den}+\tfrac12\rho^2_{[2j-4]^m_n\Den}\E[(\Delta\chi_{r+1}\pm\Delta\chi_1)^2]\bigr)\\
		&\quad +(  \wt u^n_{j-1} )^4\bigl(\tfrac12(1\pm\bone_{\{r=0\}})^2c_{[2j-4]^m_n\Den}^2 +\tfrac1{24}\rho^4_{[2j-4]^m_n\Den}\E[(\Delta\chi_{r+1}\pm\Delta\chi_1)^4]\\
		&\quad+\tfrac12(1\pm\bone_{\{r=0\}})c_{[2j-4]^m_n\Den}\rho^2_{[2j-4]^m_n\Den}\E[(\Delta\chi_{r+1}\pm\Delta\chi_1)^2]\bigr) +O(\theta_n^6).
	\end{align*}
	Thus, if $\varpi=\frac12$, we have
	\begin{align*}
		&\ov\E^n_{[2j-2\ell-2]^m_n} [ (\Re\wt L^{n,\mathrm{v}}_{2j-2\ell} (  \wt u^n_{j-\ell}) )^2+(\Re\wt L^{n,\mathrm{v}}_{2j-2\ell-1} (  \wt u^n_{j-\ell}) )^2 ]\\
		&\quad=\frac{(\wt u^n_{j-1})^4}{2k_n}\Biggl(2c^2_{[2j-4]^m_n\Den}+4c_{[2j-4]^m_n\Den}\rho^2_{[2j-4]^m_n\Den}\E[(\Delta\chi_1)^2]\\
		&\qquad   +\rho^4_{[2j-4]^m_n\Den}\biggl(\Var((\Delta \chi_1)^2)+2\sum_{r=1}^m\Cov((\Delta\chi_{r+1})^2,(\Delta\chi_1)^2)\biggr)\Biggr) + O(\theta_n^6).
	\end{align*}
	Inserting the obtained expansions  in \eqref{eq:aux29} and noting that $\E[\lvert (\wt u^n_{j-1})^2-\theta_n^2/\eta_{(2j-4)p_n\Den}\rvert]\leq  \theta_n^2(K\vee \eta_0^{-1})^2\E[\lvert \eta^n_j-\eta_{(2j-2)p_n\Den}\rvert] =o(\theta_n^2)$ by (\ref{eq:eta-cont}) and $\log \lvert \Psi(u)\rvert = \Re \Log \Psi(u)=\Re \Log (1-\frac12 u^2 \E[(\Delta\chi_1)^2]+O(u^3))=-\frac12u^2\E[(\Delta\chi_1)^2]+O(u^3)$, one can use classical Riemann approximation results  to show  $Q^{(1)}=\frac{1}{T}\int_{I_T} q_t^2 dt$  (by distinguishing the cases $\varpi>\frac12$, $\varpi=\frac12$ and $\varpi<\frac12$, and in the first case, further whether $\theta_0=0$ or not).
\end{proof}

\begin{proof}[Proof of Lemma~\ref{lem:denom}] 
	We decompose $W^n= W^{\prime n}+W^{\prime \prime n}$, where
	\begin{align*}
		W^{\prime n}&= \frac{ k_n^2 }{\lvert\calj_n\rvert }\sum_{j\in\calj_n}[(\Delta^n_{2j} \wt \cf( \wt{u}^n_{j}) )^2 +(\Delta^n_{2j-\dsone} \wt \cf( \wt{u}^n_{j-\dsone/2}) )^2]\\
		&\quad\times[(\Delta^n_{2j-2} \wt \cf( \wt{u}^n_{j-1}) )^2 + (\Delta^n_{2j-2-\dsone} \wt \cf( \wt{u}^n_{j-1-\dsone/2}) )^2],\\
		W^{\prime\prime n}&=\frac{ 2k_n^2 }{\lvert\calj_n\rvert }\sum_{j\in\calj_n}\Bigl\{2\Delta^n_{2j} \wt \cf( \wt{u}^n_{j}) \Delta^n_{2j-2} \wt \cf( \wt{u}^n_{j-1}) \Delta^n_{2j-\dsone} \wt \cf( \wt{u}^n_{j-\dsone/2}) \Delta^n_{2j-2-\dsone} \wt \cf( \wt{u}^n_{j-1-\dsone/2}) \\
		&\quad  -[(\Delta^n_{2j} \wt \cf( \wt{u}^n_{j}) )^2+(\Delta^n_{2j-\dsone} \wt \cf( \wt{u}^n_{j-\dsone/2}) )^2]\Delta^n_{2j-2} \wt \cf( \wt{u}^n_{j-1}) \Delta^n_{2j-2-\dsone} \wt \cf( \wt{u}^n_{j-1-\dsone/2}) \\
		&\quad  -[(\Delta^n_{2j-2} \wt \cf( \wt{u}^n_{j-1}) )^2+(\Delta^n_{2j-2-\dsone} \wt \cf( \wt{u}^n_{j-1-\dsone/2}) )^2]\Delta^n_{2j} \wt \cf( \wt{u}^n_{j}) \Delta^n_{2j-\dsone} \wt \cf( \wt{u}^n_{j-\dsone/2})  \Bigr\}.
	\end{align*}
	As seen in the proof of Lemma~\ref{lem:linearize}, we have $\ov\E^n_{[2j-2]^m_n}[\Delta^n_{2j} \wt \cf( \wt{u}^n_{j})]=\ov\E^n_{[2j-2]^m_n}[\Delta^n_{2j} \ov \cf( \wt{u}^n_{j})^{\rm I}_{\rm s}+\Delta^n_{2j} \ov \cf( \wt{u}^n_{j})^{\rm II}_{\rm s}]+o'(1/\sqrt{k_n})=o'(1/\sqrt{k_n})$. Therefore, invoking a martingale argument, we can apply $\ov\E^n_{[2j-2]^m_n}$ to the $j$th term defining $W^{\prime\prime n}$, after which   the only expression left to be analyzed is\linebreak $
	-\frac{ 2k_n^2 }{\lvert\calj_n\rvert }\sum_{j\in\calj_n}\ov\E^n_{[2j-2]^m_n}[(\Delta^n_{2j} \wt \cf( \wt{u}^n_{j}) )^2]\Delta^n_{2j-2} \wt \cf( \wt{u}^n_{j-1}) \Delta^n_{2j-2-\dsone} \wt \cf( \wt{u}^n_{j-1-\dsone/2}) 
	$. Similarly to how we eliminated $\Delta^n_{2j-2}\ov \cf(\wt u^n_{j-1})^{\rm I}_{\rm s}$ in the paragraph following \eqref{eq:aux2}, one can verify that this term converges in probability to $0$. 
	
	Therefore, only $W^{\prime n}$ contributes asymptotically.
	By our analysis of $\Delta^n_{2j} \wt{\cf}(\wt u^n_j)$ in the proof of Lemma~\ref{lem:linearize}, we have
	\begin{equation}\label{eq:aux31}
		W^{\prime n}= \frac{ k_n^2 }{\lvert\calj_n\rvert }\sum_{j\in\calj_n} [ ( \zeta^{n,\rm v}_{j})^2+ ( \zeta^{n,\rm v}_{j-\dsone/2})^2 ] [ ( \zeta^{n,\rm v}_{j-1})^2+ ( \zeta^{n,\rm v}_{j-1-\dsone/2})^2 ]+ o_p(1).
	\end{equation}
	Analogously to the decomposition of $\wt V^n$ at the beginning of the proof of Lemma~\ref{lem:CLT}, we can split the last line (without the $o_p(1)$-term) into four terms, say, $W^{\prime n}_1,\dots,W^{\prime n}_4$, each of which is asymptotically equivalent to the $\calf_\infty$-conditional variance of the corresponding $\wt V^n_i$-term. To see this, consider $W^{\prime n}_1$ for example, which is given by 
	$
	W^{\prime n}_1= \frac{ k_n^2 }{\lvert\calj_n\rvert }\sum_{j\in\calj_n} ( \zeta^{n,\rm v}_{j-1} \zeta^{n,\rm v}_{j})^2$.
	As usual, by a martingale argument, we can take $\ov\E^n_{[2j-4]^m_n}$-conditional expectation, which turns the previous line exactly into \eqref{eq:Q2}. This shows that $W^{\prime n}_1$ converges in probability to $Q^{(1)}$. Repeating this   argument for $W^{\prime n}_2$, $W^{\prime n}_3$ and $W^{\prime n}_4$ completes the proof of the lemma.
\end{proof}

\begin{proof}[Proof of Lemma~\ref{lem:V-alt}]
	Throughout this proof, we use $O_p(a_n)$ to signify a term whose $L^2$-norm is bounded by $Ca_n$, where $C$ is a constant   independent of $\om$, $i$ and $j$. The notations $o_p(a_n)$ and $o'_p(a_n)$ are used similarly. Also, recall that  $H_\rho=0$ by convention if $\varpi>\frac12+\frac12 H$.
	
	Our analysis of $\wt L^{n,\mathrm{v}}_{2j-\ell}(U)$, $\ell=0,1$, in the proof of Lemma~\ref{lem:bounded} was independent of whether we are under \ref{ass:H0-2} or \ref{ass:H1-2}. In particular, we still have \eqref{eq:L'}.
	Regarding $\Delta^n_{2j}\wt  L^{\mathrm{b}}(\wt u^n_j)$, we argue as in \eqref{eq:aux15} 
	and use \eqref{eq:L2-2} and (\ref{eq:prop4-alt-2}) to obtain
	\begin{align*}
		\E^n_{i-1}[e^{i(\wt u^n_j)_n\Delta^n_i y}-e^{i(\wt u^n_j)_n\Delta^n_{i,i-1} y } ]& =\E^n_{i-1}[e^{i(\wt u^n_j)_n\Delta^n_{i,i-1}x}(e^{i(\wt u^n_j)_n\Delta^n_i \eps}-e^{i(\wt u^n_j)_n\Delta^n_{i,i-1} \eps})] \\
		&\quad+\E^n_{i-1}[e^{i(\wt u^n_j)_n\Delta^n_{i,i-1}\eps}(e^{i(\wt u^n_j)_n\Delta^n_i x}-e^{i(\wt u^n_j)_n\Delta^n_{i,i-1} x})] \\
		&\quad+o'_p(\Den^{H}).
	\end{align*}
	Because $\E^n_{i-1}[e^{i(\wt u^n_j)_n\Delta^n_{i,i-1}x}(e^{i(\wt u^n_j)_n\Delta^n_i \eps}-e^{i(\wt u^n_j)_n\Delta^n_{i,i-1} \eps})]=O_p(\Den^{H_\rho+(2\varpi-1)_+})$ (cf.\ \eqref{eq:aux18}) and\linebreak
	$\E^n_{i-1}[e^{i(\wt u^n_j)_n\Delta^n_{i,i-1}\eps}(e^{i(\wt u^n_j)_n\Delta^n_i x}-e^{i(\wt u^n_j)_n\Delta^n_{i,i-1} x})]=O_p(\Den^{H+(1-2\varpi)_+})$, we conclude that  
	$
	\Delta^n_{2j}\wt  L^{\mathrm{b}}(\wt u^n_j) 
	=O_p(\Den^{H_\ast})$ with $H_\ast = (H+(1-2\varpi)_+)\wedge (H_\rho+(2\varpi-1)_+)$.
	
	By the Lévy--Khintchine formula, the mean-value theorem and \eqref{eq:vp-conv} and (\ref{eq:prop4-alt-2}), we further have that
	\begin{align*}
		&\Delta^n_{2j} \wt  L^{\mathrm{s}}(\wt u^n_j)
		=\frac{1}{k_n}\sum_{i=(j-1)p_n+1}^{(j-1)p_n+k_n}  \Bigl\{  e^{-\frac12(\wt u^n_j)^2_n c_{(i-1)\Den}\Den}\Psi((\wt u^n_j)_n\Den^\varpi\rho_{(i-1)\Den})\\
		&\qquad -e^{-\frac12(\wt u^n_j)^2_n c_{(i-p_n-1)\Den}\Den }\Psi((\wt u^n_j)_n\Den^\varpi\rho_{(i-p_n-1)\Den}) \Bigr\}+o'_p((p_n\Den)^{H}\Den^{(1-2\varpi)_+}).
	\end{align*}
	Since the last display is $O_p(\pi_n\theta_n^2)$ by \eqref{eq:psi}, $1/\sqrt{k_n}=o'((p_n\Den)^H)=o'(\pi_n)$ by (\ref{eq:rates}), $\Den^{H_\ast}=o'(\pi_n)$ and $\log \lvert \wt L^{n,\rm s}_{2j-1}(\wt u^n_j)\rvert^{-1} \geq C\theta_n^2$ by \eqref{eq:as-bounds}, we conclude that 
	\[ \Delta^n_{2j} \wt \cf(\wt u^n_j)=   \Re \biggl\{\frac{\Delta^n_{2j} \wt  L^{\mathrm{s}}(\wt u^n_j)}{\Lf(\wt  L^n_{2j-1}(\wt u^n_j))}\biggr\} + o'_p(\pi_n)=   \Re \biggl\{\frac{\Delta^n_{2j} \wt  L^{\mathrm{s}}(\wt u^n_j)}{\Lf(\call^n_j(\wt u^n_j))}\biggr\} + o'_p(\pi_n),\]
	where the second step follows from $\wt  L^n_{2j-1}(\wt u^n_j)=\call^n_j(\wt u^n_j)+o'_p(1)$ (recall the definition after (\ref{eq:Vn})).
	Then $V^{n,\rm{alt}}=V^{n,\rm{alt}}_1+V^{n,\rm{alt}}_2-V^{n,\rm{alt}}_3-V^{n,\rm{alt}}_4+o_p(1)$, where
	\begin{align*}
		V^{n,\rm{alt}}_1&=\frac{1}{ {\lvert \calj_n\rvert}\pi_n^2}\sum_{j\in\calj_n}\zeta^{n,\rm s}_{j-1} \zeta^{n,\rm s}_{j}, & 
		V^{n,\rm{alt}}_2&	=\frac{1}{ {\lvert \calj_n\rvert}\pi_n^2}\sum_{j\in\calj_n}\zeta^{n,\rm s}_{j-1-\dsone/2} \zeta^{n,\rm s}_{j-\dsone/2}, \\
		V^{n,\rm{alt}}_3&=	\frac{1}{ {\lvert \calj_n\rvert}\pi_n^2}\sum_{j\in\calj_n}\zeta^{n,\rm s}_{j-1-\dsone/2} \zeta^{n,\rm s}_{j},&
		V^{n,\rm{alt}}_4&=	\frac{1}{ {\lvert \calj_n\rvert}\pi_n^2}\sum_{j\in\calj_n}\zeta^{n,\rm s}_{j-\dsone/2} \zeta^{n,\rm s}_{j-1}
	\end{align*}
	and $ \zeta^{n,\rm s}_{j}=\Re  \{ {\Delta^n_{2j} \wt  L^{\mathrm{s}}(\wt u^n_j)}/{\Lf(\call^n_j(\wt u^n_j))} \} $.
	
	Next, with the notation $p^n_i=i-p_n-1$, 
	a first-order expansion shows that  $\Delta^n_{2j}\wt  L^{\mathrm{s}}(\wt u^n_j)$ is given by\footnote{The second term is set to  zero if $\varpi> \frac12+\frac12 H$.}
	\begin{align*}
		&-\frac{(\wt u^n_j)^2\Den^{(1-2\varpi)_+}}{2k_n}\sum_{i=(2j-1)p_n+1}^{(2j-1)p_n+k_n} e^{-\frac12 (\wt u^n_j)^2_nc_{p^n_i\Den}\Den}\Psi((\wt u^n_j)_n\Den^\varpi \rho_{p^n_i\Den})f'(v_{p^n_i\Den})\\
		&\qquad  \times\int_0^{(i-1)\Den} [g ((i-1)\Den-s)-g( p^n_i\Den-s  )]  (\si^v_sdW_s+\ov\si^v_sd\ov W_s)\\
		&+ \frac{\wt u^n_j\Den^{( \varpi-1/2)_+}}{k_n}\sum_{i=(2j-1)p_n+1}^{(2j-1)p_n+k_n} e^{-\frac12 (\wt u^n_j)^2_nc_{p^n_i\Den}\Den}\Psi'((\wt u^n_j)_n\Den^\varpi \rho_{p^n_i\Den})F'(w_{p^n_i\Den})\\
		&\qquad\times\int_0^{(i-1)\Den} [G((i-1)\Den-s)-G( p^n_i\Den-s  )]  (\si^w_sdW_s+\ov\si^w_sd\ov W_s+\wh\si^w_sd\wh W_s)
	\end{align*}
	plus an $o'_p(\pi_n)$-error. Note that  we have
	$
	\Psi((\wt u^n_j)_n\Den^\varpi \rho_{p^n_i\Den})=1+ O(\Den^{(2\varpi-1)_+}\theta_n^2)$ and $
	\Psi'((\wt u^n_j)_n\Den^\varpi \rho_{p^n_i\Den})=-\wt u^n_j \Den^{(\varpi-1/2)_+} \rho_{p^n_i\Den} \E[(\Delta\chi_1)^2]+ O(\Den^{(2\varpi-1)_+}\theta_n^2)$. 
	Besides, given an   increasing integer sequence  $\la_n$ such that $\la_n^{-1}=o'(1)$, we have  
	\begin{align*}
		&\E\Biggl[\biggl(\int_0^{((i-1)\Den-\la_n\Den)_+} [g(s-r)-g((i-1)\Den-r)](\si^v_rdW_r+\ov\si^v_rd\ov W_r)\biggr)^2\Biggr]	 \\
		&\quad \leq 2K^2K_H^{-2} \int_0^{((i-1)\Den-\la_n\Den)_+} [(s-r)^{H-1/2}-((i-1)\Den-r)^{H-1/2}]^2 dr\\
		&\quad  \leq 2K^2K_H^{-2} \int_{\la_n\Den}^\infty [(r+\Den)^{H-1/2}-r^{H-1/2}]^2 dr\\
		&\quad=2K^2K_H^{-2} \Den^{2H}\int_{\la_n}^\infty [(u+1)^{H-1/2}-u^{H-1/2}]^2 du=o'(\Den^{2H})
	\end{align*}
	for all $s\in((i-1)\Den,i\Den]$,
	since $u\mapsto[(u+1)^{H-1/2}-u^{H-1/2}]^2$ is integrable. The same type of estimate applies to the integral involving $G$. Together with \eqref{eq:aux30} (which continues to hold under the alternative hypothesis), it follows that 
	$
	\zeta^{n,\rm s}_{j}=\La^n_{2j} +o_p(\pi_n)$, where $	\La^n_j =	\La^{n,1}_j \bone_{\{\La>0\}}+	\La^{n,2}_j \bone_{\{1-\La>0\}}$ and 
	\begin{align*}
		\La^{n,1}_j
		&=\frac{\Den^{(1-2\varpi)_+}f'(v_{(j-\la_n-1)p_n\Den})}{k_n \bigl(c_{(j-\la_n-1)p_n\Den}\bone_{\{\varpi\geq\frac12\}}+\rho^2_{(j-\la_n-1)p_n\Den}\E[(\Delta\chi_1)^2]\bone_{\{\varpi\leq\frac12\}}\bigr)}\\
		&\qquad \times\sum_{i=(j-1)p_n+1}^{(j-1)p_n+k_n}  \int_{(i-\la_np_n-1)\Den}^{(i-1)\Den} [g ((i-1)\Den-s)\\
		&\qquad \qquad\qquad-g( (i-p_n-1)\Den-s  )]  (\si^v_{(j-\la_n-1)p_n\Den}dW_s+\ov\si^v_{(j-\la_n-1)p_n\Den}d\ov W_s),\\
		\La^{n,2}_j
		&=\frac{2\Den^{(2\varpi-1)_+}\E[(\Delta\chi_1)^2]\rho_{(j-\la_n-1)p_n\Den} F'(w_{(j-\la_n-1)p_n\Den})}{k_n \bigl(c_{(j-\la_n-1)p_n\Den}\bone_{\{\varpi\geq\frac12\}}+\rho^2_{(j-\la_n-1)p_n\Den}\E[(\Delta\chi_1)^2]\bone_{\{\varpi\leq\frac12\}}\bigr)}\\
		&\qquad \times\sum_{i=(j-1)p_n+1}^{(j-1)p_n+k_n}  \int_{(i-\la_np_n-1)\Den}^{(i-1)\Den} [G ((i-1)\Den-s)-G( (i-p_n-1)\Den-s  )] \\
		&\qquad \qquad\qquad\times (\si^w_{(j-\la_n-1)p_n\Den}dW_s+\ov\si^w_{(j-\la_n-1)p_n\Den}d\ov W_s+\wh\si^w_{(j-\la_n-1)p_n\Den}d\wh W_s).
	\end{align*}
	This shows that $  V^{n,\rm{alt}}_1	=M^n_1+A^n_1+o_p(1)$,
	where $
	M^n_1=\frac{1}{ {\lvert \calj_n\rvert}\pi_n^2}\sum_{j\in\calj_n}  \{\La^n_{2j} \La^n_{2j-2} -\E[\La^n_{2j} \La^n_{2j-2} \mid \calf_{(2j-\la_n-3)p_n\Den}] \}$ and 
	$
	A^n_1 	=\frac{1}{ {\lvert \calj_n\rvert}\pi_n^2}\sum_{j\in\calj_n}\E[\La^n_{2j} \La^n_{2j-2} \mid \calf_{(2j-\la_n-3)p_n\Den}]$.
	The $j$th term in $M^n_1$ is $\calf_{2jp_n\Den}$-measurable with a zero $\calf_{(2j-3-\la_n)p_n\Den}$-conditional expectation by construction. Therefore, if $j$ and $j'$ are at least  $(3+\la_n)/2$ apart, the two corresponding terms are uncorrelated. Thus, a second moment analysis shows that $M^n_1$ is negligible. On the other hand, $A^n_1$ converges: after shifting the coefficients $c$, $\rho$, $v$, $\si^v$, $\ov \si^v$, $\rho$, $w$, $\si^w$, $\ov\si^w$ and $\wh \si^w$ that appear in $\La^n_{2j}$ from $(2j-\la_n-1)p_n\Den$ to $(2j-\la_n-3)p_n\Den$, 
	the resulting term is conditionally Gaussian given $\calf_{(2j-\la_n-3)p_n\Den}$. Therefore, a tedious but entirely straightforward  calculation shows that $A^n_1\stackrel{\P}{\longrightarrow} \frac1{T}(C_{\kappa,H}\bone_{\{\La\in\{\frac12,1\}\}}+C_{\kappa,H_\rho}\bone_{\{\La=0\}}) \int_{I_T} A(t)dt$, where $A(t)$ is defined after (\ref{eq:VW}) and
	\begin{align*}
		C_{\kappa,H}&=  K_H^{-2}  \int_{0}^{1}\int_0^1\int_0^\infty[(\tfrac{v-w}{\kappa}+r+2)^{H-1/2}-(\tfrac{v-w}{\kappa}+r+1)^{H-1/2}]\\
		&\qquad\qquad\qquad\qquad\qquad\qquad\qquad\qquad\qquad\times[r^{H-1/2}-(r-1)_+^{H-1/2}] drdwdv.
	\end{align*}
	Similarly, $V^{n,\rm{alt}}_2\stackrel{\P}{\longrightarrow} \frac1{T} (C_{\kappa,H}\bone_{\{\La\in\{\frac12,1\}\}}+C_{\kappa,H_\rho}\bone_{\{\La=0\}})\int_{I_T} A(t-1)dt$. And because $\zeta^{n,\rm s}_{j-\ell_1}$ and $\zeta^{n,\rm s}_{j-\ell_2-\dsone/2}$ are asymptotically uncorrelated for  $\ell_1,\ell_2\in\{0,1\}$, we   have $V^{n,\rm{alt}}_3+V^{n,\rm{alt}}_4\stackrel{\P}{\longrightarrow}0$, which completes the proof of the first convergence in (\ref{eq:VW}).
	
	To see that $C_{\kappa,H}<0$, we change variables from $r$ to $1-s$ to get
	\begin{align*}
		C_{\kappa,H} 	&=K_H^{-2}\int_0^1\int_0^1\int_\R [(\tfrac{v-w}{\kappa}+3-s)_+^{H-1/2}-(\tfrac{v-w}\kappa +2-s)^{H-1/2}_+]\\
		&\qquad\qquad\qquad\qquad\qquad\qquad\times[(1-s)^{H-1/2}_+-(-s)_+^{H-1/2}] ds dwdv  \\
		&=\int_0^1\int_0^1 \E[(B^H_{\frac{v-w}{\kappa}+3}-B^H_{\frac{v-w}{\kappa}+2})(B^H_1-B^H_0)] dvdw,
	\end{align*}
	where $B^H$ is a standard fractional Brownian motion with Hurst parameter $H$ (see  \citet[Theorem~1.3.1]{Mishura08} for the last equation).
	Since $\lvert (v-w)/\kappa\rvert \leq 1$, the expectation above is the covariance of two non-overlapping increments of fractional Brownian motion with a Hurst index $H<\frac12$, and such a covariance is known to be negative.  This shows $C_{\kappa,H}<0$. 
	
	Finally, let us determine the limit of $W^{n,\rm{alt}}$ in (\ref{eq:W2}). Similarly to the proof of Lemma~\ref{lem:denom} (cf.\ \eqref{eq:aux31}), we have $W^{n,\rm{alt}}=W^{\prime n,\rm{alt}}+W^{\prime\prime n,\rm{alt}}+o_p(1)$, where
	\begin{align*}
		W^{\prime n,\rm{alt}}&=\frac{1 }{\pi_n^4\lvert\calj_n\rvert }\sum_{j\in\calj_n} [ ( \zeta^{n,\rm s}_{j})^2+ ( \zeta^{n,\rm s}_{j-\dsone/2})^2 ] [ ( \zeta^{n,\rm s}_{j-1})^2+ ( \zeta^{n,\rm s}_{j-1-\dsone/2})^2 ]\\
		W^{\prime\prime n,\rm{alt}}&=\frac{2 }{\pi_n^4\lvert\calj_n\rvert }\sum_{j\in\calj_n} \Bigl\{ 2 \zeta^{n,\rm s}_{j} \zeta^{n,\rm s}_{j-1} \zeta^{n,\rm s}_{j-\dsone/2} \zeta^{n,\rm s}_{j-1-\dsone/2} \\
		&\quad  - [( \zeta^{n,\rm s}_{j})^2+( \zeta^{n,\rm s}_{j-\dsone/2})^2 ] \zeta^{n,\rm s}_{j-1} \zeta^{n,\rm s}_{j-1-\dsone/2}-  [ ( \zeta^{n,\rm s}_{j-1})^2+ ( \zeta^{n,\rm s}_{j-1-\dsone/2})^2 ] \zeta^{n,\rm s}_{j} \zeta^{n,\rm s}_{j-\dsone/2}\Bigr\}.
	\end{align*}
	Note that $(\zeta^{n,\rm s}_{j-1},\zeta^{n,\rm s}_{j})$ and $(\zeta^{n,\rm s}_{j-1-\dsone/2},\zeta^{n,\rm s}_{j-\dsone/2})$ are asymptotically uncorrelated. Together with the fact that $\E[X^2Y^2]=\E[X^2]\E[Y^2]+2\E[X_1X_2]^2$ and $\E[X^2Y]=0$ for a centered bivariate Gaussian random vector $(X,Y)$, we deduce, similarly to the analysis of $A^n_1$ above, that 
	$W^{\prime n,\rm{alt}} \stackrel{\P}{\longrightarrow} \frac{2}{T}\int_{I_T} A(t)A(t-1)dt+ \frac{1}{T}(\ov C_{\kappa,H}\bone_{\{\La\in\{\frac12,1\}\}}+\ov C_{\kappa,H_\rho}\bone_{\{\La=0\}})\int_{I_T} (B(t)+B(t-1))dt$,
	where $B(t)$ is defined after (\ref{eq:VW}) and $\ov C_{\kappa,H}=2(C_{\kappa,H})^2+(C'_{\kappa,H})^2$ with
	\begin{align*}
		C'_{\kappa,H}&=K_H^{-2}\int_0^1\int_0^1\int_0^\infty [(\tfrac{  v-w }{\kappa}+r)^{H-1/2}-(\tfrac{  v-w }{\kappa}+r-1)^{H-1/2}]\\
		&\qquad\qquad\qquad\qquad\qquad\qquad\qquad\times [r^{H-1/2}-(r-1)_+^{H-1/2}] drdwdv 
	\end{align*}
	and
	$W^{\prime\prime n,\rm{alt}} \stackrel{\P}{\longrightarrow} \frac{4}{T}(C_{\kappa,H}\bone_{\{\La\in\{\frac12,1\}\}}+C_{\kappa,H_\rho}\bone_{\{\La=0\}})^2\int_{I_T} A(t)A(t-1) dt$. 
\end{proof}

\end{appendix}

\bibliographystyle{abbrvnat}
\bibliography{rough}

\end{document}